\documentclass{amsart}
\usepackage{amsmath,amsfonts,amssymb,amsthm,epsfig,amscd,epsfig,psfrag,comment}
\usepackage[all]{xy}

\newtheorem{Theorem}{Theorem}[section]
\newtheorem{Proposition}[Theorem]{Proposition} 
\newtheorem{Lemma}[Theorem]{Lemma}

\newtheorem{Corollary}[Theorem]{Corollary}
\newtheorem{Conjecture}[Theorem]{Conjecture}
\newtheorem{Specialthm}{Theorem}
\newtheorem{Specialconj}[Specialthm]{Conjecture}
\newtheorem{Question}{Question}

\newcommand{\puteps}[2][0.5]
{\includegraphics[scale=#1]{#2.eps}}

\theoremstyle{definition}

\newtheorem{Remark}[Theorem]{Remark}

\newcommand{\p}{{\mathbb{P}^{m-1}}}
\newcommand{\Z}{\mathbb{Z}}
\newcommand{\C}{\mathbb{C}}
\def\sX{{\mathcal{X}}}

\def\F{{\sf{F}}}
\def\sF{{\mathcal{F}}}
\def\G{{\sf{G}}}
\def\sG{{\mathcal{G}}}
\def\A{{\mathcal{A}}}

\def\T{{\sf{T}}}
\def\sT{{\mathcal{T}}}
\def\sV{{\mathcal{V}}}
\def\sW{{\mathcal{W}}}

\def\sL{{\mathcal{L}}}
\def\sP{{\mathcal{P}}}
\def\sR{{\mathcal{R}}}
\def\sU{{\mathcal{U}}}
\def\P{{\sf{P}}}
\def\Q{{\sf{Q}}}
\def\U{{\sf{U}}}
\def\sQ{{\mathcal{Q}}}
\newcommand{\Kh}{\mathrm{H}_{\mathrm{alg}}}

\newcommand{\Cone}{\mathrm{Cone}}
\def\O{{\mathcal O}}
\def\E{{\mathcal{E}}}
\def\H{{\mathcal{H}}}

\def\dim{\mbox{dim}}
\newcommand{\bbeta}{\overline{\beta}}
\newcommand{\bb}{\overline{b}}
\newcommand{\bE}{\overline{\mathcal{E}}}

\newcommand{\uq}{U_q(\mathfrak{sl}_m)}
\newcommand{\uqg}{U_q(\mathfrak{g})}
\renewcommand{\sl}{{\mathfrak{sl}_m}}
\newcommand{\cl}{\text{cl}}
\newcommand{\base}{\mathbb{C}[q,q^{-1}]}

\newcommand{\vect}{\C^m \otimes z^{-1} \C[z^{-1}]}
\newcommand{\Gr}{Gr}
\newcommand{\edge}{\longrightarrow}

\newcommand{\ul}{{\underline{\lambda}}}
\newcommand{\um}{{\underline{\mu}}}

\DeclareMathOperator{\Hom}{Hom}
\DeclareMathOperator{\Ext}{Ext}

\DeclareMathOperator{\Sym}{Sym}

\begin{document}

\title[Knot homology via derived categories of coherent sheaves II, $sl(m)$ case]{Knot homology via derived categories of coherent sheaves II, $sl(m)$ case}

\author{Sabin Cautis}
\email{scautis@math.harvard.edu}
\address{Department of Mathematics\\ Rice University \\ Houston, TX}

\author{Joel Kamnitzer}
\email{jkamnitz@math.berkeley.edu}
\address{Department of Mathematics\\ UC Berkeley \\ Berkeley, CA}

\begin{abstract}
Using derived categories of equivariant coherent sheaves we construct a knot homology theory which categorifies the quantum sl(m) knot polynomial. Our knot homology naturally satisfies the categorified MOY relations and is conjecturally isomorphic to Khovanov-Rozansky homology. Our construction is motivated by the geometric Satake correspondence and is related to Manolescu's by homological mirror symmetry. 
\end{abstract}

\date{\today}
\maketitle
\tableofcontents
\newpage

\section{Introduction}
The main purpose of this paper is to use algebraic geometry to categorify the quantum knot and tangles associated to $\uq $ and its standard representation, generalizing the results of \cite{kh1}.  Before giving more details of what is done in this paper, we would like to explain the motivation for this work in terms of Khovanov's categorification program and the geometric Satake correspondence.

\subsection{Categorification}
Let $ \mathfrak{g} $ denote a complex semisimple Lie algebra.  There is a diagrammatic calculus involving tangles and representations of the quantum group $ \uqg $.  Consider a tangle $ T $ where all the strands are labelled by dominant weights of $ \mathfrak{g} $.  Let $ \ul = (\lambda_1, \dots, \lambda_n) $ be the labels on the strands at the top and $ \um = (\mu_1, \dots, \mu_{n'}) $ be the labels on the strands at the bottom.  Let $ V_\lambda $ denote the irreducible $ \uqg $ representation with highest weight $ \lambda $.

Following Reshetikhin-Turaev \cite{RT}, one can associate to this tangle a map of $ \uqg $ representations 
\begin{equation*}
\psi(T) :  V_\ul := V_{\lambda_1} \otimes \dots \otimes V_{\lambda_n} \rightarrow V_\um := V_{\mu_1} \otimes \dots \otimes V_{\mu_{n'}}.
\end{equation*}
This is done by analyzing a tangle projection from top to bottom and considering each cap, cup, and crossing in turn and assigning natural maps to these elementary tangles.  For example to each right handed crossing we associate the braiding $ V_\lambda \otimes V_\mu \rightarrow V_\mu \otimes V_\lambda $, which is defined using the $R$-matrix of $ \uqg $.  Reshetikhin-Turaev \cite{RT} showed that $ \psi(T) $ does not depend on the planar projection of the tangle.  Thus we have a map
\begin{equation*}
\psi : \biggl\{ (\ul, \um) \text{ tangles } \biggr\} \rightarrow \Hom_{\uqg}(V_\ul, V_\um)
\end{equation*}
If $ T = K $ is link, then $ \psi(K) $ is a map $ \base \rightarrow \base $ and $ \psi(K)(1) $ is a polynomial.  When $ \mathfrak{g} = \mathfrak{sl}_m $ and all $ \lambda_i $ are the fundamental weight, this polynomial is a one variable specialization of the two variable HOMFLY polynomial.

Khovanov has proposed the idea of categorifying this calculus.  A \textbf{weak categorification} is a choice of graded triangulated category $ D_\ul $ for each collection $ \ul = (\lambda_1, \dots, \lambda_n) $ and a map
\begin{equation*}
\Psi : \biggl\{ (\ul, \um) \text{ tangles } \biggr\} \rightarrow \biggl\{\text{ isomorphism classes of exact functors } D_\ul \rightarrow D_\um \biggr\}
\end{equation*}
such that we recover the original calculus on the Grothendieck group, i.e. we have $ K(D_\ul) \cong V_\ul $ as $ \base $ modules and $ [\Psi(T)] = \psi(T) $.  We will also insist that $ D_\emptyset $ be the derived category of graded vector spaces.

One source of interest in categorification is bigraded knot homology theories.  Suppose that $ K $ is a link whose components are labelled by dominant weights.  Then $ \Psi(K) $ is a functor from the derived category of graded vector spaces to itself.  The cohomology of the complex of graded vector spaces $ \Psi(K)(\mathbb{C}) $ is a link invariant, whose graded Euler characteristic is the Reshetikhin-Turaev invariant $ \psi(K)(1) $.

\begin{Question} \label{qu}
Does a categorification exist for all $ \mathfrak{g} $ and all $ \lambda_i$?
Is there a natural way to construct such categorifications?
\end{Question}

The pioneering work in this area was done by Khovanov \cite{Kknot}, \cite{Ktangle} in the case where $ \mathfrak{g} = \mathfrak{sl}_2 $ and all $ \lambda_i $ the fundamental weight.  The knot homology produced is the celebrated Khovanov homology.

The case $ \mathfrak{g} = \mathfrak{sl}_m $ and all $ \lambda_i $ minuscule, has been studied by Khovanov-Rozansky and Sussan.  Khovanov-Rozansky \cite{KR} do not explicitly define categories $ D_\ul $ categorifying $V_\ul $. But using matrix factorizations they define functors $ \Psi(T) $ for tangles and produce a knot homology theory $ H^{i,j}_{\text{KR}} $, when all the strands are labelled by $\omega_1, \omega_{m-1} $.

On the other hand, Sussan \cite{Su} has constructed categories $ D_\ul $, using certain blocks of parabolic categories $ \O $ for various Lie algebras $ \mathfrak{sl}_N $.  He has also constructed functors $ \Psi(T) $ for tangles $ T $ where all strands are labelled by $ \omega_1, \omega_{m-1}$.  It is not known (but widely believed) that the knot homology constructed by Sussan agrees with that of Khovanov-Rozansky.

\subsection{Categorification using geometric Satake correspondence}
We return now to general $ \mathfrak{g} $, but we continue to assume that all $ \lambda_i $ are minuscule.  In this setting we will now present a proposal for answering Question \ref{qu} which is based on the geometric Satake correspondence. 

Consider the semisimple group $ G^\vee $ which is the Langlands dual to $ G $, the simply-connected group integrating $ \mathfrak{g} $ (note that $G^\vee $ has trivial centre).  Let $ \Gr := G^\vee(\C((z))) / G^\vee(\C[[z]]) $ denote the \textbf{affine Grassmannian} for $ G^\vee$, an ind-scheme over $ \C $.  

The affine Grassmannian is stratified by the $ G^\vee(\C[[z]]) $ orbits $ \Gr_\lambda $ which are labelled by $ \lambda \in \Lambda_+ $, the set of dominant weights of $\mathfrak{g}$.  Similarly, the orbits of $ G^\vee(\C((z))) $ acting on $ \Gr \times \Gr $ are also labelled by $ \Lambda_+ $ and we write $ L_1 \overset{\lambda}{\edge} L_2 $ if $ (L_1,L_2) $ is in the orbit labelled by $ \lambda $.  Let $ L_0 $ denote the identity coset in $ \Gr $.

Given dominant weights $(\lambda_1, \dots, \lambda_n) $, we can form the \textbf{convolution product} 
\begin{equation*}
\Gr_\ul := \Gr_{\lambda_1} \tilde{\times} \dots \tilde{\times} \Gr_{\lambda_n} = \{ (L_1, \dots, L_n) \in \Gr^n : L_0 \overset{\lambda_1}{\edge} L_1 \overset{\lambda_2}{\edge} \cdots \overset{\lambda_{n-1}}{\edge} L_{n-1} \overset{\lambda_n}{\edge} L_n  \}
\end{equation*}
We have an obvious map $ m: \Gr_\ul \rightarrow \Gr $ taking $ (L_1, \dots, L_n) $ to $ L_n $.  

We are now in a position to give the statement of the geometric Satake correspondence of Lusztig \cite{L}, Ginzburg \cite{G}, and Mir\-kovi\'c-Vilonen \cite{MV}.

\begin{Theorem}
There is an equivalence between the category of perverse sheaves on $ \Gr $ (constructible with respect to the above stratification) and the category of representations of $\mathfrak{g}$.   This equivalence commutes with the functors to vector spaces on each side (hypercohomology and the forgetful functor, respectively).  

Under this equivalence, the IC sheaf of $ \overline{\Gr}_\lambda $ corresponds to the irreducible representation $ V_\lambda $. More generally, the push forward under $ m $ of the IC sheaf of $ \overline{\Gr}_\ul $ corresponds to the tensor product $ V_{\ul} = V_{\lambda_1} \otimes \cdots \otimes V_{\lambda_n}$, so that 
$ IH(\overline{\Gr}_\ul) \cong V_\ul$.
\end{Theorem}

$ \overline{\Gr}_\ul $ is smooth if and only if $ \lambda_i $ is minuscule for each $ i $.  In this case we also have $ \overline{\Gr}_\ul = \Gr_\ul $.  Thus, when all $ \lambda_i $ are minuscule, we have a smooth projective variety, $\Gr_\ul$ whose cohomology is canonically isomorphic to $V_\ul$.  This variety can also be seen to carry a natural $ \C^\times $ action and hence we propose the choice $D_\ul = D(\text{Coh}^{\mathbb{C}^\times}(\Gr_\ul))$, the derived category of $ \mathbb{C}^\times $-equivariant coherent sheaves on $ \Gr_\ul $ for categorification (the $\C^\times $ action is used to make the category graded). 

To complete the above categorification program, it is neccesary to construct functors corresponding to tangles. There are natural correspondences between these $\Gr_\ul$ varieties which can be used for this purpose (see Remarks \ref{re:sat1}, \ref{re:sat2}).

\subsection{The current paper: the $\mathfrak{sl}_m$ case}
In general, derived categories of coherent sheaves are difficult to work with, so we are not yet able to complete this categorification program for general $ \mathfrak{g} $ and minuscule $ \lambda_i$.  In this paper, we will concentrate on the case $ \mathfrak{g} = \mathfrak{sl}_m $ and all $ \lambda_i $ corresponding to the standard or dual representatations (the same as studied by Khovanov-Rozansky, Sussan). In this case, each $ \Gr_{\lambda_i} $ is a $ \p $ and so the variety $ \Gr_\ul $ is an iterated fibre bundle of $ \p $s.  This makes the geometry much easier to work with. 

We will encode fundamental weights of $ \mathfrak{sl}_m $ by integers $ \beta_i \in \{1, \dots, m-1\}$ and $ \beta $ will denote a sequence of such integers.  We introduce the notation $ Y_\beta $ for $\Gr_\ul $ where $ \lambda_i = \omega_{\beta_i} $.  Section \ref{se:geometry} discusses the geometry of these varieties.  In particular, in section \ref{se:SS}, we explain how open subvarieties of certain $ Y_\beta $ are related by hyperK\"ahler rotation to the symplectic manifolds studied by Manolescu \cite{M}.

Consider the case where each $ \beta_i \in \{1, m-1\} $.  Given a planar projection of a $(\beta,\beta') $ tangle $ T $, we construct a functor $ \Psi(T) : D(Y_\beta) \rightarrow D(Y_{\beta'}) $ by associating to each cup, cap, and crossing certain basic functors 
$$ \G_\beta^i : D(Y_{d_i(\beta)}) \rightarrow D(Y_\beta),\quad \F_\beta^i : D(Y_\beta) \rightarrow D(Y_{d_i(\beta)}),\quad  \T_\beta^i(l) : D(Y_\beta) \rightarrow D(Y_{s_i(\beta)}). $$ 
The functors $ \G_\beta^i, \F_\beta^i $ are defined using a correspondence $ X_\beta^i $ between $ Y_{d_i(\beta)} $ and $ Y_\beta $ (see section \ref{se:capcup}).  The functors $ \T_\beta^i(l) $ are defined using a correspondence $ Z_\beta^i $ between $ Y_\beta $ and $Y_{s_i(\beta)}$.  When $ \beta_i = \beta_{i+1} $, $\T_\beta^i(l)$ can also be described as a spherical twist in a functor $ \P_\beta^i $ (see section \ref{se:sptwists}) while in the case where $ \beta_i \ne \beta_{i+1}$ the correspondence $ Z_\beta^i $ is a family of Mukai flops.

In section \ref{se:invariance} (which is the bulk of this paper), we check certain relations among these basic functors corresponding to Reidemeister and isotopy moves between different projections of the same tangle.  We obtain the following result.
\begin{Specialthm}[Theorem \ref{thm:main}]
The isomorphism class of $\Psi(T) $ is an invariant of the tangle $ T $.
\end{Specialthm}
Among the relations that we check, the most basic is that $ \T_\beta^i $ is an equivalence.  In the case $ \beta_i = \beta_{i+1}$ this follows from the theory of spherical twists (see section \ref{se:twists}).  In the case $ \beta_i \ne \beta_{i+1}$, we give a direct proof (in section \ref{se:RIIb}) following ideas of Bridgeland-King-Reid \cite{BKR}, though the result is essentially already known by the work of Namikawa \cite{Na}.  

The other important relation is that $ \T_\beta^i $ and $ \T_\beta^{i+1} $ braid.  Here our proof proceed by direct computation following a similar argument from Khovanov-Thomas \cite{KT}.

In section \ref{se:graphs}, we extend our invariant $\Psi$ to trivalent tangle graphs and prove the following result by checking categorified versions of the Murakami-Ohtsuki-Yamada \cite{MOY} relations.
\begin{Specialthm}[Theorem \ref{th:diag}]
Let $ D $ be a crossingless closed trivalent graph.  Then $ H^{i,j}(\Psi(D)) $ is concentrated on the anti-diagonal and its graded Poincar\'e polynomial is determined by $ \psi(D) $ (the RT invariant of such a graph).
\end{Specialthm}

When $T=K$ is a link let $ H^{i,j}_{\text{alg}}(K) := H^{i,j}(\Psi(K))$ denote our knot homology theory.  By considering the distinguished triangles associated to the crossings, we obtain that the graded Euler characteristic of $ H^{i,j}_{\text{alg}}(K) $ equals the Reshetikhin-Turaev invariant $ \psi(K)(1) $ for any link $K$.  We also conjecture (Conjecture \ref{th:KRconj}) that our knot homology theory coincides with that of Khovanov-Rozansky (up to a ``shear'' in the indices). 

In section \ref{se:Kgroups}, we consider $ \Psi(T) $ on the level of Grothendieck group. We show that the Grothendieck group $K(Y_\beta) = K(D(Y_\beta))$ is isomorphic to $V_\beta = V_{\omega_{\beta_1}} \otimes \dots \otimes V_{\omega_{\beta_n}}$ and conjecture the following.
\begin{Specialconj}[Conjecture \ref{th:Kconj}]
For each $ \beta $, there exists an isomorphism $ K(Y_\beta) \rightarrow V_\beta$ such that for each tangle $ T $, the following diagram commutes
\begin{equation}
\begin{CD}
K(Y_\beta) @>>> V_\beta \\
@V[\Psi(T)]VV @VV\psi(T)V \\
K(Y_{\beta'}) @>>> V_{\beta'}
\end{CD}
\end{equation}
\end{Specialconj}
We give three pieces of evidence for this conjecture.  We prove that it holds when $ q=1$ (Theorem \ref{th:Kconjq1}), that it holds if $ T $ is a link, and that the eigenvalues for the crossings are correct (see Remark \ref{re:Kconj}).

\subsection{Acknowledgements}
We would like to thank Roman Bezrukavnikov, Sasha Braverman, Andrei Caldararu, Timothy Logvinenko, Mikhail Khovanov, Scott Morrison, Yoshinori Namikawa, Lev Rozansky, Noah Snyder, and Ben Webster for useful conversations, as well as a referee for some helpful suggestions. The first author would like to thank the Mittag-Leffler Institute for their hospitality during his visit in the spring of 2007 and the mathematics department at Rice University. The second author thanks the American Institute of Mathematics for support and the mathematics department of UC Berkeley for hospitality. 

\section{Geometric background} \label{se:geometry}
We begin by describing the geometric setup upon which our work is based.

\subsection{The Varieties $ Y_\beta $ for $\mathfrak{sl}_m$} \label{se:varieties} 

When we specialize to the case $ \mathfrak{g} = \mathfrak{sl}_m $, we find that $ G^\vee = PGL_m$.  The minuscule weights of $ \mathfrak{sl}_m $ are the same as its fundamental weight $ \omega_1, \dots, \omega_{m-1} $.  The corresponding variety $ \Gr_{\omega_k} $ is isomorphic to the ordinary Grassmannian $ \Gr(k,m) $ of $k$-dimensional subspaces of $ \mathbb{C}^m $.

We will consider $ \beta = (\beta_1, \dots, \beta_n) $ where each $ \beta_i $ is one of the fundamental weights.  Since the fundamental weights are labelled by $ 1, \dots, m-1 $, we may consider each $ \beta_i \in \{1, \dots, m-1\}$ and thus $\beta $ is a sequence of length $ n $ with entries drawn from $ 1, \dots, m-1 $.  The corresponding string of weights is $ \ul $ with $ \lambda_i = \omega_{\beta_i} $. 

The variety $ Y_\beta = \Gr_\ul $ may be described in more elementary terms.
Fix a vector space $ \C^m $ and consider the vector space $ \C^m \otimes (\C((z))/\C[[z]]) = \C^m \otimes z^{-1}\C[z^{-1}]  $.  This is an infinite dimensional vector space, but we will only consider finite dimensional subspaces.  Note that we have a locally nilpotent operator $ z : \vect \rightarrow \vect $. For practical purposes, we can take the following as the definition of $Y_\beta$:
\begin{equation*}
\begin{aligned}
Y_\beta := \{ (L_1, \dots, L_n) : L_i \subset \vect, \ &0 = L_0 \subset L_1 \subset L_2 \subset \dots \subset L_{n}, \\ &\dim(L_i / L_{i-1}) = \beta_i, \text{ and } z L_i \subset L_{i-1} \}.
\end{aligned}
\end{equation*}
where $ \beta = (\beta_1, \dots, \beta_n) $ with $ \beta_i \in \{1, \dots, m-1\}$. 

\begin{Proposition}
When $ \beta $ and $ \lambda $ correspond as above, we have $ Y_\beta \cong \Gr_\ul $.
\end{Proposition}
\begin{proof}
The affine Grassmannian for $PGL_m $ is the ind-scheme of $ \C[[z]] $ lattices $ L $ in $ \C^m \otimes \C((z))$ modulo the equivalence relation $ L \sim z L $.  With this decription, two points in $\Gr $ are in relative position $ \omega_k $ if and only if there exist lattice representatives $ L_1, L_2 $ for them such that $ L_1 \subset L_2$, $z L_2 \subset L_1$ and $ \dim(L_2/L_1) = k $.  Applying this repeatedly gives the isomorphism.
\end{proof} 

Note that $ Y_\beta $ is a smooth projective variety of dimension $ |\beta| := \sum_i \beta_i(m - \beta_i) $.  This follows from the general theory outlined above, but it can also be seen explicitly as follows. There is a natural  map from $ Y_{(\beta_1, \dots, \beta_n)}$ to $ Y_{(\beta_1, \dots, \beta_{n-1})} $ with fibres $ \Gr(\beta_n, m) $.  To see this, suppose that we have $ (L_1, \dots, L_{n-1}) \in Y_{(\beta_1, \dots, \beta_{n-1})} $ and are considering possible choices of $ L_n $.  It is easy to see that we must have $ L_{n-1} \subset L_n \subset z^{-1}(L_{n-1}) $.  Since $ z^{-1}(L_{n-1}) / L_{n-1} $ is always $m $ dimensional, this fibre is a $ \Gr(\beta_n, m) $ and the map is a $\Gr(\beta_n, m) $ bundle.  

There is an action of $\C^\times$ on $ z^{-1}\C[z^{-1}] $ given by $ t \cdot z^k = t^{2k} z^k$.  This induces an action of $ \C^\times $ on  $\vect$. Notice that for any $v \in \vect $ we have 
$$t \cdot (zv) = t^{2}z(t \cdot v).$$
Thus, as vector spaces, $t \cdot (zL_i) = z(t \cdot L_i)$ so if $zL_i \subset L_{i-1}$ then $t \cdot zL_i \subset t \cdot L_{i-1}$ which means $z(t \cdot L_i) \subset t \cdot L_{i-1}$. Consequently, the $\C^\times$ action on $\vect$ induces a $\C^\times$ action on $Y_\beta$ by
$$t \cdot (L_1, \dots, L_{n}) = (t \cdot L_1, \dots, t \cdot L_{n}).$$

Viewed from the perspective of the affine Grassmannian, this action is the ``rotating the loops'' action on the variable $ z $.  In particular, it does not come from any action of $ \C^\times $ on $ G^\vee $. In this paper, all maps between varieties will be $\C^\times$-equivariant with respect to this (or some other) $\C^\times$ action. 


\subsection{Notational Conventions}
There are two natural ways to modify a sequence $ \beta $ at an index $ i $.  We introduce the following notation.
\begin{gather*}
s_i(\beta) := (\beta_1, \dots, \beta_{i-1}, \beta_{i+1}, \beta_i, \beta_{i+2}, \dots, \beta_n) \\
d_i(\beta) := (\beta_1, \dots, \beta_{i-1}, \beta_i + \beta_{i+1}, \beta_{i+2}, \dots, \beta_n) \\
\end{gather*}
These letters stand for ``switch'' and ``drop'' respectively.  Soon we will see how they arise geometrically. Denote $b_\beta^i := \sum_{j=1}^i \beta_j$ while $\bbeta_i := m - \beta_i$ and $\bb_\beta^i := \sum_{j=1}^i \bbeta_j$. 

We will also need the following notational convenience.  Up to now, we have allowed jumps between the $ L_i $ of dimension $1, \dots, m-1 $.  So we had $ \beta_i \in \{1, \dots, m-1 \} $.  Now, it will be more convenient to have $ \beta_i \in \mathbb{Z}/m $ where we identify an element of $ \mathbb{Z}/m $ with its unique representative in the range $ \{0, \dots, m-1 \} $.  Note that this allows the possibility of $ \beta_i = 0 $.  This poses no problem --- we just note that if $ \beta_i = 0 $, there is a natural isomorphism between $ Y_\beta $ and $ Y_{d_i(\beta)} $. Whenever an operation yields some $ \beta_i = 0 $ we automatically drop it so that we always have $ \beta_i \in \{1, \dots, m-1\}$. 

Another useful notation is the following. Often in this paper we will deal with a product of these spaces, such as $ Y_a \times Y_b \times Y_c $.  In such cases, we use the notation $ \E $ for $ \pi_1^*(\E)$, $\E' $ for $ \pi_2^*(\E) $ and $\E'' $ for $\pi_3^*(\E)$.  This way, the fibre of the bundle $ L'_i$ at a point $ (L_\cdot, L'_\cdot, L''_\cdot) \in Y_a \times Y_b \times Y_c $ is the vector space $ L'_i $.

\subsection{The subvarieties $X_\beta^i$}
For each $ 1 \le i \le n-1 $, define the following $\C^\times$-equivariant subvariety of $ Y_\beta $, 
\begin{equation*}
X_\beta^i := \begin{cases} 
\{ (L_1, \dots, L_{n}) \in Y_\beta : zL_{i+1} \subset L_{i-1} \} \text{ if } \beta_i + \beta_{i+1} \le m \\
\{ (L_1, \dots, L_n) \in Y_\beta : L_{i-1} \subset zL_{i+1} \text{ and } \mbox{ker}(z) \subset L_{i+1} \} \text{ if } \beta_i + \beta_{i+1} \ge m 
\end{cases}
\end{equation*}

We have a $\C^\times$-equivariant map $q: X_\beta^i \rightarrow Y_{d_i(\beta)} $ defined by 
\begin{equation*}
(L_1, \dots, L_n) \mapsto \begin{cases} (L_1, \dots, L_{i-1}, L_{i+1}, \dots, L_n) \text{ if } \beta_{i-1} + \beta_i < m \\
(L_1, \dots, L_{i-1}, zL_{i+1}, \dots, zL_n) \text{ if } \beta_{i-1} + \beta_i \ge m
\end{cases}
\end{equation*}
It is elementary to check that if $\beta_i + \beta_{i+1} \le m$ then $ X_\beta^i $ is a subvariety of codimension $ \beta_i \beta_{i+1} $ and the map $q: X_\beta^i \rightarrow Y_{d_i(\beta)} $ is a $ \Gr(\beta_i, \beta_i + \beta_{i+1}) $ bundle while if $\beta_i + \beta_{i+1} \ge m$ then $X_\beta^i$ is a subvariety of codimension $\bbeta_i \bbeta_{i+1}$ and the map $q: X_\beta^i \rightarrow Y_{d_i(\beta)}$ is a $ \Gr(\bbeta_i, \bbeta_i+\bbeta_{i+1})$ bundle. 

There are two special cases of particular relevance.  First, if we have $ \beta_i + \beta_{i+1} = m $, then $ d_i(\beta) = (\beta_1, \dots, \beta_{i-1}, \beta_{i+2}, \dots, \beta_m) $ because of our convention about dropping "0".  So we have a diagram 
\begin{equation*}
\begin{CD}
X_\beta^i @>i>> Y_\beta \\
@VqVV \\
Y_{(\beta_1, \dots, \beta_{i-1}, \beta_{i+2}, \dots, \beta_m)} 
\end{CD}
\end{equation*}
where $ i $ is the inclusion of a codimension $ \beta_i \beta_{i+1} $ subvariety and $ q $ is a $ \Gr(\beta_i, m) $ bundle. 

The other special case is when $ \beta_i = \beta_{i+1} = 1 $ or when $ \beta_i = \beta_{i+1} = m-1 $.  In this case, $ X_\beta^i $ is a divisor in $ Y_\beta $ and the map $ X_\beta^i \rightarrow Y_{d_i(\beta)} $ is a $ \mathbb{P}^1 $ bundle.

\begin{Remark} \label{re:sat1}
Let us examine this construction from the point of view of the affine Grassmannian.  For simplicity, assume that $ n =2 $.  We will use $ \lambda, \mu $ for $ \beta_1, \beta_2 $ and we will give an analysis which holds for any $ G^\vee $ under the assumption that $ \lambda $ and $ \mu $ are minuscule.

Recall the map $ m_{\lambda \mu} : \Gr_\lambda \tilde{\times} \Gr_\mu \rightarrow \Gr $ which is given by $ (L_1, L_2) \mapsto L_2 $.  The image of this map is contained in $ \overline{\Gr_{\lambda + \mu}} $.  For any $ \nu $ such that $ \lambda + \mu - \nu $ is a sum of positive roots, $\Gr_\nu \subset \overline{\Gr_{\lambda + \mu}}$ and we may consider the preimage $ X_{\lambda \mu}^\nu := m_{\lambda \mu}^{-1}(\Gr_\nu)$.  For general (i.e. not necessarily minuscule) $ \lambda, \mu $, the geometric Satake correspondence tells us that the number of components of $X_{\lambda \mu}^{\nu}$ of dimension $\langle \lambda + \mu + \nu, \rho \rangle $ equals the multiplicity of $V_\nu$ in $V_\lambda \otimes V_\mu$. In our minuscule setting, the tensor product multiplicities are always one and hence there is one component of the expected dimension.  The following recent result of Haines \cite{Haines2} shows that this is the only component.
\begin{Proposition} \label{th:Haines}
If $\lambda$ and $\mu$ are minuscule then $X_{\lambda \mu}^\nu $ is an irreducible variety of dimension $ \langle \lambda + \mu + \nu, \rho \rangle $.  Moreover the map $ X_{\lambda \mu}^\nu \rightarrow \Gr_\nu $ is a fibration.
\end{Proposition}

In particular, suppose that $ \beta = (k,l) $.  Under the identification $ Y_\beta \cong \Gr_{\omega_k} \tilde{\times} \Gr_{\omega_l} $, then the subvariety $ X_\beta^1 $ is identified with $ X_{\omega_k \omega_l}^\nu $ where $ \nu = \omega_{k+l} $ if $ k + l \ne m $ and $ \nu = 0 $ if $ k + l = m $.

\end{Remark}

\subsection{Vector bundles}

For $1 \le k \le n$ we have a $\C^\times$-equivariant vector bundle on $Y_\beta$ whose fibre over the point $(L_1, \dots, L_n) \in Y_\beta$ is $L_k/L_0$. Abusing notation a little, we will denote this vector bundle by $L_k$. Since $L_{k-1} \subset L_k$ we can consider quotient $L_k/L_{k-1}$ as well as the determinant of the quotient $\E_k = \det(L_k/L_{k-1})$ which is a $\C^\times$-equivariant line bundle on $Y_\beta$. It will be convenient to also consider $z^{-1}L_{k-1}/L_k$ and $L_{k-1}/zL_k$ and the associated determinant $\bE_k := \det(z^{-1}L_{k-1}/L_k) \{-\bbeta_k\} \cong \det(L_{k-1}/zL_k) \{\bbeta_k\}$ (the isomorphism follows from lemma \ref{lem:facts0} below).  

Similarly, we get equivariant vector bundles $L_k$ on $X_\beta^i$ as well as the corresponding line bundles $\E_k$ and $\bE_k$ for any $1 \le k \le n$. Since $L_k$ on $Y_\beta$ restricted to $X_\beta^i$ is isomorphic to $L_k$ on $X_\beta^i$ we can omit subscripts and superscripts telling us where each $L_k, \E_k$ and $\bE_k$ lives. 
The map $z:L_i \rightarrow L_{i-1}$ has weight $2$ since 
$$(t \cdot z)(v) = t \cdot (z(t^{-1} \cdot v)) = t^{2}z(t \cdot t^{-1} \cdot v) = t^{2}zv.$$
If $\sF$ is a $\C^\times$-equivariant sheaf on $Y_n$ then we denote by $\sF\{k\}$ the same sheaf but shifted with respect to the $\C^\times$ action so that if $f \in \sF(U)$ is a local section of $\sF$ then viewed as a section $f' \in \sF\{k\}(U)$ we have $t \cdot f' = t^{-k}(t \cdot f)$. Using this notation we obtain the $\C^\times$-equivariant map $z: L_{i+1} \rightarrow L_i\{2\}$.

\begin{Lemma}\label{lem:facts0} On $Y_\beta$ we have
\begin{enumerate}
\item $\det(z^{-1}L_i/L_i) \cong \O_{Y_\beta} \{2b_\beta^i + 2m\}$ 
\item $\det(z^{-1}L_{i-1}/L_i) \cong \E_{i}^\vee \{2b_\beta^{i-1} + 2m\}$ and $\det(L_{i-1}/zL_{i}) \cong \E_{i}^\vee \{2b_\beta^{i}\}$ 
\end{enumerate}
where $b_\beta^i = \sum_{j=1}^i \beta_j$.
\end{Lemma}
\begin{proof}
(i). From the exact sequence $0 \rightarrow z^{-1}L_0/L_0 \rightarrow z^{-1}L_i/L_0 \xrightarrow{z} L_i/L_0 \{2\} \rightarrow 0$ we have
\begin{align*}
\det(z^{-1}L_i/L_i) &\cong \det(z^{-1}L_i/L_0) \otimes \det(L_i/L_0)^\vee \\
&\cong \det(z^{-1}L_0/L_0) \otimes \O_{Y_\beta} \{2 b_\beta^i\}
\cong \O_{Y_\beta} \{2 b_\beta^i + 2m\}
\end{align*}
where, to get the last equality, we use that $z^{-1}L_0/L_0 \cong \O_{Y_\beta}^{\oplus m} \{2\}$. 

(ii). Similarly, using the exact sequence $0 \rightarrow L_{i+1}/L_i \rightarrow z^{-1}L_i/L_i \rightarrow z^{-1}L_i/L_{i+1} \rightarrow 0$ we have 
\begin{equation*}
\det(z^{-1}L_i/L_{i+1}) \cong \det(z^{-1}L_i/L_i) \otimes \det(L_{i+1}/L_i)^\vee 
\cong \E_i^\vee \{2b_\beta^i+2m\}.
\end{equation*}
Also, from $z^{-1}L_i/L_{i+1} \xrightarrow{z} L_i/zL_{i+1} \{2\}$ we get
\begin{equation*}
\det(L_i/zL_{i+1}) \cong \E_{i+1}^\vee \{2b_\beta^i+2m\} \otimes \O_{Y_\beta} \{-2(m-\beta_{i+1})\} 
\cong \E_{i+1}^\vee \{2b_\beta^{i+1}\}.
\end{equation*}
\end{proof}

\subsection{The Subvarieties $Z_\beta^i$}
Consider the $\C^\times$-equivariant subvariety
\begin{equation*}
 Z_\beta^i := \{ (L_\cdot, L'_\cdot) : L_j = L_j' \text{ for $ j \ne i $ } \} \subset Y_\beta \times Y_{s_i(\beta)}.
\end{equation*}

Now, we will analyze the four cases which lead to $ Z_\beta^i $ having two components.  These cases are 
\begin{equation*}
\beta_i = \beta_{i+1} = 1, \quad \beta_i = \beta_{i+1} = m-1, \quad \beta_i = 1, \beta_{i+1} = m-1, \quad \beta_i = m-1, \beta_{i+1} = 1
\end{equation*}
In all these $Z_\beta^i$ has two smooth irreducible components each of the same dimension as $ Y_\beta $.  The cases where $ Z_\beta^i $ has two components are much easier to handle, so we will limit ourselves to those in this paper.

Suppose that we are in the first case and assume that $ \beta_i = \beta_{i+1} = 1 $.  Then the first component corresponds to the locus of points where $L_i=L'_i$, and so is the diagonal $ \Delta \subset Y_\beta \times Y_\beta $.   The second component is the closure of the locus of points where $L_i \ne L'_i$.  Note that if $ L_i \ne L_i' $, then $ zL_{i+1} \subset L_i \cap L_i' =  L_{i-1}$. Thus we see that on this closure, $ zL_{i+1} \subset L_{i-1} $.  Hence this second component is the subvariety $V_\beta^i := X_\beta^i \times_{Y_{d_i(\beta)}} X_\beta^i$ where the fibre product is with respect to the map $q: X_\beta^i \rightarrow Y_{d_i(\beta)}$. Thus $Z_\beta^i = \Delta \cup V_\beta^i \subset Y_\beta \times Y_\beta$.  

If $ \beta_i = \beta_{i+1} = m - 1$, then a similar analysis holds with the same conclusion.  On the non-diagonal component, we have $ L_{i-1} \subset L_i \cap L_i' = zL_{i+1} $.  

On the other hand, suppose that $ \beta_i = 1 $ and $ \beta_{i+1} = m - 1 $.  Then the first component $W_\beta^i$ is the locus of points where $ L_i \subset L'_i $. Notice that if we let $ \beta' = (\beta_1, \dots, \beta_{i-1}, 1, m-2, 1, \beta_{i+2}, \dots, \beta_n) $ then $W_\beta^i$ is a natural subvariety of $\subset Y_{\beta'}$ corresponding to the locus where $zL_{i+1} \subset L_{i-1}$ and $zL_{i+2} \subset L_i$. The maps $ \pi_1 : W_\beta^i \rightarrow Y_\beta, \pi_2 : W_\beta^i \rightarrow Y_{s_i(\beta)} $ are both birational, since they are isomorphisms over the open sets where $ zL_{i+1} \nsubseteq L_{i-1} $. For this reason we sometimes refer to $W_\beta^i$ as the diagonal-like component.  

The second component is the closure of the locus of points where $ L_i \nsubseteq L'_i $.  On this closure, we see that $ zL_{i+1} \subset L_{i-1} $.  So this component is $ V_\beta^i := X_{\beta}^i \times_{Y_{d_i(\beta)}} X_{s_i(\beta)}^i$. Thus as above, we conclude that $ Z_\beta^i = W_\beta^i \cup V_\beta^i $.

A similar analysis holds in the $ \beta_i = m-1, \beta_{i+1} = 1 $ case.

\begin{Remark} \label{re:sat2}
Let us now consider the variety $Z_\beta^i $ from the point of view of the geometric Satake correspondence.  We retain the notation and assumptions of Remark \ref{re:sat1}.  Define
\begin{equation*}
Z_{\lambda \mu} := \Gr_\lambda \tilde{\times} \Gr_\mu \times_{\overline{\Gr_{\lambda + \mu}}} \Gr_\mu \tilde{\times} \Gr_\lambda.
\end{equation*}
Using the stratification of $ \overline{\Gr_{\lambda + \mu}} $, we see that
\begin{equation*}
Z_{\lambda \mu} = \bigsqcup_\nu m_{\lambda \mu}^{-1}(\Gr_\nu) \times_{\Gr_\nu} m_{\mu \lambda}^{-1}(\Gr_\nu).
\end{equation*}
By Proposition \ref{th:Haines}, each piece in this decomposition has dimension $2 \langle \lambda + \mu - \nu, \rho \rangle + 2 \langle \nu, \rho \rangle = \langle \lambda + \mu, 2 \rho \rangle $. Since they are all disjoint, their closures are the components of $ Z_{\lambda \mu} $.  Thus, we conclude that the components of $ Z_{\lambda \mu} $ correspond to exactly those $ \nu $ occuring in the tensor product decomposition of $ V_\lambda \otimes V_\mu $.

If we are in the case of $ \mathfrak{g} = \mathfrak{sl}_m $, $ \lambda = \omega_k, \mu = \omega_l $, we see that $ V_\lambda \otimes V_\mu $ has two direct summands if and only if $ k, l \in \{ 1, m- 1 \} $.  This matches with the conclusions drawn above.  More specifically, we have the following components if $ k =l \in \{1, m-1 \}$
\begin{equation*}
\overline{m_{\omega_k \omega_k}^{-1} (\Gr_{2\omega_k}) \times_{\Gr_{2 \omega_k}} m_{\omega_k \omega_k}^{-1} (\Gr_{2\omega_k})}\quad 
m_{\omega_k \omega_k}^{-1} (\Gr_{2\omega_k - \alpha_k}) \times_{\Gr_{2\omega_k - \alpha_k}} m_{\omega_k \omega_k}^{-1} (\Gr_{2\omega_k - \alpha_k})
\end{equation*}
which correspond to $ \Delta_{Y_\beta} $, $ V_\beta^1$ respectively (where $ \beta = (k,l)$).  Here $ \alpha_k $ denotes the simple root corresponding to $ \omega_k $.

If $ k \ne l \in \{1, m-1\} $, we have components
\begin{equation*}
\overline{ m_{\omega_k \omega_l}^{-1} (\Gr_{\omega_k + \omega_l}) \times_{\Gr_{\omega_k + \omega_l}} m_{\omega_l \omega_k}^{-1} (\Gr_{\omega_k + \omega_l}) } \quad  
m_{\omega_k \omega_l}^{-1} (\Gr_0) \times_{\Gr_0} m_{\omega_l \omega_k}^{-1} (\Gr_0)
\end{equation*}
which correspond to $ W_\beta^1, V_\beta^1 $ respectively (where $ \beta = (k,l)$).
\end{Remark}

\subsection{Comparison with resolution of slices} \label{se:SS}
The main purpose of this subsection is to make contact with the work of Manolescu.  For this section, we will work with $ \beta(k) := (1, \dots, 1, m-1, \dots, m-1) $ where there are $ k $ ones and $ k $ $m-1$s.  Let $ Y_k = Y_{\beta(k)}$.  

\subsubsection{The Spaltenstein variety}
We consider the subspace $ \C^{mk} =  \C^{m} \otimes \mathrm{span}(z^{-1}, \dots, z^{-k}) \subset \C^{m} \otimes z^{-1} \C[z^{-1}] $. Let 
\begin{equation*}
F_k := \{ (L_1, \dots, L_{2k}) \in Y_k : L_{2k} = \C^{mk} \}.
\end{equation*}
Note that $ z $ restricted to $ \C^{mk} $ is a nilpotent operator of Jordan type $ \rho = (k,\dots, k) $ (where there are $m$ $k$s).  Hence $ F_k $ is the variety of all partial flags of type $ \beta(k) $ which are respected by a fixed nilpotent of type $ \rho $ (it is called a Spaltenstein variety).  

Now, we define a certain open neighbourhood of $ F_k $, inside $ Y_k $.  Define a linear operator $ P : \C^{m} \otimes z^{-1} \C[z^{-1}]  \rightarrow \C^{mk} $ by $ P (v \otimes z^{-i}) = v \otimes z^{-i} $ if $ i \le k  $ and $ P (v \otimes z^{-i}) = 0 $ if $ i > k $.  So $ P $ is a projection onto the subspace $ \C^{mk} $.  Define
\begin{equation*}
U_k := \{ (L_1, \dots, L_{2k}) \in Y_k : P(L_{2k}) = \C^{mk} \}.
\end{equation*}

\begin{Remark}
Let us comment on the meaning of these varieties from the point of view of the affine Grassmannian.  Recall that we had the convolution variety $ \Gr_{\ul} $ which comes with a map $ m : \Gr_{\ul} \rightarrow \Gr $.  Under the isomorphism $ Y_k \cong \Gr_{\ul} $, the subvariety $ F_k $ is identified with the variety $ m^{-1}(z^0)$ where $z^0$ is the $G^\vee(\C[[z]])$-orbit of the indentity matrix in $\Gr$ (this orbit is a point). Under the geometric Satake correspondence the push forward of the constant sheaf of $F_k$ by $ m $ corresponds to the tensor product $ V_{\beta(k)}$.  It can be easily seen (for example Prop 3.1 of \cite{Haines1}) that this implies that the top homology of $ F_k $ is isomorphic to the space of invariants $ (\C^m \otimes \cdots \otimes \C^m \otimes \Lambda^{m-1}(\C^m) \otimes \cdots \Lambda^{m-1}(\C^m))^{\mathfrak{sl}_m} $. 

The point $ z^0 $ in $\Gr $ has a natural open neigbourhood $ \phantom{}_0 Gr := G^\vee([z^{-1}]) \cdot z^0 $ in $ \Gr $.  Under the isomorphism $ Y_k \cong \Gr_{\beta(k)} $, the subvariety $ U_k $ is identified with the open subset $ m^{-1}(\phantom{}_0 Gr) \subset \Gr_{\beta(k)} $.
\end{Remark}
 
\subsubsection{Nilpotent slices}
Let $ e_k $ denote the nilpotent matrix of Jordan type $ \rho(k) = (k,\dots, k) $ and let $ f_k $ denote the matrix completing it to an $ \mathfrak{sl}_2 $ triple.  Let $ S_{\rho(k)} $ denote the set $ e_k + ker(f_k \cdot) $, which is a slice to the adjoint orbit of $ e_k $.  We consider the following variety
\begin{equation*}
\widetilde{S}_{\rho(k), \beta(k)} := \{ (X, V_\cdot) : X \in S_{\rho(k)}, \ V_\cdot \text{ is a partial flag in $ \C^{mk} $ of type } \beta(k), \ X V_i \subset V_{i-1} \}.
\end{equation*}
Note that $ \widetilde{S}_{\rho(k), \beta(k)} $ contains $ F_k $.

In \cite{kh1}, we considered the case when $ m = 2$.  Via similar reasoning to that paper, it is easy to establish the following result which is a special case of Theorem 3.2 in \cite{MVy}.
\begin{Proposition} \label{prop:Unslice}
There is an isomorphism $ U_k \rightarrow \widetilde{S}_{\rho(k), \beta(k)}$ which is the identity on $ F_k $.
\end{Proposition}

\subsubsection{Manolescu construction}
We can now explain one motivation for our work.

In his paper \cite{M}, Manolescu gives a construction of a link invariant using symplectic geometry, generalizing the construction of Seidel-Smith \cite{SS}.  Following \cite{M}, consider the subgroup $ \Sigma_k \times \Sigma_k \subset \Sigma_{2k} $, where $ \Sigma_n $ denotes the symmetric group, and consider the quotient 
$ BC_k := (\C^{2k} \smallsetminus \Delta) / \Sigma_k \times \Sigma_k $.  
For any $ \tau = (\lambda, \mu) \in BC_k $, consider the affine variety 
\begin{equation*}
M_{k, \tau} := \{ X \in S_k : X \text{ has eigenvalues } \lambda \text{ of multiplicity 1 and } \mu \text{ of multiplicity $m-1$} \}.
\end{equation*}
As $ \tau $ varies over $ BC_k $, the $M_{k, \tau} $  form a symplectic fibration.  Manolescu uses this fibration to construct an action of the braid group $ B_k \times B_k = \pi_1(BC_k) $ on the set of Lagrangian submanifolds of a given fibre $ M_{k, \tau} $.  Using this action he constructs a link invariant by associating to each link $ K $, $ HF(L, (\beta \times 1)(L)) $, where $ HF(,) $ denotes Floer cohomology, $ \sigma \in B_k $ is a braid whose closure is $ K $, $ L $ is certain chosen Lagrangian in $ M_{k, \tau} $ not depending on $ K $, and $ \sigma(L) $ denotes the above action of $ \sigma $ on $ L $.

The homological mirror symmetry principle suggests that there should exist some derived category of coherent sheaves equivalent to the Fukaya category of the affine K\"ahler manifold $ M_{k, \tau} $.   In particular, the braid group should act on this derived category and it should be possible to construct a link invariant in the analogous manner, namely as $ \Ext(L, (\sigma \times 1)(L))$ for an appropriate chosen object $ L $ in a certain derived category of coherent sheaves.  

By the general theory of nilpotent slices (see for example Proposition 2.11 of \cite{M}), $ M_{k, \tau} $ is isomorphic as a manifold to $ \widetilde{S}_{\rho(k), \beta(k)} $, but has a different complex structure.  In particular, they are related by a classic ``defomation vs. resolution'' picture for hyperK\"ahler singularities.  Hence hyperK\"ahler rotation suggests that the Fukaya category of $ M_{k,\tau} $ should be related to the derived category of coherent sheaves on $ \widetilde{S}_{\rho(k), \beta(k)}$ or perhaps the subcategory of complexes of coherent sheaves whose cohomologies are supported on the Spaltenstein variety $ F_k$.  

In our paper, we found it more convenient to work with the full derived category of the compactification $Y_k $ of $ U_k$ and to work with tangles, rather than just links presented as the closures of braids.  However, suppose that a link $ K $ is presented as the closure of braid $ \beta \in B_k $.  Tracing through our definitions from section \ref{se:funtangle}, we see that (ignoring the bigrading)
\begin{equation*}
\Kh(K) = H(\Psi(K)(\C)) \cong \Ext_{D(Y_k)}(L, \sigma(L)) 
\end{equation*}
where $ L $ is the structure sheaf of a certain component of the Spaltenstein variety $ F_k $ (tensored with a line bundle).  Thus, our link invariant has the same form as would be expected from this string theory reasoning.

\section{Properties of the Varieties Involved}

In this section we list some technical properties of our varieties $X_\beta^i$, $Y_\beta$ and $Z_\beta^i$ and in so doing describe more of the geometry. These results will be handy in later sections and we will often refer back to them. 

\subsection{A topological description of $Y_\beta$}

Earlier, we saw that $ Y_\beta $ is an iterated Grassmannian bundle.  In fact these bundles are topologically trivial.  The following is a nice way of seeing this and will be useful in what follows.

Fix a vector space $ \C^{m} $ with a Hermitian inner product.  

\begin{Theorem} \label{th:isoman}
There is a diffeomorphism $Y_\beta \rightarrow \prod_i \Gr(\beta_i, m) $.  Moreover under this isomorphism, $ X_\beta^i $ is taken to the submanifold $$ A_\beta^i := \{(l_1, \dots, l_i, l_{i+1}, \dots l_{n}) : l_i \perp l_{i+1} \} \subset \prod_i \Gr(\beta_i, m). $$
\end{Theorem}

First, let $ C : \C^{m} \otimes z^{-1} \C[z^{-1}] \rightarrow \C^{m} $ denote the linear map corresponding to setting $ z^{-k} $ to 0 if $ k > 1 $ (it is a residue map).  Next introduce the Hermitian inner product of $ \C^{m} \otimes z^{-1}\C[z^{-1}] $ given by $ \langle v \otimes z^k, w \otimes z^l \rangle = \delta_{k,l} \langle v, w \rangle $.

\begin{Lemma}
Let $ W \subset \vect$  be a subspace such that $ zW \subset W $.  Then $ C $ restricts to an unitary isomorphism $ z^{-1} W \cap W^{\bot} \rightarrow \C^{m} $.
\end{Lemma}

\begin{proof} The proof is identical to the proof of Lemma 2.2 of \cite{kh1}.
\end{proof}

\begin{proof}[Proof of Theorem \ref{th:isoman}]
Given $ (L_1, \dots, L_{n}) \in Y_\beta $, let $ M_1, \dots, M_{n} $ be the sequence of subspaces of $ \vect $ such that 
\begin{equation*}
L_k = M_1 \oplus \dots \oplus M_k \quad L_{k-1} \bot M_k.
\end{equation*}
Hence $ M_k $ is a $\beta_k$ dimensional subspace of the $m$ dimensional vector space $ z^{-1} L_{k-1} \cap L_{k-1}^{\bot} $.  By the lemma $ C(M_k) $ is a $\beta_k $ dimensional subspace of $ \C^{m} $.

Thus we have a map 
\begin{align*}
Y_\beta &\rightarrow \prod_k \Gr(\beta_k, m) \\
(L_1, \dots, L_{n}) &\mapsto (C(M_1), \dots, C(M_{n}))
\end{align*}

It is easy to see that this map is an isomorphism.

Now, let us consider $ (L_1, \dots, L_{n}) \in X_n^i $.  Since $ L_{i+1} \subset z^{-1} L_{i-1} $, we see that $ M_i, M_{i+1} $ are both subspaces of $ z^{-1} L_{i-1} \cap L_{i-1}^\bot$.  As they are perpendicular, they are sent by $ C $ to perpendicular subspaces of $ \C^{m} $.  Hence $ C(M_i) \perp C(M_{i+1}) $ as desired.
\end{proof}

\begin{Remark}
The fact that $ \Gr_\beta $ is topologically trivial is easy to see from the point of view of the affine Grassmannian.  Indeed it follows immediately from the fact that the projection $ G^\vee(\C((z))) \rightarrow \Gr $ admits a smooth lift.  This lift exists because we may identify $ \Gr $ (as a manifold) with the group of based polynomial maps from $ S^1 $ into the maximal compact subgroup of $ G^\vee $ (see \cite{G}).
\end{Remark}

One important corollary of Theorem \ref{th:isoman} is the following.
\begin{Corollary}\label{lem:transverse} For any $ \beta $ and any  $1 \le i < j \le n-1$, the subvarieties $X_\beta^i $ and $X_\beta^j $ intersect transversely in $ Y_\beta $.
\end{Corollary}
\begin{proof}
Since $X_\beta^i$ and $X_\beta^j$ are smooth submanifolds of $ Y_\beta $, it is enough to work with the topological $ Y_\beta $.  By Theorem \ref{th:isoman}, $ Y_\beta \cong \prod_i \Gr(\beta_i, m) $ with $ X_\beta^i $ taken to $ A_\beta^i $.  So it suffices to show that $ A_\beta^i $ and $ A_\beta^j $ intersect transversely.  Recall that 
\begin{equation*}
A_n^i := \{ (l_1, \dots, l_i, l_{i+1},\dots , l_{n}) : l_i \perp l_{i+1} \}.
\end{equation*}
Hence it is immediate $ A_\beta^i $ and $ A_\beta^j $ intersect transversely.
\end{proof}

From the fact that $ X_\beta^i $ and $ X_\beta^j $ meet transversely, it follows that all the interesections considered in this paper are transverse, by the same arguments as given in section 5.1 of \cite{kh1}.

\subsection{Normal Bundles and Relative Dualizing Sheaves}
Now we move in a different direction and compute normal bundles and relative dualizing sheaves.
\begin{Lemma}\label{lem:facts1} As usual, let $i$ be the equivariant inclusion $X_\beta^i \rightarrow Y_\beta$ and $q$ the equivariant projection $X_\beta^i \rightarrow Y_{d_i(\beta)}$ then
\begin{enumerate}
\item 
\begin{eqnarray*}
N_{X_\beta^i/Y_\beta} \cong \left\{
\begin{array}{ll}
(L_{i+1}/L_i)^\vee \otimes L_i/L_{i-1} \{2\} & \mbox{if $\beta_i + \beta_{i+1} \le m$}  \\
(z^{-1}L_{i-1}/L_i)^\vee \otimes L_i/zL_{i+1} \{2\} & \mbox{if $\beta_i + \beta_{i+1} \ge m$}
\end{array}
\right.
\end{eqnarray*}
so that 
\begin{eqnarray*}
\det N_{X_\beta^i/Y_\beta} \cong \left\{
\begin{array}{ll}
(\E_i)^{\beta_{i+1}} \otimes (\E_{i+1})^{-\beta_i} \{2 \beta_i \beta_{i+1}\} & \mbox{if $\beta_i + \beta_{i+1} \le m$} \\
(\bE_i)^{-\bbeta_{i+1}} \otimes (\bE_{i+1})^{\bbeta_i} & \mbox{if $\beta_i + \beta_{i+1} \ge m$} 
\end{array}
\right.
\end{eqnarray*}
\item 
\begin{eqnarray*}
\Omega_{X_\beta^i/Y_{d_i(\beta)}} \cong \left\{
\begin{array}{ll}
(L_{i+1}/L_i)^\vee \otimes L_i/L_{i-1} & \mbox{if $\beta_i + \beta_{i+1} \le m$}  \\
(z^{-1}L_{i-1}/L_i)^\vee \otimes L_i/zL_{i+1} & \mbox{if $\beta_i + \beta_{i+1} \ge m$}
\end{array}
\right.
\end{eqnarray*}
so that 
\begin{eqnarray*}
\omega_{X_\beta^i/Y_{d_i(\beta)}} \cong \left\{
\begin{array}{ll}
(\E_i)^{\beta_{i+1}} \otimes (\E_{i+1})^{-\beta_i} & \mbox{if $\beta_i + \beta_{i+1} \le m$} \\
(\bE_i)^{-\bbeta_{i+1}} \otimes (\bE_{i+1})^{\bbeta_i} \{-2 \bbeta_i \bbeta_{i+1}\} & \mbox{if $\beta_i + \beta_{i+1} \ge m$} 
\end{array}
\right.
\end{eqnarray*}
\item If $\beta_i = \beta_{i+1} \in \{1,m-1\}$ then $X_\beta^i \subset Y_\beta$ is a divisor with $Y_\beta(X_\beta^i) \cong \E_i \otimes \E_{i+1}^\vee \{2\beta_i\}$
\end{enumerate}
Notice that $N_{X_\beta^i/Y_\beta} \cong \Omega_{X_\beta^i/Y_{d_i(\beta)}} \{2\}$ so that
\begin{eqnarray*}
\det N_{X_\beta^i/Y_\beta} \cong 
\left\{
\begin{array}{ll}
\omega_{X_\beta^i/Y_{d_i(\beta)}} \{2 \beta_i \beta_{i+1}\} & \mbox{if $\beta_i+\beta_{i+1} \le m$} \\
\omega_{X_\beta^i/Y_{d_i(\beta)}} \{2 \bbeta_i \bbeta_{i+1}\} & \mbox{if $\beta_i+\beta_{i+1} \ge m$} 
\end{array}
\right.
\end{eqnarray*}
\end{Lemma}
\begin{proof}

(i). The map $z: L_{i+1} \rightarrow L_i \{2\}$ induces a morphism $L_{i+1}/L_i \rightarrow L_i/L_{i-1}\{2\}$ and hence a section of $\Hom_{Y_\beta}(L_{i+1}/L_i,L_i/L_{i-1}\{2\})$. If $\beta_i + \beta_{i+1} \le m$ this section is zero precisely over the locus where $z$ maps $L_{i+1}$ to $L_{i-1}$, namely $X_\beta^i$. Thus the normal bundle $N_{X_\beta^i/Y_\beta}$ is isomorphic to $(L_{i+1}/L_i)^\vee \otimes L_i/L_{i-1} \{2\}$. Subsequently, if $\beta_i + \beta_{i+1} \le m$ then 
\begin{eqnarray*}
\det(N_{X_\beta^i/Y_\beta}) &\cong& \det(L_{i+1}/L_i)^{-\dim(L_i/L_{i-1})} \otimes \det(L_i/L_{i-1} \{2\})^{\dim(L_{i+1}/L_i)} \\
&\cong& (\E_{i+1})^{-\beta_i} \otimes (\E_i)^{\beta_{i+1}} \{2 \beta_i \beta_{i+1}\}.
\end{eqnarray*}

One can also consider the map $z: z^{-1} L_{i-1} \rightarrow L_i \{2\}$ which induces a morphism $z^{-1} L_{i-1}/L_i \rightarrow L_i/zL_{i+1} \{2\}$. If $\beta_i + \beta_{i+1} \ge m$ the corresponding section is zero precisely over the locus where $z$ maps $z^{-1} L_{i-1}$ into $z L_{i+1}$ or equivalently $L_{i-1} \subset z L_{i+1}$, namely $X_\beta^i$ again. Thus the normal bundle $N_{X_\beta^i/Y_\beta}$ is isomorphic to $(z^{-1} L_{i-1}/L_i)^\vee \otimes L_i/zL_{i+1} \{2\}$.  Subsequently, if $\beta_i + \beta_{i+1} \ge m$ then 
\begin{eqnarray*}
\det(N_{X_\beta^i/Y_\beta}) &\cong& \det(z^{-1} L_{i-1}/L_i)^{-\dim(L_i/zL_{i+1})} \otimes \det(L_i/zL_{i+1} \{2\})^{\dim(z^{-1}L_{i-1}/L_i)} \\
&\cong& (\bE_i)^{-\bbeta_{i+1}}\{\bbeta_i(-\bbeta_{i+1})\} \otimes (\bE_{i+1})^{\bbeta_i} \{2 \bbeta_{i+1} \bbeta_i - \bbeta_{i+1} \bbeta_i\} \\
&\cong& (\bE_i)^{-\bbeta_{i+1}} \otimes (\bE_{i+1})^{\bbeta_i} 
\end{eqnarray*}

(ii). If $\beta_i + \beta_{i+1} \le m$ then $q: X_\beta^i \rightarrow Y_{d_i(\beta)}$ is the $\C^\times$-equivariant $\Gr(\beta_i,\beta_i+\beta_{i+1})$ bundle $\Gr(\beta_i, L_{i+1}/L_{i-1}) \rightarrow Y_{d_i(\beta)}$. So the relative cotangent bundle of $q$ is $\Hom(L_{i+1}/L_i, L_i/L_{i-1}) = (L_{i+1}/L_i)^\vee \otimes L_i/L_{i-1}$ and so $\omega_{X_\beta^i/Y_{d_i(\beta)}} \cong (\E_{i+1})^{-\beta_i} \otimes (\E_i)^{\beta_{i+1}}$. 

If $\beta_i + \beta_{i+1} \ge m$ then $q$ is the $\Gr(m-\beta_i,2m-\beta_i-\beta_{i+1})$ bundle $\Gr(m-\beta_i, z^{-1}L_{i-1}/zL_{i+1}) \rightarrow Y_{d_i(\beta)}$. So the relative cotangent bundle of $q$ is $\Hom(z^{-1}L_{i-1}/L_i,L_i/zL_{i+1}) = (z^{-1}L_{i-1}/L_i)^\vee \otimes L_i/zL_{i+1}$. As before, this means $\omega_{X_\beta^i/Y_{d_i(\beta)}} \cong (\bE_i)^{-\bbeta_{i+1}} \otimes (\bE_{i+1})^{\bbeta_i} \{-2\bbeta_i \bbeta_{i+1}\}$. 

(iii). If $\beta_i = \beta_{i+1} = 1$ then $X_\beta^i$ is the zero locus of $z: L_{i+1}/L_i \rightarrow L_i/L_{i-1} \{2\}$ so $\O_{Y_\beta}(X_\beta^i) \cong \E_{i+1}^\vee \otimes \E_i \{2\}$. If $\beta_i = \beta_{i+1} = m-1$ it is the zero locus of $z: z^{-1}L_{i-1}/L_i \rightarrow L_i/zL_{i+1} \{2\}$. Thus 
\begin{eqnarray*}
\O_{Y_\beta}(X_\beta^i) 
&\cong& (z^{-1}L_{i-1}/L_i)^\vee \otimes (L_i/zL_{i+1} \{2\}) \\
&\cong&  \E_i \{-2b_\beta^{i-1}-2m\} \otimes \E_{i+1}^\vee \{2b_\beta^{i+1}+2\} \\
&\cong& \E_{i+1}^\vee \otimes \E_i \{2\beta_i\}.
\end{eqnarray*}
\end{proof}

\begin{Lemma}\label{lem:facts2} We have
$\omega_{Y_\beta} \cong \det(L_n/L_0)^m \{- 2m b_\beta^n - 2 \sum_i b_\beta^{i-1} \beta_i \}.$
\end{Lemma}
\begin{proof}
Consider the projection map $\pi: Y_\beta \rightarrow Y_{\beta'}$ given by $(L_1, \dots, L_{n-1},  L_n) \rightarrow (L_1, \dots, L_{n-1})$ where $\beta' = d_n(\beta)$. Then $\pi$ is the Grassmannian bundle $Gr(\beta_n, z^{-1}L_{n-1}/L_{n-1})$. Hence 
\begin{eqnarray*}
\omega_{Y_\beta} 
&\cong& \omega_{\pi} \otimes \pi^* \omega_{Y_{\beta'}} \\
&\cong& \det(L_n/L_{n-1})^{m-\beta_n} \otimes \det(z^{-1}L_{n-1}/L_n)^{-\beta_n} \otimes \pi^* \omega_{Y_{\beta'}} \\
&\cong& \E_n^{m-\beta_n} \otimes (\E_n^\vee \{2 b_\beta^{n-1}+2m\})^{-\beta_n} \otimes \pi^* \omega_{Y_{\beta'}} \\
&\cong& \E_n^m \otimes \pi^* \omega_{Y_{\beta'}} \{-2m \beta_n - 2 b_{\beta}^{n-1} \beta_n\}
\end{eqnarray*}
The result follows by induction. 
\end{proof}

\begin{Lemma}\label{lem:facts3} As usual suppose that $\beta_i, \beta_{i+1} \in \{1,m-1\}$. Then $Z^i_\beta$, which is a local complete intersection, has dualizing sheaf
$$\omega_{Z^i_\beta} \cong \begin{cases}
\E_i' \otimes \E_{i+1}^\vee \otimes \pi_1^* \omega_{Y_\beta} \{2\beta_i\} \text{ if } \beta_i = \beta_{i+1} \\
(\E_i' \otimes \E_{i+1}^\vee)^{m-1} \otimes \pi_1^* \omega_{Y_\beta} \{2(m-1)\} \text{ if } \beta_i \ne \beta_{i+1}
\end{cases}$$
Also, if $\pi_1$ and $\pi_2$ denote the natural projections from $Z_\beta^i$ then $\pi_{1*} \O_{Z_\beta^i} \cong \O_{Y_\beta}$ and $\pi_{2*} \O_{Z_\beta^i} \cong \O_{Y_{s_i(\beta)}}$ while $\pi_{1*} \omega_{Z_\beta^i} \cong \omega_{Y_\beta}$ and $\pi_{2*} \omega_{Z_\beta^i} \cong \omega_{Y_{s_i(\beta)}}$. 
\end{Lemma}
\begin{proof}
If $\beta_{i+1}=1$ consider the Grassmannian bundle $\pi: W := Gr(\beta_{i+1}, L_{i+1}/L_{i-1}) \rightarrow Y_\beta$. Let $\rho: Z_\beta^i \rightarrow W$ be the natural embedding. The map $z: L_{i+1}/L_i \rightarrow L_{i+1}/L_i' \{2\}$ where $L_i'$ is the tautological bundle on $W$ vanishes along the locus where $zL_{i+1} \subset L_i'$. 

There are two cases to consider depending on whether or not $z$ maps $L_{i+1}$ down into $L_{i-1}$. If it does then $z$ takes $L_i'$ to $L_{i-1}$. If it does not then $\mbox{span}(zL_{i+1},L_{i-1}) = L_i'$ since $\dim(L_i'/L_{i-1})=\beta_{i+1}=1$. Since $z^2 L_{i+1} \subset L_{i-1}$ we find again that $z$ maps $L_i'$ into $L_{i-1}$. This means $z$ provides a section of $V = (L_{i+1}/L_i)^\vee \otimes L_{i+1}/L_i' \{2\}$ which vanishes precisely along $\rho(Z_\beta^i)$. 

Since $\dim(V) = \beta_i \beta_{i+1} = \mbox{codim}(Z_\beta^i)$ we find that $Z_\beta^i$ is a local complete intersection while $\omega_{Z_\beta^i} \cong \omega_{W}|_{Z_\beta^i} \otimes \det(V)$. On the other hand, 
$$\omega_{W} \cong \omega_\pi \otimes \pi^* \omega_{Y_\beta} \cong \det(L_i'/L_{i-1} \otimes (L_{i+1}/L_i')^\vee) \otimes \pi^* \omega_{Y_\beta} \cong (\E_i')^{\beta_i} \otimes (\E_{i+1}')^\vee \otimes \pi^* \omega_{Y_\beta}.$$
Thus 
$$\omega_{Z_\beta^i} \cong ((\E_i')^{\beta_i} \otimes (\E_{i+1}')^\vee \otimes \pi^* \omega_{Y_\beta}) \otimes ((\E_{i+1})^{-\beta_i} \otimes \E_{i+1}') \{2\beta_i\} \cong (\E_i' \otimes \E_{i+1}^\vee)^{\beta_i} \otimes \pi_1^* \omega_{Y_\beta} \{2\beta_i\}$$
where we use that $\pi \circ \rho = \pi_1$ to replace $\pi^*$ by $\pi_1^*$. 

If $\beta_{i+1} = m-1$ and $\beta_i = 1$ the same argument works if we use instead the map $z: L_i'/L_{i-1} \rightarrow L_i/L_{i-1} \{2\}$. In this case $\omega_{W} \cong \E_i' \otimes (\E_{i+1}')^{-m+1} \otimes \pi^* \omega_{Y_\beta}$ so that
$$\omega_{Z_\beta^i} \cong (\E_i' \otimes (\E_{i+1}')^{-m+1} \otimes \pi^* \omega_{Y_\beta}) \otimes ((\E_i')^\vee \otimes \E_i^{m-1}) \{2(m-1)\} \cong (\E_i' \otimes \E_{i+1}^\vee)^{m-1} \otimes \pi_1^* \omega_{Y_\beta} \{2(m-1)\}$$
where we use that $\E_i \otimes \E_{i+1} \cong \det(L_{i+1}/L_{i-1}) \cong \E_i' \otimes \E_{i+1}'$ on $Z_\beta^i$. 

If $\beta_{i+1} = m-1 = \beta_i$ we need to use $\pi: W := Gr(1, z^{-1}L_{i-1}/zL_{i+1}) \rightarrow Y_\beta$ and a similar argument applies.

Finally, we prove that $\pi_{1*} \O_{Z_\beta^i} \cong \O_{Y_\beta}$ when $\beta_i=m-1$ and $\beta_{i+1}=1$ (the other cases follow similarly). From above we know $\O_{Z_\beta^i}$ has a Koszul resolution 
$$\wedge^{m-1} \mathcal{V}^\vee \rightarrow \dots \rightarrow \wedge^2 \mathcal{V}^\vee \rightarrow \mathcal{V}^\vee \rightarrow \O_{W}$$
where $\mathcal{V} = (L_{i+1}/L_i)^\vee \otimes L_{i+1}/L_i' \{2\}$. On the other hand, $L_{i+1}/L_i$ restricted to a fibre of $\pi_1$ is the trivial line bundle while $L_{i+1}/L_i'$ restricts to the tautological quotient bundle $Q$. Since $Q \otimes \O_{\p}(1) \cong T_{\p}$ we get that $V^\vee$ restricts to $Q^\vee \cong \Omega_{\p}(1)$ and hence $\wedge^j V^\vee$ restricts to $\Omega^j(j)$. It is then an elementary exercise to check that $H^k(\p, \Omega_{\p}^j(j)) = 0$ for $j > 0$ and any $k$. This means $\pi_{1*} \wedge^j V^\vee = 0$ for $j > 0$ and so $\pi_{1*} \O_{Z_\beta^i} \cong \pi_{1*} \O_{W} \cong \O_{Y_\beta}$. 

The fact that $\pi_{1*} \omega_{Z_\beta^i} \cong \omega_{Y_\beta}$ is a formal consequence of $\pi_{1*} \O_{Z_\beta^i} \cong \O_{Y_\beta}$. 

\end{proof}

\section{Functors from tangles} \label{se:funtangle}
Now we specialize to only consider those $ \beta $ where each $ \beta_i $ is either $ 1 $ or $ m-1 $ (note however that some of the components of $ d_i(\beta) $ may be 2 or $ m-2 $).

An oriented $(\beta,\beta') $ \textbf{tangle} is a proper, smooth embedding of $ (n+n')/2 $ arcs and a finite number of circles into $ \mathbb{R}^2 \times [0,1] $ such that the boundary points of the arc maps bijectively on the $n+n' $ points $ (1,0,0), \dots, (n,0,0), (1,0,1), \dots, (n',0,1) $ and such that the arc is oriented down (resp up) at the point $(i, 0, 0) $ (resp $(i,0,1) $) if and only if $ \beta_i = 1$.  Thus, $ \beta, \beta' $ keep track of the orientations of the endpoints of the tangle. A $ (0,0) $ tangle is a \textbf{link}.

Given an $(\beta,\beta')$ tangle $ T $ and a $ (\beta',\beta'') $ tangle $ U$, there is a composition tangle $ T \circ U $, which is the $ (\beta,\beta'') $ tangle made by stacking $ U $ on top of $ T $ and the shrinking the $z$-direction.

\subsection{Generators and Relations}

By projecting to $\mathbb{R} \times [0,1]$ from a generic point we can represent any tangle as a planar diagram. By scanning the diagram of a $(\beta, \beta')$ tangle from top to bottom we can decompose it as the composition of cups, caps and crossings. The list of all such building blocks consists of those tangles from Figure \ref{f1} together with those obtained from them by switching all the orientations (arrows). 

\begin{figure}
\begin{center}
\psfrag{cup}{\footnotesize{$(1,m-1)$}}
\psfrag{e}{\footnotesize{cup}}
\psfrag{cap}{\footnotesize{$(1,m-1)$}} 
\psfrag{f}{\footnotesize{cap}}
\psfrag{crossing 1}{\footnotesize{$(1,1)$ crossing}} 
\psfrag{a}{\footnotesize{of type 1}}
\psfrag{crossing 2}{\footnotesize{$(1,1)$ crossing}} 
\psfrag{b}{\footnotesize{of type 2}}
\psfrag{crossing 3}{\footnotesize{$(1,m-1)$ crossing}}
\psfrag{c}{\footnotesize{of type 1}}
\psfrag{crossing 4}{\footnotesize{$(1,m-1)$ crossing}} 
\psfrag{d}{\footnotesize{of type 2}}
\puteps[0.35]{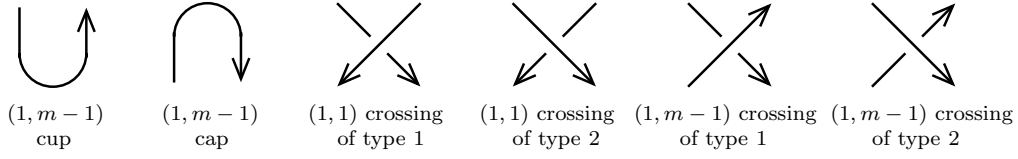}
\end{center}
\caption{the generators for tangle diagrams.}\label{f1}
\end{figure}

We will give these generators names depending on their position.  A cap creating the $ i $ and $ i+1 $ strands and so making a $ (d_i(\beta), \beta)  $ tangle will be denoted $ g_\beta^i $.  A cup connecting the $ i $ and $ i+1 $ strands and so making a $ (\beta, d_i(\beta)) $ tangle will be denoted by $ f_\beta^i $.  A crossing of the $ i $ and $ i+1 $ strands and so making a  $ (\beta, s_i(\beta)) $ strand tangle will be denoted $ t_\beta^i(l) $, where $ l $ varies from 1 to 2, depending on the type of crossing, as shown in Figure \ref{f1}. Any generator which is obtained from these generators by reversing the direction of all arrows involved will be denoted by the same symbol.

The following theorem tells us when two tangle diagrams represent isotopic tangles.

\begin{Lemma}[{\cite[Lemma X.3.5]{K}}]\label{lem:relations} Two tangle diagrams represent isotopic tangles if and only if one can be obtained from the other by applying a finite number of the following operations:
\begin{itemize}
\item a Reidemeister move of type (0),(I),(II) or (III).
\item an isotopy exchanging the order with respect to height of two caps, cups, or crossings (e.g. the left Figure in \ref{f3} shows such an isotopy involving a cup and a cap). 
\item the rightmost two isotopies in Figure \ref{f3}, which we call the pitchfork move.
\end{itemize}
More concisely, we have the following relations:
\begin{enumerate}
\item Reidemeister (0) : $ f_\beta^i \circ g_\beta^{i+1} = id = f_\beta^{i+1} \circ g_\beta^i $ 
\item Reidemeister (I) : $f_{\beta}^i \circ t_\beta^{i \pm 1}(2) \circ g_\beta^i = id = f_{\beta}^i \circ t_\beta^{i \pm 1}(1) \circ g_\beta^i$ 
\item Reidemeister (II) : $ t_{s_i(\beta)}^i(2) \circ t_\beta^i(1) = id = t_{s_i(\beta)}^i(1) \circ t_\beta^i(2) $ 
\item Reidemeister (III) : $ t_{s_{i+1}(s_i(\beta))}^i(l_1) \circ t_{s_i(\beta)}^{i+1}(l_2) \circ t_\beta^i(l_3) = t_{s_i(s_{i+1}(\beta))}^{i+1}(l_3) \circ t_{s_{i+1}(\beta)}^i(l_2) \circ t_\beta^{i+1}(l_1) $ 
\item changing height isotopies, such as : $ g_{\beta}^{i+k} \circ g_{d_k(\beta)}^i = g_{\beta}^i \circ g_{d_i(\beta)}^{i+k-2} $ 
\item pitchfork move : $ t_{s_i(\beta)}^i(1) \circ g_{s_i(\beta)}^{i+1} = t_{s_{i+1}(\beta)}^{i+1}(2) \circ g_{s_{i+1}(\beta)}^i, \quad  t_{s_i(\beta)}^i(2) \circ g_{s_i(\beta)}^{i+1} = t_{s_{i+1}(\beta)}^{i+1}(1) \circ g_{s_{i+1}(\beta)}^i $
\end{enumerate}
where in each case $ \beta $ is chosen such that all the expressions are well defined.
\end{Lemma}

\begin{figure}
\begin{center}
\psfrag{move (0)}{R-move (0)} 
\psfrag{move (I)}{R-move (I)} 
\psfrag{move (II)}{R-move (II)} 
\psfrag{move (III)}{R-move (III)} 
\psfrag{sim}{$\sim$} 
\puteps[0.20]{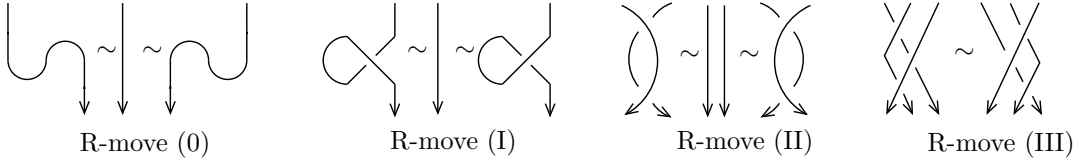}
\end{center}
\caption{Reidemeister relations for tangle diagrams.}\label{f2}
\end{figure}

\begin{figure}
\begin{center}
\psfrag{sim}{$\sim$} 
\puteps[0.25]{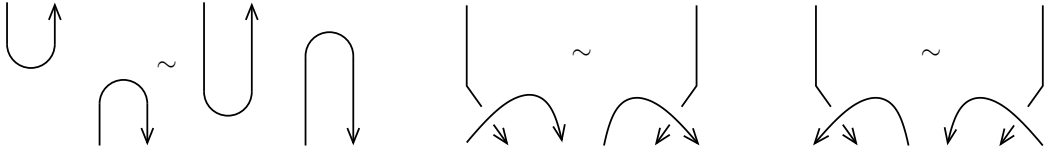}
\end{center}
\caption{Relations for tangle diagrams.}\label{f3}
\end{figure}

\subsection{The Functor $ \Psi(T): D(Y_{\beta}) \rightarrow D(Y_{\beta'}) $}\label{sse:functor}
To each $(\beta,\beta')$ tangle $T$, we will construct an isomorphism class of functor $ \Psi(T) : D(Y_\beta) \rightarrow D(Y_{\beta'}) $.  These functors will satisfy the property that $ \Psi(T) \circ \Psi(U) \cong \Psi(T \circ U) $. 

First we will define functors for each of the elementary tangles in Figure \ref{f1}. We define the functors corresponding to the cap $g_\beta^i$, cup $f_\beta^i$ and crossing $t_\beta^i(l)$ in equations (\ref{eq:Gdef}), (\ref{eq:Fdef}) and (\ref{eq:Cdef}) respectively.  

More precisely, we will assign a Fourier-Mukai kernel $ \sG_\beta^i, \dots $ to each basic tangle.  Recall that if $ X , Y $ are two smooth projective varieties with a $ \C^\times $ action, then a \textbf{Fourier-Mukai kernel} is any object $ \sP $ of the derived category of $\C^\times$-equivariant coherent sheaves on $ X \times Y $.  Given $\sP \in D(X \times Y)$, we may define the associated \textbf{Fourier-Mukai transform}, which is the functor
\begin{equation*}
\begin{aligned}
\Phi_\sP : D(X) &\rightarrow D(Y) \\
\sF &\mapsto {\pi_2}_* (\pi_1^* (\sF) \otimes \sP) 
\end{aligned}
\end{equation*}
Fourier-Mukai transforms have right and left adjoints which are themselves Fourier-Mukai transforms.  In particular, the right adjoint of $ \Phi_\sP $ is the FM transform with respect to $ \sP_R := \sP^\vee \otimes \pi_2^* \omega_X [\dim(X)] \in D(Y \times X)$.  Similarly, the left adjoint of $ \Phi_\sP $ is the FM transform with respect to $ \sP_L := \sP^\vee \otimes \pi_1^* \omega_Y [\dim(Y)] $, also viewed as a sheaf on $ Y \times X $.

We can express composition of Fourier-Mukai transforms in terms of their kernels.  If $ X, Y, Z $ are varieties and $\Phi_\sP : D(X) \rightarrow D(Y),  \Phi_\sQ : D(Y) \rightarrow D(Z) $ are Fourier-Mukai transforms, then $ \Phi_\sQ \circ \Phi_\sP $ is a FM transform with respect to the kernel
\begin{equation*}
\sQ \ast \sP := {\pi_{13}}_*(\pi^*_{12}(\sP) \otimes \pi^*_{23}(\sQ)).
\end{equation*}
The operation $*$ is associative and $ (\sQ * \sP)_R \cong \sP_R * \sQ_R $.

See section 3.1 of \cite{kh1} or section 5.1 of \cite{H} for more details regarding Fourier-Mukai kernels and transforms.
 
\subsubsection{Cups and caps} \label{se:capcup}
Suppose for now that $\beta_i + \beta_{i+1} = m$. Recall that we have the equivariant subvariety $ X_\beta^i \subset Y_\beta $ of codimension $ m-1 $ which projects to $ Y_{d_i(\beta)}$ with fibre $ Gr(\beta_i, m) $. Thus, we may regard $ X_\beta^i $ as a subvariety of the product $ Y_{d_i(\beta)} \times Y_\beta $.  Let $ \sG_\beta^i $ denote the $\C^\times$-equivariant sheaf on $ Y_{d_i(\beta)} \times Y_\beta $ defined by 
\begin{equation*} 
\sG_\beta^i := \O_{X_\beta^i} \otimes (\E'_i)^{\beta_{i+1}} \{-(i-1)(m-1)\}. 
\end{equation*}

Now, we define the functor $ \G_\beta^i : D(Y_{d_i(\beta)}) \rightarrow D(Y_\beta) $ to be the Fourier-Mukai transform with respect to the kernel $ \sG_\beta^i $. We will use this functor for the cap, so we define 
\begin{equation} \label{eq:Gdef}
\Psi(g_\beta^i) := \G_\beta^i.
\end{equation}

We can give a different description of the functor $ \G_\beta^i $.  Namely, $\G_\beta^i(\cdot) = i_\ast(q^\ast(\cdot) \otimes (\E_i)^{\beta_{i+1}} \{- (i-1) (m-1)\})$ where we make use of the diagram 
\begin{equation*}\begin{CD}X_\beta^i @>i>> Y_\beta \\@VqVV \\ Y_{d_i(\beta)}\end{CD} \end{equation*}

Now, we define $\F_\beta^i : Y_\beta \rightarrow Y_{d_i(\beta)} $ by $\F_\beta^i(\cdot) = q_\ast(i^\ast(\cdot) \otimes (\E_{i+1})^{-\beta_i})\{i (m-1) \} $. We define
\begin{equation} \label{eq:Fdef}
\Psi(f_\beta^i) := \F_\beta^i.
\end{equation} 
As with $ \G_\beta^i $, the functor $ \F_\beta^i $ can also be described as a Fourier-Mukai transform with respect to the kernel $\sF_\beta^i := \O_{X_\beta^i} \otimes \E_{i+1}^\vee \{i (m-1) \}$.  

\subsubsection{Crossings} \label{se:crossing}
To a crossing connecting boundary points $i$ and $i+1$ we assign a Fourier-Mukai kernel $ \mathcal{T}_\beta^i(l) \in  D(Y_\beta \times Y_{s_i(\beta)})$ according to the type of crossing:
\begin{itemize}
\item crossing \#1: 
$$ \sT_\beta^i(1) := \begin{cases}
\O_{Z_\beta^i}[m-1]\{-m+1\} \text{ if } \beta_i=\beta_{i+1} \in \{1, m-1\} \\
\O_{Z_\beta^i} \otimes (\E_i' \otimes \E_{i+1}^\vee)^{\beta_i-1} [-m+1]\{2(m-1)\} \text{ if } \beta_i \ne \beta_{i+1} \in \{1,m-1\} 
\end{cases} $$
\item crossing \#2: 
$$ \sT_\beta^i(2) := \begin{cases}
\O_{Z_\beta^i} \otimes \E_{i+1}^\vee \otimes \E'_i [-m+1] \{m-1 + 2\beta_i\} \text{ if } \beta_i= \beta_{i+1} \in \{1, m-1\} \\
\O_{Z_\beta^i} \otimes (\E'_i \otimes \E_{i+1}^\vee)^{\beta_i} [m-1] \text{ if } \beta_i \ne \beta_{i+1} \in \{1, m-1\} 
\end{cases} $$
\end{itemize}
Now that we have these kernels, we associate to each crossing a functor 
\begin{equation}\label{eq:Cdef}
\Psi(t_\beta^i(l)) := \T_\beta^i(l) := \Phi_{\sT_\beta^i(l)} : D(Y_\beta) \rightarrow D(Y_{s_i(\beta)})
\end{equation}
where as usual $ l \in \{1,2\}$. 

The reader may wonder what is the origin of mysterious looking line bundles and shifts occuring in the definitions of $\sT_\beta^i(l)$. The thing to keep in mind in that once the line bundle and shifts on one of the crossings are fixed (say for $\sT_{1,1}^1(1)$) then the choices of line bundles and shifts on the other crossings are uniquely determined via Reidemeister II and the pitchfork identities (section \ref{se:pitchfork}). 

\subsubsection{Functor for a tangle} 
Let $ T $ be a tangle.  Then, scanning a projection of $ T $ from top to bottom and composing along the way, gives us a functor $ \Psi(T) : D(Y_\beta) \rightarrow D(Y_{\beta'}) $.  However, this functor may depend on the choice of tangle projection.  

\begin{Theorem}\label{thm:main}
The isomorphism class of the functor $\Psi(T): D(Y_\beta) \rightarrow D(Y_{\beta'})$ associated to the planar diagram of an $(\beta,\beta')$ tangle $T$ is a tangle invariant. 
\end{Theorem}

To prove this theorem, we must check that the functors assigned to the elementary tangles satisfy the relations from Lemma \ref{lem:relations}.  This will be done in section \ref{se:invariance}.

This construction associates to any link $L$ a functor $\Psi(L): D(Y_0) \rightarrow D(Y_0)$. Since $Y_0$ is a point with the trivial action of $\C^\times$, $\Psi(L)$ is determined by $\Psi(L)(\C) \in D(Y_0)$ which is a complex of graded vector spaces. We denote by $\Kh^{i,j}(L)$ the $j$-graded piece of the $i$th cohomology group of $\Psi(L)(\C)$ (so $i$ marks the cohomological degree and $j$ marks the graded degree). Since $\Psi(L)$ is a tangle invariant $\Kh^{i,j}(L)$ is an invariant of the link $L$. 

\subsection{Properties of the kernels}
Recall that if $ \sP $ is a Fourier-Mukai kernel, then the right and left adjoints of $\Phi_\sP$ are given as FM transforms with respect to the kernels $\sP_R := \sP^\vee \otimes \pi_2^* \omega_X [\dim(X)] $ and $ \sP_L := \sP^\vee \otimes \pi_1^* \omega_Y [\dim(Y)] $, viewed as sheaves on $ Y \times X $.

Using this terminology, we see that the functors $ \F_\beta^i $ and $ \G_\beta^i $ are related by the following Lemma.

\begin{Lemma} \label{th:Gadj} If $\beta_i+\beta_{i+1}=m$ then 
$$\sF_\beta^i \cong {\sG_\beta^i}_R[(m-1)]\{-(m-1)\} \cong {\sG_\beta^i}_L[-(m-1)]\{(m-1)\}.$$
In particular, $\F_\beta^i(\cdot) \cong {\G_\beta^i}_R(\cdot) [m-1] \{-(m-1)\} \cong {\G_\beta^i}_L(\cdot) [-(m-1)] \{m-1\}.$  
\end{Lemma}
\begin{proof}
The dual of $\O_{X_\beta^i}$ in $Y_{d_i(\beta)} \times Y_\beta$ is 
$$\O_{X_\beta^i}^\vee \cong \omega_{X_\beta^i} \otimes \pi_1^* \omega_{Y_{d_i(\beta)}}^\vee \otimes \pi_2^* \omega_{Y_\beta}^\vee [-\dim(X_\beta^i)]$$
and so
\begin{eqnarray*}
{\sG_\beta^i}_R &\cong& \O_{X_\beta^i}^\vee \otimes (\E_i)^{-\beta_{i+1}} \{(i-1) (m-1) \} \otimes \pi_1^* \omega_{Y_{d_i(\beta)}} [\dim(Y_{d_i(\beta)})] \\
&\cong& \omega_{X_\beta^i} \otimes \pi_2^* \omega_{Y_\beta}^\vee \otimes (\E_i)^{-\beta_{i+1}} \{(i-1) (m-1)\} [-(m-1)] \\
&\cong& \O_{X_\beta^i} \otimes (\E_i)^{\beta_{i+1}} \otimes (\E_{i+1})^{-\beta_i} \{2 (m-1) \} \otimes (\E_i)^{-\beta_{i+1}} \{(i-1) (m-1) \} [-(m-1)] \\
&\cong& \sF_\beta^i [-(m-1)] \{(m-1)\}
\end{eqnarray*}
where, to obtain the third line, we use Lemma \ref{lem:facts1}. 

Similarly,
\begin{eqnarray*}
{\sG_\beta^i}_L &\cong& \O_{X_\beta^i}^\vee \otimes (\E_i)^{-\beta_{i+1}} \{(i-1) (m-1) \} \otimes \pi_2^* \omega_{Y_\beta} [\dim(Y_\beta)] \\
&\cong& \omega_{X_\beta^i} \otimes \pi_1^* \omega_{Y_{d_i(\beta)}}^\vee \otimes (\E_i)^{-\beta_{i+1}} \{(i-1) (m-1) \} [m-1] \\
&\cong& \O_{X_\beta^i} \otimes (\E_i)^{\beta_{i+1}} \otimes (\E_{i+1})^{-\beta_i} \otimes (\E_i)^{-\beta_{i+1}} \{(i-1) (m-1) \} [m-1] \\
&\cong& \sF_\beta^i [(m-1)] \{-(m-1)\}.
\end{eqnarray*}
\end{proof}

The functors for under and over crossings are related by the following Lemma. 

\begin{Lemma} \label{th:twistadjoint} We have that  
$ {\sT_\beta^i(1)}_L \cong \sT_{s_i(\beta)}^i(2) $
\end{Lemma}
\begin{proof}
In \cite{kh1} we computed ${\sT_\beta^i(2)}_L$ by breaking up the kernel into two pieces, computing their dual and reassambling. This time we compute the dual of $\O_{Z_\beta^i}$ directly since by Lemma \ref{lem:facts3} we know $\omega_{Z_\beta^i}$.

If $\beta_i = \beta_{i+1} $, then $\beta= s_i(\beta) $ and we have 
\begin{eqnarray*}
\sT_\beta^i(1)_L &\cong& (\O_{Z_\beta^i} [m-1] \{-m+1\})^\vee \otimes \pi_2^\ast \omega_{Y_{s_i(\beta)}}[\dim(Y_\beta)] \\
&\cong& (\omega_{Z_\beta^i} \otimes \omega_{Y_\beta \times Y_{s_i(\beta)}}^\vee [-\dim(Y_\beta)]) [-m+1]\{m-1\} \otimes  \pi_2^* \omega_{Y_{s_i(\beta)}} [\dim(Y_\beta)] \\
&\cong& \O_{Z_\beta^i} \otimes \E_i' \E_{i+1}^\vee \otimes \pi_1^* \omega_{Y_\beta} \{2\beta_i\}) \otimes \pi_1^* \omega_{Y_\beta}^\vee [-m+1] \{m-1\} \\
&\cong& \O_{Z_\beta^i} \otimes \E_i' \E_{i+1}^\vee [-m+1] \{m-1+2\beta_i)\}.
\end{eqnarray*}

If $\beta_i \ne \beta_{i+1}$ we have 
\begin{align*}
\sT_\beta^i(1)_L &\cong (\O_{Z_\beta^i} \otimes (\E_i {\E'_{i+1}}^\vee)^{\beta_i -1} [-m+1] \{2(m-1)\})^\vee \otimes \pi_2^\ast \omega_{Y_{s_i(\beta)}}[\dim(Y_\beta)] \\
&\cong (\omega_{Z_\beta^i} \otimes \omega_{Y_\beta \times Y_{s_i(\beta)}}^\vee [-\dim(Y_\beta)]) \otimes (\E_i {\E'_{i+1}}^\vee)^{1-\beta_i} [m-1]\{-2(m-1)\}) \otimes \pi_2^* \omega_{Y_{s_i(\beta)}} [\dim(Y_\beta)] \\
&\cong \O_{Z_\beta^i} \otimes (\E_i' \E_{i+1}^\vee)^{m-1} \otimes \pi_1^* \omega_{Y_\beta} \{2(m-1)\}) \otimes (\E_i' \E_{i+1}^\vee)^{1-\beta_i} \otimes \pi_1^* \omega_{Y_\beta}^\vee [m-1] \{-2(m-1)\} \\ 
&\cong \O_{Z_\beta^i} \otimes (\E_i' \E_{i+1}^\vee)^{m-\beta_i} [m-1] 
\end{align*}
Let $ \beta' = s_i(\beta) $.  Then $ \beta'_i = \beta_{i+1} = m - \beta_i $ while $ \beta'_{i+1} = \beta_i $. The result follows.
\end{proof}

\subsubsection{Twists}  \label{se:sptwists}
We now give an alternate description of the functors $ \T_\beta^i(l) $ when $ \beta_i = \beta_{i+1} $ in terms of twists. We briefly recall the concept of a twist. Let $ \Phi_\sP : D(X) \rightarrow D(Y) $ be a FM transform with respect to $ \sP \in D(X \times Y)$. The FM kernel $\sT_\P := \Cone(\sP * \sP_R \xrightarrow{\beta_{\sP}} \O_\Delta) \in D(Y \times Y)$ constructed using the natural adjunction map $\beta_\sP$ induces a functor $\Phi_{\sT_\sP}: D(Y) \rightarrow D(Y)$ which is called the \textbf{twist} with respect to $\Phi_{\sP}$. 

Let $\sP_\beta^i = \O_{X_\beta^i} \otimes \E'_i \in D(Y_{d_i(\beta)} \times Y_\beta)$. It turns out $\T_\beta^i(2)$ can be described as a twist with respect to $ \P_\beta^i := \Phi_{\sP_\beta^i} : D(Y_{d_i(\beta)}) \rightarrow D(Y_\beta) $. The essential reason for this is that $ Z_\beta^i $ has two equi-dimensional components, one of which is the diagonal $ \Delta $ and the other is the fibre product $ V_\beta^i = X_\beta^i \times_{Y_{d_i(\beta)}} X_\beta^i $ which intersect in a divisor.

\begin{Theorem} \label{th:kerneltwist}
Suppose $\beta_i = \beta_{i+1}$. The functor $ \T_\beta^i(2) $ is the twist in the functor $ \P_\beta^i $ shifted by $[-m+1]\{m-1\}$. In particular
\begin{equation*}
\sT_\beta^i(2) \cong \sT_{\sP_\beta^i}[-m+1]\{m-1\}.
\end{equation*}
\end{Theorem}
\begin{proof} 
Since $\beta_i=\beta_{i+1}$, $s_i(\beta)=\beta$ so we just write $Y_\beta$ instead of $Y_{s_i(\beta)}$. We need to show that 
$$\O_{Z_\beta^i} \otimes \E_{i+1}^\vee \otimes \E_i' \{2\beta_i\} \cong \Cone(\sP_\beta^i \ast (\sP_\beta^i)_R \rightarrow \O_\Delta) \in D(Y_\beta \times Y_\beta).$$

Working on $Y_\beta \times Y_{d_i(\beta)} \times Y_\beta$, by Lemma \ref{lem:Padj} below, we have
\begin{eqnarray*}
\sP_\beta^i \ast (\sP_\beta^i)_R &\cong& \pi_{13 \ast} ( \pi_{12}^\ast (\O_{X_\beta^i} \otimes \E_{i+1}^\vee [-1] \{2\beta_i\}) \otimes \pi_{23}^\ast (\O_{X_\beta^i} \otimes \E_i')) \\
&\cong& \pi_{13*}(\pi_{12}^* \O_{X_\beta^i} \otimes \pi_{23}^* \O_{X_\beta^i} \otimes \E_{i+1}^\vee \otimes \E_i'' [-1] \{2 \beta_i\})
\end{eqnarray*}
Since $\pi_{12}^{-1}(X_\beta^i)$ and $\pi_{23}^{-1}(X_\beta^i)$ intersect transversely $\pi_{12}^\ast \O_{X_\beta^i} \otimes \pi_{23}^\ast \O_{X_\beta^i} \cong \O_W$ where $W = \pi_{12}^{-1}(X_\beta^i) \cap \pi_{23}^{-1}(X_\beta^i)$ inside $Y_\beta \times Y_{d_i(\beta)} \times Y_\beta$. The projection $\pi_{13}$ maps $W$ isomorphically onto $V_\beta^i \subset Y_\beta \times Y_\beta$. Using the projection formula we find
$$\sP_\beta^i * (\sP_\beta^i)_R \cong \O_{V_\beta^i} \otimes \E_i \otimes (\E_{i+1}')^\vee [-1]\{2\beta_i\}.$$
Since the map $\sP_\beta^i * (\sP_\beta^i)_R \rightarrow \O_\Delta$ is the adjoint map it is not identically zero. From the distinguished triangle 
$$\O_\Delta \rightarrow \Cone(\sP_\beta^i \ast (\sP_\beta^i)_R \rightarrow \O_\Delta) \rightarrow \sP_\beta^i \ast (\sP_\beta^i)_R [1]$$
we find that $\Cone(\sP_\beta^i \ast (\sP_\beta^i)_R \rightarrow \O_\Delta)$ is a non-trivial extension of $\O_{V_\beta^i} \otimes \E_i' \otimes \E_{i+1}^\vee \{2\beta_i\}$ by $\O_\Delta$ which is supported in degree zero (here we used that $\E_i \otimes (\E_{i+1}')^\vee \cong \E_i' \otimes \E_{i+1}^\vee$ on $V_\beta^i$). 

On the other hand, $V_\beta^i \cap \Delta$ is the divisor $X_\beta^i \hookrightarrow \Delta$ so we have the standard exact sequence
$$0 \rightarrow \O_\Delta(-X_\beta^i) \rightarrow \O_{Z_\beta^i} \rightarrow \O_{V_\beta^i} \rightarrow 0.$$

If $\beta_i=\beta_{i+1}$ then by Lemma \ref{lem:facts1} we know $\O_{Y_\beta}(-X_\beta^i) \cong \E_{i+1} \otimes \E_i^\vee \{-2\beta_i\}$. Hence we have the triangle
$$\O_\Delta \rightarrow \O_{Z_\beta^i} \otimes \E_{i+1}^\vee \otimes \E_i \{2\beta_i\} \rightarrow \O_{V_\beta^i} \otimes \E_{i+1}^\vee \otimes \E_i \{2\beta_i\}$$
which means $\O_{Z_\beta^i} \otimes \E_{i+1}^\vee \otimes \E_i' \{2\beta_i\}$ is a non-trivial extension of $\O_{V_\beta^i} \otimes \E_{i+1}^\vee \otimes \E_i \{2\beta_i\}$ by $\O_{\Delta}$. Since the same is true of $\Cone(\sP_\beta^i \ast (\sP_\beta^i)_R \rightarrow \O_\Delta)$ it is enough to show that there exists a unique non-trivial such extension. 

We do this by showing that $\Ext^1(\O_{V_\beta^i} \otimes \E_{i+1}^\vee \otimes \E_i' \{2\beta_i\}, \O_\Delta) $ is one dimensional or equivalently that $\Ext^1(\O_{V_\beta^i}, \O_\Delta(-X_\beta^i))$ is one dimensional. Since $\Delta$ and $V_\beta^i$ are smooth with $\Delta \cap V_\beta^i = X_\beta^i$ a smooth divisor in $\Delta$, we know from corollary \ref{cor:abstractnonsense3} below that $\Hom^1_{Y_\beta \times Y_\beta}(\O_{V_\beta^i},\O_\Delta(-X_\beta^i)) \cong \Hom_{X_\beta^i}(\O_{X_\beta^i}, \O_{X_\beta^i})$ is one dimensional. 
\end{proof}

\begin{Lemma} \label{lem:Padj} If $\beta_i=\beta_{i+1} $ then 
$${\sP_\beta^i}_R[1] \cong \O_{X_\beta^i} \otimes \E_{i+1}^\vee \{2\beta_i\}$$ 
while ${\sP_\beta^i}_L \cong {\sP_\beta^i}_R [2]\{-2\}$. 
\end{Lemma}
\begin{proof}
The dual of $\O_{X_\beta^i}$ in $Y_{d_i(\beta)} \times Y_\beta$ is 
$$\O_{X_\beta^i}^\vee \cong \omega_{X_\beta^i} \otimes \pi_2^* \omega_{Y_{d_i(\beta)}}^\vee \otimes \pi_1^* \omega_{Y_\beta}^\vee [-\dim(X_\beta^i)].$$
Hence 
\begin{eqnarray*}
{\sP_\beta^i}_R &\cong& \O_{X_\beta^i}^\vee \otimes \E_i^\vee \otimes \pi_2^* \omega_{Y_{d_i(\beta)}} [\dim(Y_{d_i(\beta)})] \\
&\cong& \omega_{X_\beta^i} \otimes \pi_1^* \omega_{Y_\beta}^\vee \otimes \E_i^\vee [-1]
\end{eqnarray*}
If $\beta_i=\beta_{i+1}=1$ then by Lemma \ref{lem:facts1} $\omega_{X_\beta^i} \otimes \pi_1^* \omega_{Y_\beta}^\vee \cong \O_{X_\beta^i}(\E_i \otimes \E_{i+1}^\vee) \{2\}$ so we get 
$${\sP_\beta^i}_R \cong \O_{X_\beta^i} \otimes \E_{i+1}^\vee [-1] \{2\}.$$
If $\beta_i=\beta_{i+1}=m-1$ then by Lemma \ref{lem:facts1} $\omega_{X_\beta^i} \otimes \pi_1^* \omega_{Y_\beta}^\vee \cong \O_{X_\beta^i}(\E_i \otimes \E_{i+1}^\vee) \{2(m-1)\}$ so we get
$${\sP_\beta^i}_R \cong \O_{X_\beta^i} \otimes \E_{i+1}^\vee [-1] \{2(m-1)\}.$$

Finally,
\begin{eqnarray*}
{\sP_\beta^i}_L 
&\cong& \O_{X_\beta^i}^\vee \otimes \E_i^\vee \otimes \pi_1^* \omega_{Y_\beta} [\dim(Y_\beta)] \\
&\cong& \omega_{X_\beta^i} \otimes \pi_2^* \omega_{Y_{d_i(\beta)}}^\vee \otimes \E_i^\vee [1] \\
&\cong& \omega_{X_\beta^i} \otimes \pi_1^* \omega_{Y_\beta}^\vee \otimes \E_i^\vee [1] \{-2\} \\
&\cong& {\sP_\beta^i}_R [2]\{-2\}
\end{eqnarray*}
where the third equality follows since $\omega_{X_\beta^i/Y_\beta} \cong \omega_{X_\beta^i/Y_{d_i(\beta)}} \{2\}$ if $\beta_i = \beta_{i+1} $ (Lemma \ref{lem:facts1}). 
\end{proof}

\begin{Lemma}\label{lem:abstractnonsense2} Let $S,S' \subset T$ be local complete intersection (l.c.i.) subvarieties of a l.c.i. quasi-projective variety $T$. Suppose $S \cap S' \subset T$ is a l.c.i. and $\sF$ a locally free sheaf on $S$. Then
\begin{eqnarray*}
\H^l(\O_{S'}^\vee \otimes i_\ast \sF) =
\left\{
\begin{array}{ll}
0 & \mbox{if $l<c$}  \\
\O_{S \cap S'}(j^\ast \sF \otimes \det N_{S \cap S'/S}) & \mbox{if $l=c$} 
\end{array}
\right.
\end{eqnarray*}
where $i: S \hookrightarrow T$, $c$ is the codimension of the embedding $j: S \cap S' \hookrightarrow S$ and $N_{S \cap S'/S}$ is the (locally free) normal bundle of $S \cap S' \subset S$. All this also holds  $\C^\times$-equivariantly.
\end{Lemma}
\begin{proof}
This is a slight generalization of Lemma 4.7 from \cite{kh1} but the proof remains the same. 
\end{proof}

\begin{Corollary}\label{cor:abstractnonsense3} Let $S,S' \subset T$ be l.c.i. subvarieties of a l.c.i. quasi-projective variety $T$. Suppose $S \cap S' \subset S$ is codimension $c$ and carved out by the section of a rank $c$ vector bundle $\mathcal{V}$ on $S$. Then $\Hom_T(\O_{S'}[-c], \O_S \otimes \det \mathcal{V}^\vee) \cong \Hom_D(\O_D,\O_D)$ so, if $T$ is projective, there exists a non-zero map $\O_{S'}[-c] \rightarrow \O_S \otimes \det \mathcal{V}^\vee$ which is unique up to a non-zero multiple. All this also holds $\C^\times$-equivariantly.
\end{Corollary}
\begin{proof}
This is a slight generalization of corollary 4.8 from \cite{kh1} and follows easily from Lemma \ref{lem:abstractnonsense2}. 
\end{proof}

We finish the section with a result, which we will need later, related to Lemma \ref{lem:abstractnonsense2}. For simplicity, consider three smooth subvarieties $S,S',S'' \subset T$ whose pairwise intersections are also smooth. othly. Denote by $c'$ and $c$ the codimensions of $S'' \cap S' \subset S'$ and $S' \cap S \subset S$ respectively. Suppose that $S'' \cap S \subset S$ is codimension $c+c'$. An elementary check shows that this implies $S \cap S'' = S \cap S' \cap S''$ since they have the same dimension and that $S \cap S'$ and $S'' \cap S'$ intersect in the expected codimension inside $S'$. Furthermore, let $\sL$ and $\sL'$ denote line bundles which, for simplicity, we assume to be globally defined on $T$. Consider the following diagram:
\begin{equation*}
\begin{CD}
\Ext^{c'}(\O_{S''}, \O_{S'}\sL' ) @. \times @. \Ext^c(\O_{S'}\sL', \O_S \sL \sL')  @>>> \Ext^{c+c'}(\O_{S''}, \O_S \sL \sL') \\
@VV{\cong}V @. @VV{\cong}V @VV{\cong}V \\
H^0(S' \cap S'', \sL' \det N_{S' \cap S''/S'}) @. \times @. H^0(\O_{S \cap S'} , \sL \det N_{S \cap S'/S}) @>>> H^0(\O_{S \cap S''}, \sL \sL' \det N_{S \cup S''/S})
\end{CD}
\end{equation*}
where the top line is the natural composition map while the bottom line is multiplication of the restriction of sections to $S \cap S' \cap S''$. The vertical isomorphisms follow from \ref{lem:abstractnonsense2} where we use that $\det(N_{S' \cap S''/S'}) \otimes \det(N_{S \cap S'/S}) \cong \det(N_{S \cap S''/S})$. The last isomorphism is a consequence of the short exact sequence $0 \rightarrow N_{S \cap S''/S \cap S'} \rightarrow N_{S \cap S''/S} \rightarrow N_{S \cap S'/S} \rightarrow 0$ coming from $S \cap S'' \subset S \cap S' \subset S$ and the isomorphism $N_{S \cap S''/S \cap S'} \cong N_{S' \cap S''/S'}$ (which follows from the general result $N_{X/Z}|_{X \cap Y} \cong N_{X \cap Y/Z \cap Y}$ if $X$ and $Y$ intersect in the expected dimension inside $Z$). The naturality of all the maps means that: 
\begin{Lemma}\label{lem:abstractnonsense4} The diagram above commutes.
\end{Lemma}

\section{Invariance of $\Psi$} \label{se:invariance}
In the previous section we defined the functor $\Psi(T)$ by first choosing a planar diagram representation of $T$. In this section we prove that $\Psi$ does not depend on this choice. To do this we will show that $\Psi$ is invariant under the operations described in Lemma \ref{lem:relations}. 

\subsection{Invariance Under Reidemeister Move (0)}

The first relation we deal with is Reidemeister move (0) from Figure \ref{f2} which follows from the following identity.  

\begin{Proposition}\label{prop:R0} 
Assume that $ \beta $ contains a string $ (1,m-1,1) $ or $(m-1,1,m-1) $ starting the $ i$th slot. Then 
$$\sF_\beta^{i+1} * \sG_\beta^i \cong \O_\Delta \cong \sF_\beta^i * \sG_\beta^{i+1}.$$
In particular, $\F_\beta^{i+1} \circ \G_\beta^i \cong id \cong \F_\beta^i \circ \G_\beta^{i+1}$.
\end{Proposition}

\begin{proof}
To prove the first isomorphism recall that 
$$ \sF_\beta^{i+1} = \O_{X_\beta^{i+1}} \otimes (\E_{i+2})^{-\beta_{i+1}} \{(i+1) (m-1) \} \text{ and } \sG_\beta^i = \O_{X_\beta^i} \otimes (\E'_i)^{\beta_{i+1}} \{-(i-1) (m-1) \}.$$

The intersection $ W := \pi_{12}^{-1}(X_\beta^i) \cap \pi_{23}^{-1}(X_\beta^{i+1}) $ inside $ Y_{d_i(\beta)} \times Y_\beta \times Y_{d_i(\beta)} $ is transverse.  Moreover 
\begin{equation*}
\begin{aligned}
W = \{ (L_\cdot, L'_\cdot, L''_\cdot) : &L_j = L'_j \text{ for } j \le i,\ L_j = zL'_{j+2} \text{ for } j \ge i, \\ &L''_j = L'_j \text{ for } j \le i+1,\ L''_j = zL'_{j+2} \text{ for } j \ge i+1 \} .
\end{aligned}
\end{equation*}

Examining this variety we see that $ \pi_{13} $ maps $ W $ isomorphically to the diagonal in $Y_{d_i(\beta)} \times Y_{d_i(\beta)}$.  Moreover, as $ zL'_{i+2} = L'_i\{2\} $ and $ zL'_{i+1} = L'_{i-1}\{2\}$, we see that the operator $ z $ induces an isomorphism $L'_{i+2}/L'_{i+1} \cong L'_i/L'_{i-1} \{2\}$ on $W$. Taking determinants this implies $\E'_{i+2} \cong \E'_i \{2 \beta_i\}$ on $W$.

Hence 
\begin{align*}
\sF_\beta^{i+1} * \sG_\beta^i &= {\pi_{13}}_*( \pi_{12}^*(\O_{X_\beta^i} \otimes (\E'_i)^{\beta_{i+1}} \{-(i-1) (m-1) \}) \otimes \pi_{23}^*(\O_{X_\beta^{i+1}} \otimes (\E_{i+2})^{-\beta_{i+1}} \{(i+1) (m-1) \})) \\
&\cong {\pi_{13}}_*( \O_{\pi_{12}^{-1}(X_\beta^i)} \otimes \O_{\pi_{23}^{-1}(X_\beta^{i+1})} \otimes (\E'_i \otimes {\E'_{i+2}}^\vee)^{\beta_{i+1}} \{2 (m-1)\}) \\
&\cong {\pi_{13}}_*(\O_W) = \O_{\Delta}
\end{align*}
where to obtain the last line, we use that $\E'_i \cong \E'_{i+2} \{ -2 \beta_i \}$. 

The second isomorphism is actually easier to prove. We have
\begin{eqnarray*}
\sF_\beta^i * \sG_\beta^{i+1} &=& {\pi_{13}}_*( \pi_{12}^*(\O_{X_\beta^{i+1}} \otimes (\E'_{i+1})^{\beta_{i+2}} \{-i (m-1) \}) \otimes \pi_{23}^*(\O_{X_\beta^{i}} \otimes (\E'_{i+1})^{-\beta_i} \{i (m-1) \})) \\
&\cong& {\pi_{13}}_*( \O_{\pi_{12}^{-1}(X_\beta^{i+1})} \otimes \O_{\pi_{23}^{-1}(X_\beta^{i})} ) \\
&\cong& \O_{\pi_{13}(\pi_{12}^{-1}(X_\beta^{i+1}) \cap \pi_{23}^{-1}(X_\beta^{i}))} = \O_{\Delta}.
\end{eqnarray*}
\end{proof}

\subsection{Invariance Under Reidemeister Move (I)}

To prove invariance under Reidemeister move (I) we need a few preliminary results.

\begin{Lemma} \label{lem:nontransint} Let $ Y $ be a smooth projective variety and let $ V $ be a smooth subvariety in $ Y $.  Let $ V_1, V_2 $ be smooth subvarieties of $ V $ and assume that they meet in the right codimension inside $ V $.  Then, inside $D(Y)$, we have
$$\H^k (\O_{V_1} \otimes \O_{V_2}) \cong \O_{V_1 \cap V_2} \otimes \wedge^{-k} N^\vee_{V/Y}$$
where $N_{V/Y}$ is the normal bundle of $V$ inside $Y$ and $V_1 \cap V_2$ is the scheme theoretic intersection (this also holds $\C^\times$-equivariantly). 
\end{Lemma}
\begin{proof}
This is a generalization of Lemma 5.2 of \cite{kh1}. Denote by $i$ and $j$ the sequence of inclusions $V_1 \xrightarrow{i} V \xrightarrow{j} Y$. Then 
$$ \H^k(\O_{V_1} \otimes \O_{V_2}) \cong \H^k((j \circ i)_* (j \circ i)^* \O_{V_2}) \cong j_* i_* \H^k (i^* j^* j_* \O_{V_2}).$$
Now $\H^k(j^* j_* \O_{V_2}) \cong \O_{V_2} \otimes \wedge^{-k} N_{V/Y}^\vee$. On the other hand, $i^* \O_{V_2} = \O_{V_1 \cap V_2}$ since $V_1$ and $V_2$ intersect in the right dimension inside $V$. Thus $i^* \O_{V_2}$ has no lower cohomology and so 
$$\H^k(i^* j^* j_* \O_{V_2}) \cong i^* \H^k (j^* j_* \O_{V_2}) \cong i^* \O_{V_2} \otimes \wedge^{-k} N_{V/Y}^\vee \cong \O_{V_1 \cap V_2} \otimes \wedge^{-k} N_{V/Y}^\vee.$$
The result now follows. 
\end{proof}
\begin{Remark} 
You can replace $\O_{X_1}$ and $\O_{X_2}$ in Lemma \ref{lem:nontransint} by vector bundles $\sW_1$ and $\sW_2$ supported on $X_1$ and $X_2$. Then $\H^k (\sW_1 \otimes \sW_2) \cong \sW_1|_{X_1 \cap X_2} \otimes \sW_2|_{X_1 \cap X_2} \otimes \wedge^{-k} N_{V/Y}^\vee|_{X_1 \cap X_2}$.
\end{Remark}

\begin{Lemma}\label{lem:updown} In the sequence of $\C^*$-equivariant maps shown below suppose $X,Y,Y'$ are smooth, projective varieties with $\pi_1(Y')=0$. Further assume $q$ is a smooth map and $i$ an inclusion such that $N_{X/Y} \cong \Omega_{X/Y'}\{s\}$. Denote by $\sP \in D(Y' \times Y)$ the kernel $\O_X \otimes \sL$ where $\sL$ is some tensor of line bundles pulled back from $Y'$ and $Y$. If a fibre $F$ of $q$ (and hence all fibres) has a cellular decomposition then
$$\H^n(\sP_R * \sP) \cong \O_{\Delta} \otimes_{\C} H^n(F,\C) \{sn/2\} \text{ and } \H^n(\sP_L * \sP) \cong \O_{\Delta} \otimes_{\C} H^{-n}(F,\C)^\vee \{sn/2\}.$$
\end{Lemma}
\begin{proof}
We write $\sL := \pi_1^* \sL_1 \otimes \pi_2^* \sL_2$ where $\sL_1$ and $\sL_2$ are line bundles on $Y'$ and $Y$. We have
\begin{eqnarray*}
\sP_R &=& \O_X^\vee \otimes \sL^\vee \otimes \pi_2^* \omega_{Y'}[\dim(Y')] \\
&\cong& \omega_X \otimes \sL^\vee \otimes \omega_{Y \times Y'}^\vee [-\dim(X)] \otimes \pi_2^* \omega_{Y'}[\dim(Y')] \\
&\cong& \omega_X \otimes \sL^\vee \otimes \pi_1^* \omega_Y^\vee [-r] \cong \det N_{X/Y} \otimes \sL^\vee [-r] 
\end{eqnarray*}  
where $r$ is the codimension of $i: X \subset Y$ or, equivalently, the dimension of the fibre of $q$. Hence 
$$\sP_R * \sP \cong \pi_{13*}(\pi_{12}^* (\O_X) \otimes (\pi_1^* \sL_1 \otimes \pi_2^* \sL_2) \otimes \pi_{23}^*(\det N_{X/Y} [-r]) \otimes (\pi_2^* \sL_2^\vee \otimes \pi_3^* \sL_1^\vee))$$ 
using the natural projection maps on $Y' \times Y \times Y'$. Now $V_1 = \pi_{12}^{-1}(X)$ and $V_2 = \pi_{23}^{-1}(X)$ intersect transversely inside $V = Y' \times X \times Y'$. Applying Lemma \ref{lem:nontransint} we get a spectral sequence 
$$E_2^{p,q} = \H^p \left( \pi_{13*} i_* (\det N_{X/Y} \otimes \wedge^{r-q} N^\vee_{V/Y' \times Y \times Y'}|_X \otimes \pi_1^* \sL_1 \otimes \pi_3^* \sL_1^\vee) \right) \Longrightarrow \H^{p+q}(\sP_R * \sP).$$ 
On the other hand,
\begin{eqnarray*}
\H^p \left( \pi_{13*} i_* (\det N_{X/Y} \otimes \wedge^{r-q} N^\vee_{V/Y' \times Y \times Y'}|_X) \right) 
&\cong& \H^p \left( \pi_{13*} i_* \wedge^q N_{X/Y} \right) \\
&\cong& \H^p \left( \pi_{13*} i_* \Omega^{q}_{X/Y'} \{sq\} \right) \\ 
\end{eqnarray*}
where $i$ denotes the inclusion $V_1 \cap V_2 = X \rightarrow Y' \times Y \times Y'$. To obtain the first isomorphism we used that $N_{V/Y' \times Y \times Y'}|_X \cong N_{X/Y}$.

Now $\pi_{13} \circ i = \Delta \circ q$ where $\Delta: Y' \rightarrow Y' \times Y'$ is the diagonal embedding. On the other hand, since each fibre $F$ of $q$ has a cellular decomposition we have $H^{k,l}(F,\Z) = \delta_{k,l} \Z^{\oplus b_{k,l}}$ where $b_{k,l}$ is the corresponding Hodge number. Thus $\H^l(q_* \wedge^k \Omega_{X/Y'}) \cong \delta_{k,l} R^{k+l} q_* \C_X \otimes_{\C} \O_{Y'}$. Since $\pi_1(Y')=0$ we conclude $R^{k+l} q_\ast \C_X$ is constant and hence $q_* \Omega^k_{X/Y'} = \O_{Y'} \otimes_{\C} H^{2k}(F,\C)[-k]$. Consequently
\begin{eqnarray*}
\H^p \left( \pi_{13*} i_* \Omega^{q}_{X/Y'} \{sq\} \right) \otimes \pi_1^* \sL_1 \otimes \pi_2^* \sL_1^\vee 
&\cong& \H^p \left( \Delta_* q_* \Omega^{q}_{X/Y'} \{sq\} \right) \otimes \pi_1^* \sL_1 \otimes \pi_2^* \sL_1^\vee \\
&\cong& \O_{\Delta} \otimes_{\C} H^{p+q}(F,\C) \{sq\} \delta_{p,q}
\end{eqnarray*}
where we use that $\O_\Delta \otimes \pi_1^* \sL_1 \otimes \pi_2^* \sL_1^\vee \cong \O_\Delta$ to obtain the second line. 

So $E_2^{p,q}$ is supported on the line $p=q$ which means that all the higher differentials vanish. Thus 
$$\H^n(\sP_R * \sP) \cong \O_\Delta \otimes_{\C} H^{n}(F,C) \{sn/2\}.$$

Finally, 
\begin{eqnarray*}
\sP_L &\cong& \O_X^\vee \otimes \sL^\vee \otimes \pi_2^* \omega_Y [\dim(Y)] \\
&\cong& \omega_X \otimes \sL^\vee \otimes \pi_1^* \omega_{Y'}^\vee [r] \\
&\cong& \det \Omega_{X/Y'}[r] \otimes \sL^\vee \cong \det N_{X/Y} [r] \{-sr\} \otimes \sL^\vee \cong \sP_R [2r] \{-sr\}
\end{eqnarray*}
so that the desired expression for $ \sP_L * \sP $ follows from that of $ \sP_R * \sP $ by Poincar\'e duality on $H^*(F,\C)$.
\end{proof}

Before proving the invariant under the first Reidemeister move, we will establish the following Corollary whose proof actually relies on invariance results (pitchfork and Reidemister II) proved later in this section.  There is no problem however, since this Corollary is not used until section 6.

\begin{Corollary}\label{cor:circle} If $\beta_i + \beta_{i+1} = m$ then $\sF_\beta^i * \sG_\beta^i \cong \O_\Delta \otimes_\C V$ where 
$$V = H^\star(\p) = \C[-m+1]\{m-1\} \oplus \dots \oplus \C[m-1]\{-m+1\}.$$  
In particular, if $T$ is a tangle then $\Psi(T \cup O)(\cdot) \cong \Psi(T)(\cdot) \otimes_\C V$.
\end{Corollary}
\begin{proof}
We would like to use Lemma \ref{lem:updown} with $Y = Y_\beta$, $X = X^i_\beta$, $Y' = Y_{d_i(\beta)}$ and $s=2$. Since $\E_i$ on $X$ is the pullback of $\E_i$ on $Y$ we have $\sF_\beta^i * \sG_\beta^i \cong (\sG_\beta^i)_R * \sG_\beta^i [m-1]\{-m+1\} \cong (\O_X)_R * \O_X [m-1]\{-m+1\}$.

Since $\beta_i$ equals $1$ or $m-1$ the map $q: X \rightarrow Y'$ is a $\p$ bundle. In particular, the fibres have a cellular decomposition and since $Y'$ is an iterated Grassmannian bundle we get $\pi_1(Y')=0$. Furthermore, $N_{X/Y} \cong \Omega_{X/Y'}\{2\}$ by Lemma \ref{lem:facts1} so the hypothesis of Lemma \ref{lem:updown} are satisfied. Unfortunately, this only tells us the cohomology of $\sF_\beta^i * \sG_\beta^i$ but not that it breaks up as a direct sum. 

To conclude that $\sF_\beta^i * \sG_\beta^i$ is formal we need the following stronger version of Lemma \ref{lem:updown}. It states that $\sP_R * \sP$ is formal under the additional hypothesis that $Y$ contains a Zariski open neighbourhood $U$ of $X$ which retracts onto $X$. To see this applies to our situation move the circle past the rightmost strand using pitchfork and Reidemeister II moves. Then take $U=Y_\beta$ with the retraction $U \rightarrow X^{n-1}_\beta$ given by 
$$(L_1 \subset \dots \subset L_{n-1} \subset L_n) \mapsto (L_1 \subset \dots \subset L_{n-1} \subset z^{-1}L_{n-2}).$$
\end{proof}

Next we show invariance under Reidemeister move (I).

\begin{Theorem}\label{th:RI}
Assume that $ \beta $ has a string $ (1, 1, m-1) $ or $ (m-1, m-1, 1) $ starting in the $i$th slot.  Then
$$ \sF_\beta^{i+1} * \sT_\beta^i(2) * \sG_\beta^{i+1} \cong \O_\Delta \cong \sF_\beta^{i+1} * \sT_\beta^i(1) * \sG_\beta^{i+1} $$
\end{Theorem}
\begin{proof}
We will do the case where  $\beta = (\dots, 1, 1, m-1, \dots)$.  The other case is similar.

By Theorem \ref{th:kerneltwist} we have $\sT_\beta^i(2) \cong \Cone( \sP * \sP_R \rightarrow \O_\Delta )[-m+1]\{m-1\}$ where $\sP = \sP_\beta^i \in D(Y_{d_i(\beta)} \times Y_\beta)$. Thus
$$\sF_\beta^{i+1} * \sT_\beta^i(2) * \sG_\beta^{i+1} \cong \Cone (\sF_\beta^{i+1} * \sP * \sP_R * \sG_\beta^{i+1} \rightarrow \sF_\beta^{i+1} * \sG_\beta^{i+1})[-m+1]\{m-1\}.$$
Applying Lemma \ref{lem:updown} as in the first part of the proof of Corollary \ref{cor:circle} we see that $$\H^n(\sF_\beta^{i+1} * \sG_\beta^{i+1}) \cong \O_\Delta\{n\} $$ for $n = -m+1, \dots, m-1 $ and 0 otherwise. 

On the other hand, from Lemma \ref{lem:Padj} we know $\sP_R \cong \O_{X_\beta^i} \otimes \E_{i+1}^\vee [-1] \{2\}$. Hence $\sQ = \sP_R * \sG_\beta^{i+1}$ acts by 
$$(\cdot) \mapsto q_{2*} (i_2^* i_{1*} (q_1^* (\cdot) \otimes (\E_{i+1})^{\beta_{i+2}}\{-i (m-1)\}) \otimes \E_{i+1}^\vee [-1] \{2\})$$
where the maps are depicted in the diagram below. 
\begin{equation*}
\begin{CD}
X_\beta^i @>i_2>> Y_\beta @<i_1<< X_\beta^{i+1} \\
@Vq_2VV @. @Vq_1VV \\ 
Y_{d_i(\beta)} @. @. Y_{d_{i+1}(\beta)}
\end{CD} 
\end{equation*}
Now 
$$
W = X_\beta^i \cap X_\beta^{i+1} = \{ (L_\cdot): zL_{i+1} \subset L_{i-1} \text{ and } zL_{i+2} = L_i \}. $$
is a transverse intersection inside $Y_\beta$. In fact, $W \cong X_{\beta'}^i$ where $\beta' = d_i(\beta)$ has a $(2,m-1)$ string starting in the $i$th slot. Hence
$$\Phi_{\sQ}(\cdot) \cong i_* (q^* (\cdot) \otimes \sL) [-1] \{-i (m-1) +2\}$$
where the maps are depicted in the diagram below and $\sL$ is some tensor of line bundles pulled back from $Y_{\beta'}$ and $Y_{d_i(\beta')}$ which we will subsequently ignore.
\begin{equation*}\begin{CD}X_{\beta'}^i @>i>> Y_{\beta'} \\@VqVV \\ Y_{d_i(\beta')} \end{CD} \end{equation*}

Notice $q$ is a $\mathbb{P}^{m-2}$ bundle while $i$ is an inclusion of codimension $m-2$. By Lemma \ref{lem:facts1} we know $N_{X_{\beta'}^i/Y_{\beta'}} \cong \Omega_{X_{\beta'}^i/d_i(\beta')} \{2\}$. As before, $Y_{d_i(\beta')}$ has trivial fundamental group so applying Lemma \ref{lem:updown} we conclude that
$\H^n(\sQ_L * \sQ) \cong \O_\Delta \{n\},$
for $n = -2m+4,-2m+6 \dots,2,0 $ and 0 otherwise.

Since ${\sG_\beta^{i+1}}_L \cong \sF_\beta^{i+1}[m-1]\{-m+1\}$ we find that
$$\H^n(\sF_\beta^{i+1} * \sP * \sP_R * \sG_\beta^{i+1}) \cong \O_\Delta\{n\},$$
for $ n = -m+3, \dots, m-1 $ and 0 otherwise.

Let $ \A = \sF_\beta^{i+1} * \sT_\beta^i(2) * \sG_\beta^{i+1}[m-1]\{1-m\} $ and consider the long exact sequence in $ \H $ coming from the distinguished triangle 
\begin{equation*}
\sF_\beta^{i+1} * \sP * \sP_R * \sG_\beta^{i+1} \rightarrow \sF_\beta^{i+1} * \sG_\beta^{i+1} \rightarrow \A
\end{equation*}

From our calculations of $ \H^n(\sF_\beta^{i+1} * \sP * \sP_R * \sG_\beta^{i+1}) $ and $\H^n(\sF_\beta^{i+1} * \sG_\beta^{i+1}) $, we deduce that for each $ n = -m+3, \dots, m-1 $ either we have an isomorphism $  \H^n(\sF_\beta^{i+1} * \sP * \sP_R * \sG_\beta^{i+1}) \rightarrow \H^n(\sF_\beta^{i+1} * \sG_\beta^{i+1})$ in which case $ \H^n(\A) = 0 = \H^{n-1}(\A) $ or this map is zero in which case $ \H^n(\A) = \O_\Delta = \H^{n-1}(\A) $ (here we are using the fact that $\Hom(\O_\Delta, \O_\Delta) = 0 $).  To complete the proof of the theorem, it suffices to check that this second possibility never occurs (i.e. that the map above is an isomorphism for all $n=-m+3, \dots, m-1$).

To do this it suffices to show that 
\begin{equation*} 
\Cone \left( \Phi_{\sF_\beta^{i+1}} \circ \Phi_{\sP} \circ \Phi_{\sP_R} \circ \Phi_{\sG_\beta^{i+1}} (\O_{Y_{d_{i+1}(\beta)}}) \rightarrow \Phi_{\sF_\beta^{i+1}} \circ  \Phi_{\sG_\beta^{i+1}} (\O_{Y_{d_{i+1}(\beta)}}) \right) \cong \O_{Y_{d_{i+1}(\beta)}} [m-1] \{-m+1\}.
\end{equation*}

Now $\Phi_{\sG_\beta^{i+1}} (\O_{Y_{d_{i+1}(\beta)}}) \cong \O_{X_\beta^{i+1}} \otimes (\E_{i+1})^{\beta_{i+2}}\{-i (m-1)\} \in D(Y_\beta)$ while 
\begin{align*}
\Phi_{\sP} \circ \Phi_{\sP_R} \circ \Phi_{\sG_\beta^{i+1}} (\O_{Y_{d_{i+1}(\beta)}}) 
&\cong i_{2*} \Big( q_2^* q_{2*}(i_2^*(\O_{X_\beta^{i+1}} \otimes (\E_{i+1})^{\beta_{i+2}} \{-i (m-1)\}) \E_i \E_{i+1}^\vee [-1] \{2\}) \Big) \\
&\cong \O_{q_2^{-1}(q_2(W))} \otimes (\E_{i+1})^{\beta_{i+2}-1} \E_i [-1] \{-i (m-1) + 2\}
\end{align*}
So the map $ \Phi_{\sP} \circ \Phi_{\sP_R} \circ \Phi_{\sG_\beta^{i+1}} (\O_{Y_{d_{i+1}(\beta)}}) \rightarrow \Phi_{\sG_\beta^{i+1}} (\O_{Y_{d_{i+1}(\beta)}})$ is the same as a map 
$$\alpha: \O_{W'} \otimes (\E_{i+1})^{\beta_{i+2}} \E_i  \E_{i+1}^\vee [-1] \{-i (m-1) + 2\} \rightarrow \O_{X_\beta^{i+1}} \otimes (\E_{i+1})^{\beta_{i+2}} \{-i (m-1)\}$$
 in $D(Y_\beta)$ where 
$$W' = q_2^{-1}(q_2(W)) = 
\{ (L_\cdot): zL_{i+1} \subset L_{i-1} \text{ and } L_{i-1} \subset zL_{i+2} \}.
$$
We claim $\alpha$ is (up to a non-zero multiple) the unique such non-zero map. To see $\alpha$ is non-zero notice that the composition 
$${\Phi_{\sG_\beta^{i+1}}}_L \circ \Phi_{\sP} \circ \Phi_{\sP_R} \circ \Phi_{\sG_\beta^{i+1}} (\O_{Y_{d_{i+1}(\beta)}}) \rightarrow {\Phi_{\sG_\beta^{i+1}}}_L \circ  \Phi_{\sG_\beta^{i+1}} (\O_{Y_{d_{i+1}(\beta)}}) \rightarrow \O_{Y_{d_{i+1}(\beta)}}$$
is the adjunction map for $ \Phi_{\sP_R} \circ \Phi_{\sG_\beta^{i+1}} $ and hence is non-zero. On the other hand,  $W'$ and $X_\beta^{i+1}$ in a divisor $D = X_\beta^{i+1} \cap X_\beta^i$ in each of them. By Corollary \ref{cor:abstractnonsense3} $\Hom(\O_{W'}(D), \O_{X_\beta^{i+1}}[1]) \cong \C$ while the cone of the corresponding non-zero map $\O_{W'}[-1] \rightarrow \O_{X_\beta^{i+1}}(-D)$ is isomorphic to $\O_{W' \cup X_\beta^{i+1}}$. But, by Lemma \ref{lem:facts1} 
$$\O_{X_\beta^{i+1}}(-D) \cong \O_{X_\beta^{i+1}} \otimes \O_{Y_\beta}(-X_\beta^i) \cong \O_{X_\beta^{i+1}} \otimes \E_i^\vee \E_{i+1} \{-2\}$$ 
so we find that 
\begin{eqnarray*}
\Cone(\alpha) &\cong& \O_{W' \cup X_\beta^{i+1}} \otimes (\E_{i+1})^{\beta_{i+2}} \E_i \E_{i+1}^\vee \{-i (m-1) + 2\} \\
&\cong& \O_{W' \cup X_\beta^{i+1}} \otimes \E_i (\E_{i+1})^{\beta_{i+2}-1} \{-i (m-1) + 2\}.
\end{eqnarray*}

It remains to show 
\begin{equation} \label{eq:PhiF}
\Phi_{\sF_\beta^{i+1}}(\O_{W' \cup X_\beta^{i+1}} \otimes \E_i \otimes (\E_{i+1})^{\beta_{i+2}-1}) \cong \O_{Y_{d_{i+1}(\beta)}} [m-1] \{(i-1)(m-1) + 2\}.
\end{equation}
To do this we give an explicit locally free resolution of $\O_{W' \cup X_\beta^{i+1}}$. First notice that 
$$W' \cup X_\beta^{i+1} = 
\{ (L_\cdot): L_{i-1} \subset zL_{i+2} \}. $$ 
This is because if $L_{i-1} \subset zL_{i+2}$ then either $zL_{i+2} = L_i$ (belongs to $X_\beta^{i+1}$) or $\mbox{span}(zL_{i+2},L_i) = L_{i+1}$ which means $zL_{i+1} \subset L_{i-1}$ (belongs to $W'$). The map
$$z: z^{-1}L_{i-1}/L_i \rightarrow L_{i+1}/zL_{i+2} \{2\}$$
on $Y_\beta$ vanishes precisely when $L_{i-1} \subset zL_{i+2}$. Denoting
$$\sV =  (z^{-1}L_{i-1}/L_i)^\vee \otimes L_{i+1}/zL_{i+2} \{2\}. $$ 
this map gives a section of the $m-1$ dimensional vector bundle $\sV$ which vanishes precisely along the $m-1$ codimensional subscheme $W' \cup X_\beta^{i+1}$. This gives the Koszul resolution 
$$0 \leftarrow \O_{W' \cup X_\beta^{i+1}} \leftarrow \O_{Y_\beta} \leftarrow \sV^\vee \leftarrow \wedge^2 \sV^\vee \leftarrow \dots \leftarrow \wedge^{m-1} \sV^\vee \leftarrow 0.$$

Consider now $\sR_k := \Phi_{\sF_\beta^{i+1}}(\wedge^k \sV^\vee \otimes \E_i \otimes (\E_{i+1})^{\beta_{i+2}-1})$. We have 
\begin{align*}
\sR_k &\cong q_{1*} ( i_1^*(\wedge^k \sV^\vee  \otimes \E_i (\E_{i+1})^{m-2}) \E_{i+2}^\vee \{(i+1) (m-1)\} ) \\
&\cong q_{1*} ( (i_1^* \wedge^k z^{-1}L_{i-1}/L_i)  \otimes (L_{i+1}/zL_{i+2})^{-k} \{-2k\} \E_i(\E_{i+1})^{m-2} \E_{i+2}^\vee \{(i+1) (m-1)\} ) \\
&\cong q_{1*} ( (q_1^* \wedge^k z^{-1}L_{i-1}/L_i) \otimes (\E_{i+1})^{-k} q_1^* \E_i  (\E_{i+1})^{m-2} \E_{i+1} ) \{-2b_\beta^i-2m + (i+1) (m-1) - 2k\}  \\
&\cong \wedge^k z^{-1}L_{i-1}/L_i \otimes \E_i \otimes q_{1*} (\E_{i+1})^{m-1-k} \{(i+1) (m-1) - 2k - 2b_\beta^i - 2m\}
\end{align*}
where to get the third line we use that 
$$\E_{i+1} \otimes \E_{i+2} \cong \det L_{i+2}/L_i \cong \det z^{-1}L_i/L_i \cong \O_{X_\beta^{i+1}} \{2 b_\beta^i + 2m\}$$ 
by Lemma \ref{lem:facts0}. Notice the restriction of $(\E_{i+1})^{m-1-k}$ to a fibre of $q_1$ is $\O_\p(-m+1+k)$ whose cohomology vanishes for $k=0,1, \dots, m-2$.

Hence, $\sR_k = 0$ unless $k=m-1$ in which case we get 
$$\sR_{m-1} \cong 
\wedge^{m-1} (z^{-1}L_{i-1}/L_i) \otimes \E_i \otimes q_{1*} \O_{X_\beta^{i+1}} \{(i+1)(m-1) - 2b_\beta^i - 4m+2\}.$$
Now $q_{1*} \O_{X_\beta^{i+1}} \cong \O_{Y_{d_{i+1}(\beta)}}$ and
$$\wedge^{m-1} (z^{-1}L_{i-1}/L_i) \cong \det (z^{-1}L_{i-1}/L_i) \cong \E_i^\vee \{2b_\beta^{i-1}+2m\}$$
by Lemma \ref{lem:facts0}.  Hence
\begin{equation*}
\sR_{m-1} \cong \O_{Y_{d_{i+1}(\beta)}} \{(i-1)(m-1)+2\}
\end{equation*}
which establishes \ref{eq:PhiF} and so we are done.

To show $\sF_\beta^{i+1} * \sT_\beta^i(1) * \sG_\beta^{i+1} \cong \O_\Delta$ we just take the left adjoints of both sides of $\sF_\beta^{i+1} * \sT_\beta^i(2) * \sG_\beta^{i+1} \cong \O_\Delta$ and use Lemma \ref{th:Gadj} together with the fact ${\sT_\beta^i(2)}_L \cong \sT_\beta^i(1)$. 
\end{proof}

\subsection{Invariance Under Reidemeister Moves (II)}

In this section we will prove the following result.

\begin{Theorem}\label{th:R2}
The map $ \Psi $ is invariant under Reidemeister II moves, i.e.
$$\T_\beta^i(1) \circ \T_{s_i(\beta)}^i(2) \cong Id \cong \T_\beta^i(2) \circ \T_{s_i(\beta)}^i(1)$$
\end{Theorem}

By Lemma \ref{th:twistadjoint} it is enough to show that $\T_\beta^i(2)$ is an equivalence. There are two cases to consider. 

\subsubsection{Crossing like strands} \label{se:twists}
Assume $ \beta_i = \beta_{i+1} $.

In Theorem \ref{th:kerneltwist}, we prove that $\T_\beta^i(2)$ is isomorphic to the twist (up to a shift) in the functor $\Phi_{\sP_\beta^i}$.  This allows us to make us of the machinery of spherical functors.

Let $k = \dim(X) - \dim(Y)$ and suppose $\sP_R \cong \sP_L[k]\{l\}$ for some $l$. Suppose that the sequence of adjoint maps 
$$\O_\Delta \rightarrow \sP_R \ast \sP \rightarrow \O_\Delta[k]\{l\}$$
inside $D(X \times X)$ forms a distinguished triangle. Finally, assume that 
\begin{eqnarray*}
\Hom(\sP,\sP[i]\{j\}) \cong \left\{
\begin{array}{ll}
\C & \mbox{if } i=0 \text{ and } j=0\\
0  & \mbox{if } i=k,k+1 \text{ and } j=l
\end{array}
\right.\end{eqnarray*} 
If $\sP \in D(X \times Y)$ satisfies all these conditions, then $ \Phi_{\sP} $ is called a \textbf{spherical functor} and $\Phi_{\sT_\sP}: D(Y) \rightarrow D(Y)$ is a \textbf{spherical twist}. We care about spherical functors because of the following result which is originally due to Horja \cite{Ho} and Rouquier \cite{Rou}.

\begin{Theorem} \label{th:sphericaltwist} 
A spherical twist induces an equivalence of derived categories. 
\end{Theorem}
\begin{proof}
See Theorem 3.4 of \cite{kh1}.
\end{proof}

Hence in order to prove that $\T_\beta^i(2)$ is an equivalence, we just need to prove the following statement.

\begin{Lemma} \label{th:Pissphere}
$\P_\beta^i : D(Y_{d_i(\beta)}) \rightarrow D(Y_\beta)$ is a spherical functor.
\end{Lemma}
\begin{proof}
By Lemma \ref{lem:Padj} we know ${\sP_\beta^i}_R \cong {\sP_\beta^i}_L [k]\{l\}$ where $k = \dim(Y_{d_i(\beta)}) - \dim(Y_\beta) = -2$ and $l = 2$. It is also easy to check that $\Hom(\sP_\beta^i, \sP_\beta^i) = \Hom(\O_{X_\beta^i}, \O_{X_\beta^i}) \cong \C$ and $\Hom(\sP_\beta^i, \sP_\beta^i [k]\{l\} = \Hom(\sP_\beta^i, \sP_\beta^i [k+1]\{l\}) = 0$ since $k+1 < 0$. 

It remains to show that the standard sequence of adjunctions $\O_\Delta \rightarrow {\sP_\beta^i}_R * \sP_\beta^i \rightarrow \O_\Delta [-2]\{2\}$ is a distinguished triangle. To see this note that 
\begin{eqnarray*}
\bigoplus_n \H^n({\sP_\beta^i}_R * \sP_\beta^i)
&\cong& \bigoplus_n \H^n( \pi_{13*}(\pi_{12}^* (\O_{X_\beta^i} \otimes \E_i') \otimes \pi_{23}^*({\O_{X_\beta^i}}_R \otimes \E_i^\vee)) ) \\
&\cong& \bigoplus_n \H^n( \pi_{13*}(\pi_{12}^* (\O_{X_\beta^i}) \otimes \pi_{23}^*({\O_{X_\beta^i}}_R) \otimes \E_i' \otimes (\E_i')^\vee) ) \\
&\cong& \bigoplus_n \H^n( {\O_{X_\beta^i}}_R * \O_{X_\beta^i} ) \\
&\cong& \O_{\Delta} \oplus \O_{\Delta}[-2]\{2\}
\end{eqnarray*}
where the last equality follows from Lemma \ref{lem:updown} where the fibre $F$ is a $\mathbb{P}^1$. Thus there exists a distinguished triangle $\O_\Delta \rightarrow {\sP_\beta^i}_R * \sP_\beta^i \rightarrow \O_\Delta[-2]\{2\}$. To see the maps in this triangle are the same as those in the sequence above it suffices to notice that all these maps are non-zero (since in the original sequence they were adjunction maps) and 
$$\Hom(\O_\Delta, {\sP_\beta^i}_R * \sP_\beta^i) \cong \Hom(\sP_\beta^i, \sP_\beta^i) \cong \C$$
$$\Hom({\sP_\beta^i}_R * \sP_\beta^i, \O_\Delta [-2]\{2\}) \cong \Hom({\sP_\beta^i}_L * \sP_\beta^i, \O_\Delta) \cong \Hom(\sP_\beta^i, \sP_\beta^i) \cong \C$$ 
because ${\sP_\beta^i}_R = {\sP_\beta^i}_L [-2]\{2\}$ (Lemma \ref{lem:Padj}). 
\end{proof}

\subsubsection{Crossing unlike strands} \label{se:RIIb}
The second case is $\beta_i \ne \beta_{i+1} $. 

It is enough to show that the kernel $\O_{Z_\beta^i}$ induces an equivalence. For simplicity, let $Y = Y_\beta$, $Y' = Y_{s_i(\beta)}$ and $\sP := \O_{Z_\beta^i} \in D(Y \times Y')$. We use a technique introduced in \cite{BKR}. Consider the loci
$$N_k = \overline{\{ (p_1,p_2) \in Y \times Y \smallsetminus \Delta: \Ext^k_{Y'}(\Phi_{\sP}(\O_{p_1}), \Phi_{\sP}(\O_{p_2})) \ne 0 \}} \subset Y \times Y.$$
These loci are of interest because 
\begin{eqnarray*}
\Ext^k_{Y'}(\Phi_{\sP}(\O_{p_1}), \Phi_{\sP}(\O_{p_2})) 
&\cong& \Ext^k_{Y}(\Phi_{\sP_L * \sP} (\O_{p_1}), \O_{p_2}) \\
&\cong& \Ext^k_{Y}(i_{p_1}^*(\sP_L * \sP), \O_{p_2}) \\
&\cong& \Ext^k_{Y \times Y}(\sP_L * \sP, i_{p_1*} \O_{p_2}) \\
&\cong& \Ext^k_{Y \times Y}(\sP_L * \sP, \O_{(p_1,p_2)}) \\
&\cong& \Ext^k_{(p_1,p_2)}(i_{(p_1,p_2)}^* \sP_L * \sP, \O_{(p_1,p_2)}) \\
&\cong& \H^k(i_{(p_1,p_2)}^* \sP_L * \sP)^\vee
\end{eqnarray*}
where $i_{p_1}: Y \rightarrow Y \times Y$ is the inclusion $p \mapsto (p_1,p)$ and $i_{(p_1,p_2)}$ is the inclusion of a point $p \mapsto (p_1,p_2) \in Y \times Y$. So these $N_k$ measure the support of $\sP_L * \sP$ off the diagonal. 

\begin{Theorem}
$\sP_L * \sP$ is supported on the diagonal $\Delta \subset Y \times Y$.
\end{Theorem}

\begin{proof}
It suffices to show that $ N_k $ is empty for all $ k$.

Suppose $N_k$ is non-empty and let $C$ be a component of $\mbox{supp}(\sP_L * \sP)$. The above calculation implies $C$ is a component of $N_k$ for some $k$. Consider a small open neighbourhood $W$ of a general point of $C$. Given $E \in D(W)$ it follows from the Intersection Theorem (see \cite{BKR} section 5) that $\mbox{codim}(\mbox{supp}(E)) \le n$ where $n$ (called the torsion amplitude $\mbox{Tor-amp}(E)$ in \cite{Lo}) is the length of the smallest interval in $\Z$ containing the set
$$\{j \in \Z: \exists w \in W \text{ such that } H^j i_w^*(E) \ne 0 \}.$$
But by Lemma \ref{th:N_kineq} below $\mbox{codim}_W(N_k) \ge \dim(Y) + 1$ so we get $\mbox{Tor-amp}_W(N_k) \ge \dim(Y)+1$. On the other hand, by Proposition \ref{prop:vanext} below we know $H^n(i_{(p_1,p_2)}^* \sP_L * \sP) = \Ext^n_{Y'}(\Phi_{\sP}(\O_{p_1}), \Phi_{\sP}(\O_{p_2})) = 0$ if $n < 0$ or $n > \dim(Y)$. This implies $\mbox{Tor-amp}_W(N_k) \le \dim(Y)$ (contradiction). 
\end{proof}

\begin{Lemma}\label{th:N_kineq} $\mbox{codim}(N_k) \ge \dim(Y) + 1$.
\end{Lemma}
\begin{proof}
Let $N = \overline{\{ (p_1,p_2) \in Y \times Y \smallsetminus \Delta: \mbox{supp} (\Phi_{\sP})(\O_{p_1}) \cap \mbox{supp} (\Phi_{\sP})(\O_{p_2}) \ne \emptyset \}} \subset Y \times Y.$ Since $N_k \subset N$ it is enough to show $\mbox{codim}(N) \ge \dim(Y) + 1$. 

If $\beta_i \ne \beta_{i+1} \in \{1,m-1\}$ the map $\pi_1: Z_\beta^i \rightarrow Y_\beta$ is one-to-one away from $X_\beta^i \subset Y_\beta = Y$ and $\p:1$ over $X_\beta^i$. If $p = (L_1, \dots, L_n) \not\in X_\beta^i$ then $\mbox{supp} (\Phi_{\sP})(\O_p)$ is a point while if $p \in X_\beta$ then $\Phi_{\sP}(\O_p)$ is supported on a $\p$. In both cases the support parametrizes all possible $L_{i-1} \subset L_i' \subset L_{i+1}$ such that $zL_{i+1} \subset L_i'$ and $zL_i' \subset L_{i-1}$ where $\dim(L_i'/L_{i-1}) = \beta_{i+1}$. 

Suppose $p_1 = (L_1, \dots, L_n)$ and $p_2 = (M_1, \dots, M_n)$ where $p_1 \ne p_2$. From the description above it is clear $\mbox{supp} (\Phi_{\sP})(\O_{p_1}) \cap \mbox{supp} (\Phi_{\sP})(\O_{p_2}) = \emptyset$ unless $L_j = M_j$ for $j \ne i$ and $p_1, p_2 \in X_\beta^i$. The natural projection map $\pi_1: N \rightarrow Y$ maps onto $X_\beta^i$ with fibre $\p$ corresponding to the $Gr(\beta_i, L_{i+1}/L_{i-1})$ locus of possible $L_{i-1} \subset M_i \subset L_{i+1}$. Thus $\dim(N) \le \dim(X_\beta^i) + m-1 = \dim(Y) - 1$ and so $\mbox{codim}(N) \ge \dim(Y) + 1$. 
\end{proof}

\begin{Proposition}\label{prop:vanext} $\Ext^n_{Y'}(\Phi_{\sP}(\O_{p_1}), \Phi_{\sP}(\O_{p_2})) = 0$ if $n<0$ or $n>\dim(Y)$.
\end{Proposition}
\begin{proof}
Suppose $\beta_i=1$ and $\beta_{i+1}=m-1$. Let $p_1 = (L_1, \dots, L_n)$ and $p_2 = (M_1, \dots, M_n)$. As noted in the proof of Lemma \ref{th:N_kineq} the supports of $\Phi_{\sP}(\O_{p_1})$ and $\Phi_{\sP}(\O_{p_2})$ are disjoint unless $L_j = M_j$ for $j \ne i$ so from now on assume this is the case. If $p_1 \not\in X_\beta^i$ then $p_1=p_2=p$ and $\Phi_{\sP}(\O_p)$ is the structure sheaf a point so there are no negatives exts (and similarly if $p_2 \not\in X_\beta^i$). We need to study what happens if $p_1,p_2 \in X_\beta^i$. 

So suppose $p = (L_1, \dots, L_n) \in X_\beta^i$. Denote by $\pi_1$ and $\pi_2$ the standard projections from $Z_\beta^i$ to $Y=Y_\beta$ and $Y'=Y_{s_i(\beta)}$. We need to compute $\Phi_{\sP}(\O_{p}) = \pi_{2*} \pi_1^* \O_{p}$. In Lemma \ref{lem:facts3} we saw that $Z_\beta^i$ embeds in $\pi: Gr_\beta^i := Gr(\beta_{i+1}, L_{i+1}/L_{i-1}) \rightarrow Y$. If we denote this embedding by $j$ then $\pi_1^* \O_{p} \cong j^* \pi^* \O_{p}$. Notice that $\pi^* \O_{p}$ is the structure sheaf of the $\pi$-fibre over $p$ which is contained in $j(Z_\beta^i)$ because $p \in X_\beta^i$. This means that 
$\H^{-k}(j^* \pi^* \O_{p}) \cong \wedge^k (N_{Z_\beta^i/Gr_\beta^i}|_{\pi^{-1}(p)})^\vee.$

On the other hand, as we saw in Lemma \ref{lem:facts3}, $Z_\beta^i \subset W = Gr(\beta^i, L_{i+1}/L_{i-1})$ is carved out by a section of the vector bundle $V := (L_i'/L_{i-1})^\vee \otimes L_i/L_{i-1} \{2\}$. Thus the normal bundle $N_{Z_\beta^i/Gr_\beta^i}|_{\pi^{-1}(p)}$ is isomorphic to $V|_{\pi^{-1}(p)}$. Now, on $\pi^{-1}(p) \cong \p$, $L_i/L_{i-1}$ restricts to the trivial line bundle while $L_i'/L_{i-1}$ restricts to $\Omega_{\p}^1(1)$. Hence,
$$\H^{-k}(j^* \pi^* \O_p) \cong \wedge^k (\Omega_{\p}^1(1)^\vee \{2\})^\vee \cong \Omega_{\p}^k(k) \{-2k\}.$$
Now $j^* \pi^* \O_p$ is supported on a $\p$ which is mapped one-to-one by $\pi_2$ onto the fibre of $q: X_{s_i(\beta)}^i \rightarrow Y_{d_i(\beta)}$ over the point $(L_1, \dots, L_{i-1}, zL_{i+2}, \dots, zL_n)$. If we denote by $f: \p \rightarrow Y' = Y_{s_i(\beta)}$ the inclusion of this fibre into $Y'$ then 
$$\H^{-k}(\Phi_{\sP}(\O_p)) \cong \H^{-k}(\pi_{2*} \pi_1^* \O_p) \cong \H^{-k}(\pi_{2*} j^* \pi^* \O_p) \cong f_* \Omega_{\p}^k(k) \{-2k\}.$$

If we have $p_1,p_2 \in X_\beta^i$ as above then $\H^{-k}(\Phi_{\sP}(\O_{p_1})) \cong \H^{-k}(\Phi_{\sP}(\O_{p_2})) \cong f_* \Omega_{\p}^k(k) \{-2k\}$. There exists a spectral sequence 
$$E_2^{p,q} = \bigoplus_{a \in \Z} \Ext^p_{Y'}(\H^a(\Phi_{\sP}(\O_{p_1})), \H^{a+q}(\Phi_{\sP}(\O_{p_2}))) \Longrightarrow \Ext^{p+q}_{Y'}(\Phi_{\sP}(\O_{p_1}), \Phi_{\sP}(\O_{p_2})).$$
So to prove $\Ext^n_{Y'}(\Phi_{\sP}(\O_{p_1}), \Phi_{\sP}(\O_{p_2})) = 0$ for $n < 0$ it suffices to show that 
$$\Ext^k_{Y'}(\H^{-j}(\Phi_{\sP}(\O_{p_1})), \H^{-i}(\Phi_{\sP}(\O_{p_2}))) \cong \Ext^k_{Y'}(f_* \Omega_{\p}^j(j) \{-2j\}, f_* \Omega_{\p}^i(i) \{-2i\}) = 0$$
for $-i+j+k < 0$. 

Now $\Ext^k_{Y'}(f_* \Omega_{\p}^j(j) \{-2j\}, f_* \Omega_{\p}^i(i)) \cong \Ext^k_{\p}(f^* f_* \Omega_{\p}^j(j), \Omega_{\p}^i(i))$.  In order to show the vanishing of this Ext, note that $\H^{-l}(f^* f_* \Omega_{\p}^j(j)) \cong \Omega_{\p}^j(j) \otimes \wedge^l N_{f(\p)/Y'}^\vee$ and hence from the spectral sequence 
$$E_2^{p,q} = \Ext^p_{\p}(\H^{-q}(f^* f_* \Omega_{\p}^j(j)), \Omega_{\p}^i(i)) \Longrightarrow \Ext^{p+q}(f^* f_* \Omega_{\p}^j(j), \Omega_{\p}^i(i))$$
it suffices to show that $\Ext^p_{\p}(\Omega_{\p}^j(j) \otimes \wedge^q N_{f(\p)/Y'}^\vee, \Omega_{\p}^i(i)) = 0$ for $p+q < i-j$.  

To identify $N_{f(\p)/Y'}$ consider the exact sequence 
$$0 \rightarrow N_{f(\p)/X_{s_i(beta)}^i} \rightarrow N_{f(\p)/Y'} \rightarrow N_{X_{s_i(\beta)}^i/Y'}|_{f(\p)} \rightarrow 0.$$
Since $f(\p)$ is a fibre of $q: X_{s_i(\beta)}^i \rightarrow Y_{d_i(\beta)}$ we have $N_{f(\p)/X_{s_i(\beta)}^i} \cong \O_{\p}^{\oplus N}$ for some $N \in \Z$. Meanwhile, from Lemma \ref{lem:facts1},
\begin{align*}
N_{X_{s_i(\beta)}^i/Y'}|_{f(\p)} &\cong ((L_{i+1}/L_i)^\vee \otimes L_i/L_{i-1} \{2\})|_{f(\p)} \\
&\cong f_*(\O_{\p}(1)^\vee \otimes \Omega_{\p}^1(1) \{2\}) \cong f_*(\Omega_{\p}^1) \{2\}).
\end{align*}
Thus we get the exact sequence $0 \rightarrow \O_{\p}^{\oplus N} \rightarrow N_{f(\p)/Y'} \rightarrow \Omega_{\p}^1 \{2\} \rightarrow 0$. \\ Since $\Ext^1_{\p}(\Omega_{\p}^1, \O_{\p}) = 0$ this sequence splits and we obtain
$N_{f(\p)/Y'} \cong \O_{\p}^{\oplus N} \oplus \Omega_{\p}^1 \{2\}$. 

Subsequently we are left with showing $\Ext^p_{\p}(\Omega_{\p}^j(j) \otimes \wedge^q (\O_{\p}^{\oplus N} \oplus \Omega_{\p}^1)^\vee, \Omega_{\p}^i(i)) = 0$ whenever $p+q < i-j$. It is easy to see this is the same as showing 
$$\Ext^p_{\p}(\Omega_{\p}^j(j), \Omega_{\p}^i(i) \otimes \Omega_{\p}^q) = 0 \text{ for } i-j > p+q$$
which is proved in corollary \ref{cor:vanexts2}. 

Vanishing for $n > \dim(Y)$ follows by Serre duality. More precisely, from above we know the normal bundle $N_{f(\p)/Y'}$ is $\Omega_{\p}^1 \oplus \O_{\p}^{\oplus N} \{2\}$. Since $\Phi_{\sP}(\O_{p_1})$ and $\Phi_{\sP}(\O_{p_2})$ are supported on $f(\p)$ it is enough to consider the formal neighbourhood of $f(\p)$ in order to compute ext's between them. Since the total space $\Omega_{\p}^1$ has trivial canonical bundle so does the formal neighbourhood of $f(\p)$. Thus 
$$\Ext^n_{Y'}(\Phi_{\sP}(\O_{p_1}), \Phi_{\sP}(\O_{p_2})) \cong \Ext^{\dim(Y')-n}_{Y'}(\Phi_{\sP}(\O_{p_2}), \Phi_{\sP}(\O_{p_1})) = 0$$
if $\dim(Y')-n<0$. 
\end{proof}
It remains to establish Corollary \ref{cor:vanexts2} which was used in the above proof.  This will be a somewhat length computation.

\begin{Lemma} \label{lem:vanhom} We have:
\begin{itemize}
\item $H^0(\Omega_{\p}^i(l)) = 0$ if $l \le i \ne 0$
\item $H^k(\Omega_{\p}^i(l)) = 0$ for $1 \le k \le m-2$ unless $(i,l) = (k,0)$
\item $H^{m-1}(\Omega_{\p}^i(l)) = 0$ if $i \le m-1+l$
\end{itemize}
\end{Lemma}
\begin{proof}
As pointed out by the referee, this follows from Bott's formula (e.g. \cite{OSS}, page 8). 
\end{proof}

\begin{Corollary} \label{cor:vanexts1} 
For $1 \le k \le m-2$ we have $\Ext^k(\Omega_{\p}^j, \Omega_{\p}^i(l)) = 0$ if any of the following three conditions hold
\begin{itemize}
\item $l+k+j-i \ne 0,-1$
\item $l+k+j-i=0$ and either $k > i$ or $l > 0$
\item $l+k+j-i=-1$ and either $m-1 > k+j$ or $l < -m$.
\end{itemize}
Also
\begin{itemize}
\item $\Ext^0(\Omega_{\p}^j, \Omega_{\p}^i(l)) = 0$ if $l+j < i \ne 0$
\item $\Ext^{m-1}(\Omega_{\p}^j, \Omega_{\p}^i(l)) = 0$ if $-l-m+i < j \ne 0$ 
\end{itemize}
\end{Corollary}
\begin{proof}
Applying $Hom(\cdot, \Omega^i(l))$ to $0 \rightarrow \Omega^{j} \rightarrow \O(-j)^{\oplus {\binom{m}{j}}} \rightarrow \Omega^{j-1} \rightarrow 0$ we get the long exact sequence$$\dots \rightarrow \Ext^k(\Omega^i(l+j))^{\oplus \binom{m}{j}} \rightarrow \Ext^k(\Omega^{j}, \Omega^i(l)) \rightarrow \Ext^{k+1}(\Omega^{j-1}, \Omega^i(l)) \rightarrow \Ext^{k+1}(\Omega^i(l+j))^{\oplus \binom{m}{j}} \rightarrow \dots$$

We deal with the case $1 \le k \le m-2$ first. From the sequence above we get an injection $\Ext^k(\Omega^j, \Omega^i(l)) \rightarrow \Ext^{k+1}(\Omega^{j-1}, \Omega^i(l))$ if $\Ext^k(\Omega^i(l+j))=0$. By Lemma \ref{lem:vanhom}, $\Ext^k(\Omega^i(l+j))=0$ unless $(k,l+j)=(i,0)$. If $l+k+j-i < 0$ (or $l+k+j-i=0$ and either $k > i$ or $l > 0$) these conditions are satisfied. We can repeat and it is easy to check that at each step the condition for vanishing from \ref{lem:vanhom} is satisfied so we get a series of injections whose composition is the injective map 
$$\Ext^k(\Omega^j, \Omega^i(l)) \rightarrow \Ext^{m-1}(\Omega^{j+k-m+1}, \Omega^i(l)).$$ 
Finally there is a surjection from a direct sum of $H^{m-1}(\Omega^i(l+k+j-m+1))$ onto $\Ext^{m-1}(\Omega^{k+j-m+1}, \Omega^i(l))$. By Lemma \ref{lem:vanhom}, $H^{m-1}(\Omega^i(l+k+j-m+1))$ vanishes if $l+k+j-i \ge 0$ and we are done. The case $l+k+j-i < 0$ follows from the Serre duality isomorphism $\Ext^k(\Omega^j,\Omega^i(l)) \cong \Ext^{m-1-k}(\Omega^i, \Omega^j(-l-m))$.

Next we deal with $k=0$. As above we have an injection $\Ext^0(\Omega^j, \Omega^i(l)) \rightarrow \Ext^{1}(\Omega^{j-1}, \Omega^i(l))$ if $H^0(\Omega^i(l+j))=0$. If $l+j < i$ we have $H^0(\Omega^i(l+j))=0$ by \ref{lem:vanhom}. From above $\Ext^{1}(\Omega^{j-1}, \Omega^i(l)) = 0$ if $l+j-i<-1$ or $l+j-i=-1$ and $m-1>j$. The condition $m-1>j$ is vacuous unless $j=m-1$ in which case $\Ext^0(\Omega^j,\Omega^i(l)) \cong H^0(\Omega^i(l+m)) \cong H^0(\Omega^i(i))=0$ if $i \ne 0$. 

The result for $k=m-1$ follows by Serre duality.
\end{proof}

\begin{Corollary} \label{cor:vanexts2} 
$\Ext^q(\Omega_{\p}^j, \Omega_{\p}^i \otimes \Omega_{\p}^p(i-j)) = 0$ if $i-j > p+q$.
\end{Corollary}
\begin{proof}
Applying $\Hom(\Omega^j, \cdot)$ to $0 \rightarrow \Omega^i \otimes \Omega^p(l) \rightarrow \O(-i)^{\oplus \binom{m}{i}} \otimes \Omega^p(l) \rightarrow \Omega^{i-1} \otimes \Omega^p(l) \rightarrow 0$ we get the long exact sequence
\begin{eqnarray*}
\dots \rightarrow \Ext^{q}(\Omega^j, \Omega^p(l-i-1))^{\oplus \binom{m}{i+1}} \rightarrow \Ext^{q}(\Omega^j, \Omega^{i} \otimes \Omega^p(l)) \rightarrow \\
\rightarrow \Ext^{q+1}(\Omega^j, \Omega^{i+1} \otimes \Omega^p(l)) \rightarrow \Ext^{q+1}(\Omega^j, \Omega^p(l-i-1))^{\oplus \binom{m}{i+1}} \rightarrow \dots
\end{eqnarray*}

For $0 \le q \le m-2$ then we get an injection $\Ext^q(\Omega^j, \Omega^{i} \otimes \Omega^p(i-j)) \rightarrow \Ext^{q+1}(\Omega^{j}, \Omega^{i+1} \otimes \Omega^p(i-j))$ if $\Ext^q(\Omega^j, \Omega^p(-j-1))=0$. Notice $(-j-1)+j+q-p = q-p-1$. If $q > p$ then by corollary \ref{cor:vanexts1} we have $\Ext^q(\Omega^j, \Omega^p(-j-1)) = 0$. Now we can repeat to get a series of injections whose composition is an injective map 
$$\Ext^q(\Omega^j, \Omega^i \otimes \Omega^p(i-j)) \rightarrow \Ext^{m-1}(\Omega^j, \Omega^{i+m-1-q} \otimes \Omega^p(i-j)).$$ 
Since $i-j > p+q$ we find $i+m-1-q > m-1$ so $\Omega^{i+m-1-q}=0$ and we are done. The same argument works if $q < p$.

If $q=p>0$ then $(-j-1)+j+q-p = -1$ so by corollary \ref{cor:vanexts1}, $\Ext^q(\Omega^j, \Omega^p(-j-1))=0$ because $m-1 \ge i > q+j$. Repeating this $a$ times we get the map 
$$\Ext^{q+a}(\Omega^j, \Omega^{i+a} \otimes \Omega^p(i-j)) \rightarrow \Ext^{q+a+1}(\Omega^j, \Omega^{i+a+1} \otimes \Omega^p(i-j))$$ 
which is injective if $\Ext^{q+a}(\Omega^j, \Omega^p(-j-a-1))=0$. By \ref{cor:vanexts1} this vanishes if $m-1 > (q+a)+j$. In particular, we get injections for $0 \le a \le m-i-1$ since $q+(m-i-1)+j < m-1$. Thus $\Ext^q(\Omega^j, \Omega^p(-j-1))$ injects into $\Ext^{q+m-i}(\Omega^j, \Omega^m \otimes \Omega^p(i-j))$ which is zero since $\Omega^m=0$. Finally, the case $p=q=0$ follows directly from \ref{cor:vanexts1}.
\end{proof}

Returning to our goal of showing that $ \Phi_{\sP} $ is an equivalence, we will now prove the following.

\begin{Proposition}
$\sP_L * \sP \cong \O_\Delta $
\end{Proposition}

\begin{proof}
Next we consider the adjoint map $\alpha: \sP_L * \sP \rightarrow \O_\Delta \in D(Y \times Y)$ and let $K = \Cone(\alpha)$. We will show $K = 0$. To do this let's play a little with the fibre product diagram:
\begin{equation*}
\begin{CD}
Y @>i_{p \times Y}>> Y \times Y \\
@V{\pi}VV @V{\pi_2}VV \\
p @>i_p>> Y
\end{CD}
\end{equation*}
where $i_p: p \rightarrow Y$ is the inclusion of an arbitrary point. We know 
\begin{equation*}
\pi_{2*}(\sP_L * \sP) 
\cong \pi_{2*}((\sP_L * \sP) \otimes \pi_1^* \O_{Y}) 
\cong \Phi_{\sP_L * \sP}(\O_{Y}) \cong {\Phi_{\sP}}_L \circ \Phi_{\sP} (\O_{Y}) 
\cong {\Phi_{\sP}}_L (\O_{Y'}) \cong \O_{Y}.
\end{equation*}
Here we used that 
$\Phi_{\sP}(\O_Y) \cong \pi_{2*} \O_{Z_\beta^i} \cong \O_{Y_{s_i(\beta)}}$
from Lemma \ref{lem:facts3}. Similarly 
\begin{equation*}
{\Phi_{\sP}}_L (\O_{Y'}) 
\cong \pi_{2*}(\sP_L) 
\cong \pi_{2*}(\O_{Z_\beta^i}^\vee \otimes \pi_1^* \omega_{Y'} [\dim(Y')])
\cong \pi_{2*}(\omega_{Z_\beta^i} \otimes \pi_2^* \omega_{Y}^\vee) 
\cong \O_{Y}
\end{equation*}
where the last isomorphism also follows from Lemma \ref{lem:facts3}. 

Since 
$$\pi_{2*}(\sP_L * \sP) \cong {\Phi_{\sP}}_L \circ \Phi_{\sP} (\O_Y) \rightarrow \Phi_{\O_\Delta}(\O_Y) \cong \pi_{2*}(\O_\Delta)$$
is an adjunction map it follows that it is non-zero. But $\Hom_Y(\O_Y, \O_Y) \cong \C$ so $\pi_{2*}(\alpha)$ is an isomorphism and hence $\pi_{2*}(K) = 0$. 

Since $\pi_2$ is smooth we have $\pi_* i_{p \times Y}^* (K) \cong i_p^* \pi_{2*} (K) = 0.$ On the other hand, $K$ is supported on the diagonal so $i_{p \times Y}^*(K)$ is supported on the point $p$. Since $\pi_* i_{p \times Y}^*(K) = 0$ this means $i_{p \times Y}^*(K)=0$. This is true for any $p$ so $K$ has empty support. Consequently $K=0$ and we have shown $\sP_L * \sP \cong \O_\Delta$. 
\end{proof}

Finally, we are ready for the main result of this section.
\begin{Theorem} $\Phi_{\O_{Z_\beta^i}}: D(Y_\beta) \rightarrow D(Y_{s_i(\beta)})$ is an equivalence when $\beta_i \ne \beta_{i+1} $. 
\end{Theorem}
\begin{proof}
Since $\sP_L * \sP \cong \O_\Delta$ it follows $\Phi_{\sP}$ is fully faithful. It remains to show that $\sP \otimes \pi_1^* \omega_{Y} \cong \sP \otimes \pi_2^* \omega_{Y'}$. By Lemma \ref{lem:facts2} we have
\begin{eqnarray*}
\sP \otimes \pi_1^* \omega_{Y} 
&\cong& \O_{Z_\beta^i} \otimes \pi_1^* \det(L_n/L_0)^m \{- 2m b_\beta^n - 2 \sum_i b_\beta^{i-1} \beta_i \} \\
&\cong& \O_{Z_\beta^i} \otimes \pi_2^* \det(L_n/L_0)^m \{- 2m b_{\beta'}^n - 2 \sum_i b_{\beta'}^{i-1} \beta'_i \} \\
&\cong& \sP \otimes \pi_2^* \omega_{Y'}
\end{eqnarray*}
where to obtain the shift in the second equality we used that $b_\beta^{i-1} \beta_i + b_\beta^i \beta_{i+1} = b_\beta^{i-1}(\beta_i + \beta_{i+1}) + \beta_i \beta_{i+1} = b_{\beta'}^{i-1} \beta'_i + b_{\beta'}^i \beta'_{i+1}$ where $\beta' = s_i(\beta)$. 
\end{proof}

\begin{Remark}
In the simplest situation, the two projections from $Y_{1,m-2,1}$ to $Y_{1,m-1}$ and $Y_{m-1,1}$ describe a Mukai flop. The fact that $\O_{Z_{1,m-1}^1}$ induces an equivalence $D(Y_{1,m-1}) \rightarrow D(Y_{m-1,1})$ was proven by Kawamata in \cite{kaw} and Namikawa in \cite{Na}. More generally, if $\beta$ contains a string $(1,m-1)$ starting in the $i$th position then projections from $W_\beta^i$ (the diagonal-like first component of $Z_\beta^i$) to $Y_\beta$ and $Y_{s_i(\beta)}$ describe a birational surgery given as a family of Mukai flops. Though stated a little differently, Namikawa in \cite{Na} considers such a family of Mukai flops and proves that $\O_{Z_\beta^i}$ induces an equivalence of derived categories. 

In \cite{Na} Namikawa first proves the simplest situation using spanning classes and then builds up from there to the more general case. In this paper we do not consider spanning classes and prove the general case in one step (our method does not differentiate between a Mukai flop and a family of Mukai flops). Our approach is more in the spirit of \cite{BO} and \cite{BKR}. 

\end{Remark}

\subsection{Invariance Under Reidemeister Moves (III)}\label{section:RIII}

\begin{Theorem}\label{prop:R3} 
The map $ \Psi $ is invariant under Reidemeister III moves, i.e.
\begin{equation*}
\T_{s_{i+1}(s_i(\beta))}^i(l_1) \circ \T_{s_i(\beta)}^{i+1}(l_2) \circ \T_\beta^i(l_3) = \T_{s_i(s_{i+1}(\beta))}^{i+1}(l_3) \circ \T_{s_{i+1}(\beta)}^i(l_2) \circ \T_\beta^{i+1}(l_1)
\end{equation*}
\end{Theorem}

We will prove this invariance in the case that all crossing are of type 1 and $ \beta_i = \beta_{i+1} = \beta_{i+2} = 1 $:
\begin{equation} \label{eq:casetoprove}
\T_\beta^i(1) \circ \T_\beta^{i+1}(1) \circ \T_\beta^i(1) \cong \T_\beta^{i+1}(1) \circ \T_\beta^i(1) \circ \T_\beta^{i+1}(1)
\end{equation}
The rest of the RIII cases follows from this case by use of RO, RI and RII.

In order to prove (\ref{eq:casetoprove}), we will compute both sides explicitly..  

In this section we will suppress all subscripts $ \beta $ since $ \beta $ will not change at all.  In particular, let
\begin{equation*}
\sU^i = \O_{V_\beta^i} \otimes \E'_i \otimes \E_{i+1}^\vee, \quad \sT^i = \O_{Z^i}.
\end{equation*}
This $ \sT^i $ is related to the real $ \sT^i_\beta(1) $ by $ \sT^i = \sT^i_\beta[-m+1]\{m-1\} $. 

Since $ Z^i $ has two components $ \Delta $ and $ V^i $ which meet at a divisor in each of them, we have exact sequence
\begin{equation*}
\O_{V_\beta^i}(-\Delta) \rightarrow \O_{Z^i} \rightarrow \O_\Delta
\end{equation*}
which leads to
\begin{equation} \label{eq:Tiascone}
\sT^i \cong \Cone( \O_\Delta[-1] \rightarrow \sU^i).
\end{equation}
(The attentive reader will note that this is the dual of the isomorphism obtained in the proof of Theorem \ref{th:kerneltwist}).

Consider the set-theoretic convolution
\begin{equation*}
Z^{i i+1 i} := Z^i * Z^{i+1} * Z^i = \{ (L_\bullet, L'_\bullet) : L_j = L'_j \text{ for } j \ne i, i+1 \}.
\end{equation*}

We will prove the following result which easily implies (\ref{eq:casetoprove}).
\begin{Theorem} \label{th:impliesR3}
$\sT^i * \sT^{i+1} * \sT^i = \O_{Z^{i i+1 i}} = \sT^{i+1} * \sT^i * \sT^{i+1} $
\end{Theorem}

Khovanov-Thomas \cite{KT} have given a braid group action on the cotangent bundle of the flag variety.  The overall geometry of our situation is different, however, the geometry which arises when trying to prove the braid relation (Theorem \ref{th:impliesR3}) is quite similar.  This allows us to make use of the method of proof from Theorem 4.5 of \cite{KT}.  Following their method we will compute $\sT^i * \sT^{i+1} * \sT^i$  explicitly, repeatedly using equation (\ref{eq:Tiascone}). 

The variety $ Z^{i i+1 i} $ has 6 components which are shown in Figure \ref{fi:YYvarieties}.

\begin{figure}
\begin{center}
\psfrag{Li-1}{$L_{i-1}$} 
\psfrag{Li}{$L_i$}
\psfrag{Li+1}{$L_{i+1}$} 
\psfrag{Li+2}{$L_{i+2}$} 
\psfrag{Lip}{$L_i'$} 
\psfrag{Li+1p}{$L_{i+1}'$} 
\psfrag{Vii+1i}{$V^{i i+1 i}$} 
\psfrag{Vii+1}{$V^{i i+1}$} 
\psfrag{Vi+1i}{$V^{i+1 i}$} 
\psfrag{Vi}{$V^i$} 
\psfrag{Vi+1}{$V^{i+1}$} 
\psfrag{Delta}{$\Delta$} 
\puteps[0.35]{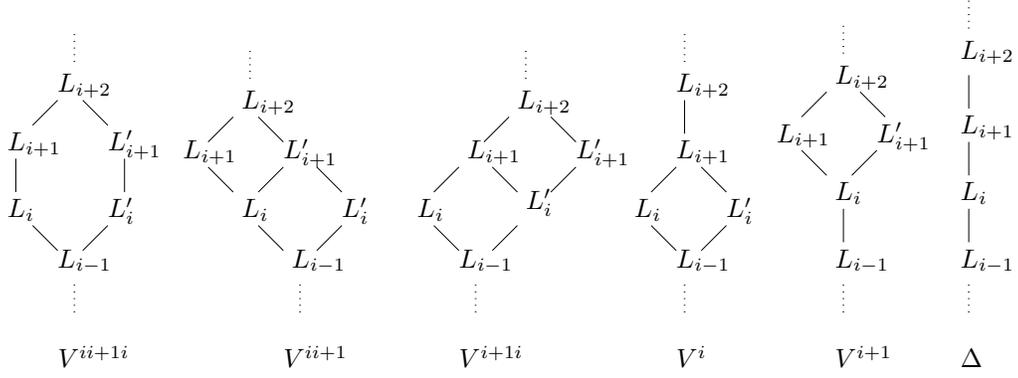}
\end{center}
\caption{The components of $ Z^{i i+1 i}$.}
\label{fi:YYvarieties}
\end{figure}

The first component is the ``shield''
\begin{equation*} 
V^{i i+1 i} := \{ (L_\bullet, L'_\bullet) : z L_{i+2} \subset L_{i-1} \}.
\end{equation*}

There is a corresponding sheaf
\begin{equation*}
\sU^{i i+1 i} = \O_{V_\beta^{i i+1 i}} \otimes (\E_i \E'_i) \otimes (\E_{i+2} \E'_{i+2})^\vee. 
\end{equation*}

The next two components are called the ``left-pointing ladder'' and the ``right-pointing ladder'' respectively.
\begin{equation*}
\begin{gathered}
V^{i i+1} = \{ (L_\bullet, L'_\bullet) : L_i \subset L_{i+1}', z L_{i+1}' \subset L_{i-1}, z L_{i+2} \subset L_i \} \\
V^{i+1 i} = \{ (L_\bullet, L'_\bullet) : L'_i \subset L_{i+1}, z L_{i+1} \subset L_{i-1}, z L_{i+2} \subset L'_i \}
\end{gathered}
\end{equation*}

The other three components are $ V^i $, the lower diamond, $ V^{i+1} $ the upper diamond, and $ \Delta $ the vertical line.

It is convenient to combine the last 4 components together as
\begin{equation*}
Z^{i+1 i} = V^{i+1 i} \cup V^i \cup V^{i+1} \cup \Delta
\end{equation*}

\begin{Proposition} \label{th:firststage}
$ \sU^i * \sU^{i+1} * \sU^i \cong \sU^{i i+1 i} \oplus \sU^i \{-2\}$
\end{Proposition}

\begin{proof}
Since $\pi_{12}^{-1}(V^{i+1})$ and $\pi_{23}^{-1}(V^i)$ intersect transversally in the section 
$$W = \{ (L_\bullet, L'_\bullet, L''_\bullet) : L_i = L_i', L_{i+1}' = L_{i+1}'', zL_{i+2} \subset L_i, zL_{i+1}' \subset L_{i-1} \}$$ 
over $V^{i i+1}$ we have 
\begin{equation*}
\sU^i * \sU^{i+1} \cong \pi_{13*}(\O_W \otimes \E_{i+2}^\vee \otimes \E_i'') \cong \O_{V^{i i+1}} \otimes \E'_i \otimes \E_{i+2}^\vee .
\end{equation*}

Next we see that
\begin{equation*}
\sU^i * \sU^{i+1} * \sU^i \cong {\pi_{13}}_* (\O_{A \cup B} \E'_i \E''_i \E_{i+2}^\vee \E_{i+1}^\vee)
\end{equation*}
where the union $ A \cup B $ is the ``peace variety''. Here
\begin{equation*}
A = \{ (L_\bullet, L'_\bullet, L''_\bullet) : z L_{i+2} \subset L_{i-1}, L_{i+1} = L'_{i+1} \}
\end{equation*}
which maps to the shield $ V^{i i+1 i}  $  and 
\begin{equation*}
B = \{ (L_\bullet, L'_\bullet, L''_\bullet) : z L_{i+2} \subset L'_i, L_{i+1} = L'_{i+1} = L''_{i+1}, zL_{i+1} \subset L_{i-1} \}
\end{equation*} 
which maps to the lower diamond $V^i$.  Let $ E = A \cap B $.  The map $ \pi_{13} $ fails to be 1-1 exactly over $ E $, where it is $ \mathbb{P}^1 $ to 1.  In particular
\begin{equation*}
E = \{ (L_\bullet, L'_\bullet, L''_\bullet) : z L_{i+2} \subset L_{i-1}, L_{i+1} = L'_{i+1} = L''_{i+1} \}.
\end{equation*}
See Figure \ref{fi:YYYvarieties} for incidence diagrams which describe the varieties $A, B, E$.

\begin{figure}
\begin{center}
\psfrag{Li-1}{$L_{i-1}$} 
\psfrag{Li}{$L_i$}
\psfrag{Li+1}{$L_{i+1}$} 
\psfrag{Li+2}{$L_{i+2}$} 
\psfrag{Lip}{$L_i'$} 
\psfrag{Li+1p}{$L_{i+1}'$} 
\psfrag{Lipp}{$L_i''$} 
\psfrag{Li+1pp}{$L_{i+1}''$}
\psfrag{A}{$A$} 
\psfrag{B}{$B$} 
\psfrag{E}{$E$} 
\psfrag{z}{$z$} 
\puteps[0.35]{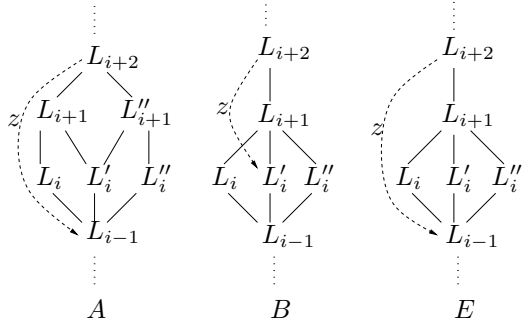}
\end{center}
\caption{The varieties $ A, B, E$.}
\label{fi:YYYvarieties}
\end{figure}

We claim that $ A $ is a blowup with exceptional divisor $E$.  To see this note that $ \pi_{13}(E) $ is cut out of $ \pi_{13}(A) = V^{i i+i i} $ by a section $s$ of the rank 2 vector bundle $ \mathcal{V} = \Hom(L_{i+1}/L_{i-1}, L_{i+2}/L'_{i+1}) $ (since it is given by the condition $ L_{i+1} = L'_{i+1}$).  Hence we see that the blowup of $ V^{i i+1 i} $ along $ \pi_{13}(E) $ is the locus
\begin{equation*}
Bl_{\pi_{13}(E)}(V^{i i+1 i}) = \{ (L_\bullet, L''_\bullet, M) \in V^{i i+1 i} \times \mathbb{P}(\mathcal{V}) : s(L_\bullet, L''_\bullet) \in M \}.
\end{equation*}
Given a point $(L_\bullet, L''_\bullet, M) $, we may define $L_i' = \{ v \in L_{i+1}/L_{i-1} : m(v) = 0 \text{ for } m \in M\}$.  This gives us a map $ Bl_{\pi_{13}(E)}(V^{i i+1 i}) \rightarrow A $ which is easily seen to be an isomorphism.  Under this map, the exceptional divisor of the blowup is carried to $ E$ and the blowup map is the restriction of $ \pi_{13} $ to $ A $.

In a similar fashion, we also see that $ B $ is the blowup of $ \pi_{13}(B) = V^i $ along $ \pi_{13}(E)$.

Returning to our computation, we have the short exact sequence
\begin{equation*}
0 \rightarrow \O_{A}(-E) \oplus \O_B(-E) \rightarrow \O_{A \cup B} \rightarrow \O_E \rightarrow 0
\end{equation*}
which leads to the exact triangle
\begin{equation*}
\begin{gathered}
{\pi_{13}}_* \O_{A}(-E) \E'_i \E''_i \E_{i+2}^\vee \E_{i+1}^\vee  \oplus {\pi_{13}}_*  \O_B(-E) \E'_i \E''_i \E_{i+2}^\vee \E_{i+1}^\vee \\ 
\rightarrow {\pi_{13}}_*  \O_{A \cup B} \E'_i \E''_i \E_{i+2}^\vee \E_{i+1}^\vee  \rightarrow {\pi_{13}}_*  \O_E \E'_i \E''_i \E_{i+2}^\vee \E_{i+1}^\vee 
\end{gathered}
\end{equation*}

We now claim that ${\pi_{13}}_*  \O_E \E'_i \E''_i \E_{i+2}^\vee \E_{i+1}^\vee = 0 $.  To see this note that this line bundle is $ \O(-1) $ on the fibres of $ \pi_{13}|_E$.  Since this pushforward is 0, we conclude that
\begin{equation*}
\sU^i * \sU^{i+1} * \sU^i \cong {\pi_{13}}_* \O_{A}(-E) \E'_i \E''_i \E_{i+2}^\vee \E_{i+1}^\vee \oplus {\pi_{13}}_*  \O_B(-E) \E'_i \E''_i \E_{i+2}^\vee \E_{i+1}^\vee 
\end{equation*}

On the other hand, $ E $ is cut out of $ A $ by a section of the line bundle $ \Hom(L''_{i+1}/L'_i, L_{i+2}/L_{i+1}) $.  Hence $ \O_A(-E) \cong (L''_{i+1}/L'_i) \E_{i+2}^\vee $. Thus
\begin{equation*}
\O_{A}(-E) \E'_i \E''_i \E_{i+2}^\vee \E_{i+1}^\vee \cong  \O_A \E'_i (L''_{i+1}/L'_i) \E''_i (\E_{i+2}^\vee)^2 \E_{i+1}^\vee \cong \O_A {\E''_i}^2 \E''_{i+1} \E_{i+1}^\vee (\E_{i+2}^\vee)^2,
\end{equation*}
where in the last step we used the ``diamond identity'' on $ L_{i-1}, L'_i, L''_i, L''_{i+1} $.  This line bundle is pulled back from the base, so we can apply push-pull.  Since $ \pi_{13}|_A $ is a blowup map, we see that $ {\pi_{13}}_* \O_A = \O_{V^{i i+1 i}} $ and so we conclude that
\begin{equation*}
{\pi_{13}}_* \O_{A}(-E) \E'_i \E''_i \E_{i+2}^\vee \E_{i+1}^\vee \cong \O_{V^{i i+1 i}} {\E'_i}^2 \E'_{i+1} \E_{i+1}^\vee (\E_{i+2}^\vee)^2 \cong \sU^{i i+1 i}
\end{equation*}
where the second isomorphism follows from the fact that $\E_i \E_{i+1} \E_{i+2} \cong \det(L_{i+2}/L_{i-1}) \cong \E'_i \E'_{i+1} \E'_{i+2}$ on $V^{i i+1 i}$.

Now, we compute the other summand.  First, note that $ E $ is cut out of $ B $ by a section of the line bundle $ \Hom(L_{i+2}/L_{i+1}, L'_i/L_{i-1})\{2\}$.  Hence $ \O_B(-E) \cong \E_{i+2} {\E'_i}^\vee\{-2\}$.  Thus,
\begin{equation*}
{\pi_{13}}_* \O_B(-E) \E'_i \E''_i \E_{i+2}^\vee \E_{i+1}^\vee \cong {\pi_{13}}_* \O_B \E''_i \E_{i+1}^\vee\{-2\}  \cong \O_{V^i} \E'_i \E_{i+1}^\vee \{-2\},
\end{equation*}
where again we have used that $ B $ is a blowup to conclude that $ {\pi_{13}}_* \O_B = \O_{V^i} $.

Combining everything, we conclude that
\begin{equation*}
\sU^i * \sU^{i+1} * \sU^i \cong \sU^{i i+1 i} \oplus \sU^i\{-2\}
\end{equation*}
\end{proof}

Before calculating $ \sU^i * \sT^{i+1} * \sU^i $, we will need some preliminary lemmas.

First, consider the diagram
\begin{equation*}
\begin{CD}
\sU^i * \sU^{i+1} [-1] @>f>> \sU^i * \sU^{i+1} * \sU^i @>p_2>> \sU^i \{-2\} \\
@. @VVp_1V @. \\
 @. \sU^{ii+1i} @.
\end{CD}
\end{equation*}
where $p_1$ and $p_2$ are the two projections (see Proposition \ref{th:firststage}) and where $ f $ comes from the map $ \O_\Delta[-1] \rightarrow \sU^i $. 

\begin{Lemma}\label{lem:p1p2} Using the notation above we have $p_1 \circ f \ne 0$ and $p_2 \circ f \ne 0$. 
\end{Lemma}
\begin{proof}
Consider 
$$\sU^i * \sU^{i+1} * \sT^i \cong \Cone(\sU^i * \sU^{i+1} [-1] \rightarrow \sU^i * \sU^{i+1} * \sU^i).$$ 
Now $\pi_{12}^{-1}(Z^i)$ and $\pi_{23}^{-1}(V^{ii+1})$ meet in the complete intersection $A \cup B \cup W$ which is in the right codimension (see the notation used in the proof of \ref{th:firststage}). Hence the left hand side is $\pi_{13*}$ of $\O_{A \cup B \cup W}\E_i'' \E_{i+2}^\vee$. The right hand map is $\pi_{13*}$ of $\O_W \E_i'' \E_{i+2}^\vee [-1] \rightarrow \O_{A \cup B} \E_i'' \E_{i+2}^\vee \E_i \E_{i+1}^\vee$. Since the cone of this map is $\O_{A \cup B \cup W} \E_i'' \E_{i+2}^\vee$ the map is non-zero when restricted to $A \smallsetminus E$ and $B \smallsetminus E$. Since $\pi_{13}$ maps $A \smallsetminus E$ and $B \smallsetminus E$ one-to-one the pushforwards of these restricted maps are also non-zero. But these pushforwards are precisely $p_1 \circ f$ and $p_2 \circ f$ restricted to $V^{ii+1i} \smallsetminus \pi_{13}(E)$ and $V^i \smallsetminus \pi_{13}(E)$ respectively. 
\end{proof}

\begin{Lemma} \label{th:nohom}
$ \Hom^k(\sU^i, \sU^{i i+1 i}) = 
\begin{cases} 
0 \text{ if } k = 0,1 \\ 
\mathbb{C} \text{ if } k = 2 
\end{cases}
$
\end{Lemma}
\begin{proof}
Note that $ \Hom^k(\sU^i, \sU^{i i+1 i}) = \Hom^k(\O_{V^i}, \O_{V^{i i+1 i}} \E_i' \E_{i+1}' (\E_{i+2}^\vee)^2) $ since $\sU^{i i+1 i}$ is also isomorphic to $\O_{V^{i i+1 i}} {\E'_i}^2 \E'_{i+1} \E_{i+1}^\vee (\E_{i+2}^\vee)^2$.

Now, $ V^i $ and $ V^{i i+1 i} $ intersect in a smooth subvariety of codimension 2 in $ V^{i i+1 i} $.  This subvariety is defined by the vanishing of a section of the vector bundle $ \mathcal{V} = \Hom(L'_{i+1}/L_{i-1}, L_{i+2}/L_{i+1}) $ . Hence $ \det(\mathcal{V}^\vee) = \E_i' \E_{i+1}' (\E_{i+2}^\vee)^2 $ and the result follows by Corollary \ref{cor:abstractnonsense3}.
\end{proof}

\begin{Lemma} \label{th:UiUi}
$\sU^i * \sU^i \cong \sU^i[1]\{-2\} \oplus \sU^i[-1]$
\end{Lemma}
\begin{proof}
From the calculation in the proof of Theorem \ref{th:kerneltwist}, we see that $ \sP_\beta^i * {\sP_\beta^i}_R \cong \sU^i[-1]\{2\}$.  Hence 
\begin{equation*} \label{eq:UiUi1}
 \sU^i * \sU^i \cong \sP_\beta^i * {\sP_\beta^i}_R * \sP_\beta^i * {\sP_\beta^i}_R [2]\{-4\}.
 \end{equation*}

From the proof of Lemma \ref{th:Pissphere}, we have that $  {\sP_\beta^i}_R * \sP^i \cong \O_\Delta \oplus \O_\Delta [-2]\{2\} $ and the result follows.
\end{proof}

We also need the following lemma which Bar-Natan \cite{BN} calls ``Gaussian elimination''.
\begin{Lemma}\label{lem:BN} 
Let $ X, Y, Z, W$ be four objects in a triangulated category.  Let $f = \left( \begin{matrix} A & B \\ C & D \end{matrix} \right) : X \oplus Y \rightarrow Z \oplus W $ be a morphism.  If $D$ is an isomorphism, then $\Cone(f) \cong \Cone(A - BD^{-1}C: X \rightarrow Z)$. 
\end{Lemma}
\begin{proof}
This is essentially Lemma 4.2 from \cite{BN}, but we include a proof for completeness. They key is the following commutative diagram. The two vertical maps are isomorphisms and consequently the two horizontal cones must be isomorphic. 
\begin{equation*}
\begin{CD}
X \oplus Y @>{{\left( \begin{matrix} A & B \\ C & D \end{matrix} \right)}}>> Z \oplus W \\
@V{{\left( \begin{matrix} 1 & 0 \\ D^{-1}C & 1 \end{matrix} \right)}}VV @VV{{\left( \begin{matrix} 1 & -BD^{-1} \\ 0 & 1 \end{matrix} \right)}}V \\
X \oplus Y @>>{{\left( \begin{matrix} A-BD^{-1}C & 0 \\ 0 & D \end{matrix} \right)}}> Z \oplus W  
\end{CD}
\end{equation*}
\end{proof}

Now we compute the next stage. 

\begin{Proposition}\label{prop:secondstage}
$ \sU^i * \sT^{i+1} * \sU^i \cong \Cone(\sU^i[-2] \rightarrow \sU^{i i+1 i}) $ where the map is non-zero.
\end{Proposition}
\begin{proof}
By equation (\ref{eq:Tiascone}), we have
\begin{equation*}
\begin{split}
\sU^i * \sT^{i+1} * \sU^i &\cong \Cone( \sU^i * \sU^i[-1] \rightarrow \sU^i * \sU^{i+1} * \sU^i) 
\end{split}
\end{equation*}

However, by Lemma \ref{th:UiUi}, we have that
\begin{equation*}
 \sU^i * \sU^i[-1] \cong \sU^i[-2] \oplus \sU^i \{-2\}.
\end{equation*}
and by Proposition \ref{th:firststage}, we have
\begin{equation*}
\sU^i * \sU^{i+1}* \sU^i \cong \sU^{i i+1 i} \oplus \sU^i\{-2\}.
\end{equation*}

Denote by $\left( \begin{matrix} A & B \\ C & D \end{matrix} \right)$ the map $\sU^i[-2] \oplus \sU^i \{-2\} \rightarrow \sU^{i i+1 i} \oplus \sU^i \{-2\}$ above. Since $\Hom(\sU^i\{-2\}, \sU^{i i+1 i}) = 0$ we get $B=0$ so by Lemma \ref{lem:BN} it suffices to show that $D$ is invertible and $A \ne 0$.

We first show that $A \ne 0$. Consider the following commutative diagram equipped with the natural adjunction morphisms: 
\begin{equation*}
\begin{CD}
\sU^i[-2] @>g>> \sU^i * \sU^{i+1}[-1] @>h>> \sU^{i i+1 i} \\
@Vf'VV @V{f}VV @| \\
\sU^i * \sU^i [-1] @>{g'}>> \sU^i * \sU^{i+1} * \sU^i @>p_1>> \sU^{i i+1 i}
\end{CD}
\end{equation*}
By Lemma \ref{lem:p1p2} the map $h := p_1 \circ f$ is non-zero. Now $\Cone(g) \cong \sU^i * \sT^{i+1}[-1]$ which is isomorphic to $\O_{V^i \cup V^{i i+1}}$ up to twisting by a line bundle. This means $g$ is non-zero. 

Now, as we saw in the proof of \ref{th:firststage}, $\sU^i * \sU^{i+1} \cong \O_{V^{i i+1}} \otimes \E'_i \E_{i+2}^\vee$. Meanwhile, $ V^{i i+1 i} \cap V^{i i+1} $ is cut out of $ V^{i i+1 i} $ by a section of the line bundle $ \Hom(L_i/L_{i-1}, L'_{i+2}/L'_{i+1})$, since we are imposing the condition that $ L_i \subset L'_{i+1}$.  Hence using Corollary \ref{cor:abstractnonsense3} we get
\begin{equation*}
\Hom(\sU^i * \sU^{i+1}[-1], \sU^{i i+1 i}) \cong \Hom(\O_{V^{i i+1}}[-1], \O_{V^{i i+1 i}}  \E_i  {\E'_{i+2}}^\vee) \cong \C.
\end{equation*}
Thus $h$ is the unique non-zero map. 

Similarly, $V^i \cap V^{i i+1}$ inside $V^{i i+1}$ is cut out by a section of $\Hom(L_{i+1}/L_{i}, L'_{i+2}/L'_{i+1})$ and so
\begin{equation*}
\Hom(\sU^i[-1], \sU^i * \sU^{i+1}) \cong \Hom(\O_{V^i}[-1], \O_{V^{i i+1}} \E_{i+1} {\E'_{i+2}}^\vee) \cong \C
\end{equation*}
and thus $g$ is the unique non-zero map. 

Taking $S'' = V^i$, $S' = V^{i i+1}$ and $S = V^{i i+1 i}$ we have the situation from Lemma \ref{lem:abstractnonsense4}. The commutativity of the diagram there implies that the composition $h \circ g$ is non-zero. This means $A = p_1 \circ g' \circ f' = p_1 \circ f \circ g = h \circ g$ is non-zero.


Next we show that $D$ is invertible. This time consider the following commutative diagram equipped with the natural adjunction morphisms and the inclusion $i$: 
\begin{equation*}
\begin{CD}
\sU^i[-1]\{-2\} @>i>> \sU^i * \sU^i[-2] @>{g'}>> \sU^i * \sU^{i+1} * \sU^i [-1] \\
@. @VV{\alpha}V @VVV \\
@. \sU^i[-1]\{-2\} @>{\alpha'}>> \sU^{i+1} * \sU^i \{-2\}
\end{CD}
\end{equation*}
First we claim that $\alpha' \ne 0$. As above, this is because $\Cone(\alpha') \cong \sT^{i+1} * \sU^i \{-2\}$ is isomorphic to $\O_{V^i \cup V^{i+1 i}}$ up to twisting by some line bundle. 

Next we claim that $\alpha \circ i $ is a non-zero multiple of the identity. This is because we have the diagram
\begin{equation*}
\begin{CD}
\sP^i * \sP^i_R \{-4\} @>>> \sP^i * (\sP^i_R * \sP^i) * \sP^i_R \{-4\} \cong \sP^i * \sP^i_R * (\sP^i * \sP^i_R) \{-4\} @>>> \sP^i * \sP^i_R \{-4\} \\
@VV{\cong}V @VV{\cong}V @VV{\cong}V \\
\sU^i[-1]\{-2\} @>i>> \sU^i * \sU^i[-2] @>{\alpha}>> \sU^i[-1]\{-2\}
\end{CD}
\end{equation*}
where, as in the proof of \ref{th:UiUi}, $\sP^i = \O_{X_\beta^i} \otimes \E_i'$. The top row is a composition of the adjoint maps indicated by the parentheses. The left square commutes (up to non-zero scalar) because the only non-zero map from $\sU^i[-1]\{-2\}$ to $\sU^i * \sU^i[-2] \cong \sU^i[-1]\{-2\} \oplus \sU^i[-3]$ is the inclusion map $i$ into the first factor. The right square commutes by definition. 

The composition along the top row is the identity because the composition of adjoint maps
$$\sP^i_R \rightarrow (\sP^i_R * \sP^i) * \sP^i_R \cong \sP^i_R * (\sP^i * \sP^i_R) \rightarrow \sP^i_R$$
is the identity. This can be seen as in the proof of Theorem 3.4 from \cite{kh1} to follow from the general theory of adjoint functors (one can view $\sP^i *$ and $\sP^i_R *$ as adjoint functors). Consequently, $\alpha \circ i \ne 0$ and since $\Hom(\sU^i, \sU^i) \cong \C$ it must be the identity (up to non-zero scalar). 

All this means that $\alpha' \circ \alpha \circ i \cong \alpha' \ne 0$ and hence $g' \circ i \ne 0$. Finally, since $\Hom(\sU^i[-1]\{-2\}, \sU^{i i+1 i}[-1]) = 0$, we find that $D = p_2 \circ g' \circ i \ne 0$ and hence $D$ must be invertible (since the only non-zero endomorphism of $\sU^i$ is the identity).

\end{proof}

\begin{Proposition}\label{prop:thirdstage}
$ \sU^i * \sT^{i+1} * \sT^i \cong \O_{V^{i i+1 i} \cup V^{i i+1}} \E'_i \E_{i+2}^\vee $
\end{Proposition}

\begin{proof}
Let's study the following commutative diagram where every row is a distinguished triangle:
\begin{equation*}
\begin{CD}
\sU^i[-2] @>>> \sU^i * \sU^{i+1} [-1] @>>> \sU^i * \sT^{i+1} [-1] \\
@VV{\beta}V @VV{f}V @VV{m}V \\
\sU^i * \sU^i [-1] @>>> \sU^i * \sU^{i+1} * \sU^i @>>> \sU^i * \sT^{i+1} * \sU^i \\
@VV{p}V @VV{p_1}V @VV{\phi}V \\
\sU^i[-2] @>A>> \sU^{i i+1 i} @>>> \sU^i * \sT^{i+1} * \sU^i
\end{CD}
\end{equation*}
The top row of vertical maps are adjunction maps. The maps $p$ and $p_1$ are projection maps while $\phi$ is an isomorphism. We know from the proof of Proposition \ref{prop:secondstage} that the two lower rectangles commute. We wish to compute $\sU^i * \sT^{i+1} * \sT^i \cong \Cone(m) \cong \Cone(\phi \circ m)$. 

The composition $p_1 \circ f$ is non-zero by Lemma \ref{lem:p1p2}. The composition $p \circ \beta$ can be seen to be the identity from the following commutative diagram
\begin{equation*}
\begin{CD}
\sP^i * \sP^i_L [-3] @>>> \sP^i * \sP^i_L * (\sP^i * \sP^i_L) [-3] \cong \sP^i * (\sP^i_L * \sP^i) * \sP^i_L) [-3] @>>> \sP^i * \sP^i_L [-3] \\
@VV{\cong}V @VV{\cong}V @VV{\cong}V \\
\sU^i[-2] @>{\beta}>> \sU^i * \sU^i[-1] @>p>> \sU^i[-2]
\end{CD}
\end{equation*}
where we use that $\sU^i \cong \sP * \sP_L [-1]$. The reason the right rectangle commutes is that there is a unique non-zero map from $\sU^i * \sU^i [-1] \cong \sU^i[-2] \oplus \sU^i \{-2\}$ to $\sU^i[-2]$, namely the projection $p$. The composition in the upper row is the identity because the following natural composition of adjoint maps is the identity
$$\sP^i_L \rightarrow \sP^i_L * (\sP^i * \sP^i_L) \cong (\sP^i_L * \sP^i) * \sP^i_L \rightarrow \sP^i_L.$$
Thus $p \circ \beta \ne 0$ which means it must be the identity. Subsequently we end up with the following commutative diagram
\begin{equation*}
\begin{CD}
\sU^i[-2] @>>> \sU^i * \sU^{i+1} [-1] @>>> \sU^i * \sT^{i+1} [-1] \\
@VV{p \circ \beta}V @VV{p_1 \circ f}V @VV{\phi \circ m}V \\
\sU^i[-2] @>A>> \sU^{i i+1 i} @>>> \sU^i * \sT^{i+1} * \sU^i
\end{CD}
\end{equation*}

Since $p \circ \beta = id$, direct appeal to the octahedron axiom of triangulated categories implies that $\Cone(\phi \circ m) \cong \Cone(p_1 \circ f).$ Thus 
\begin{equation} \label{eq:tri2}
\sU^i * \sT^{i+1} * \sT^i \cong \Cone(p_1 \circ f: \sU^i * \sU^{i+1} [-1] \rightarrow \sU^{i i+1 i}).
\end{equation}

As we saw in the proof of in the proof of Proposition \ref{prop:secondstage} $h := p_1 \circ f$ is the unique non-zero map. On the other hand, there is a distinguished triangle
\begin{equation} \label{eq:tri1}
\O_{V^{i i+1}}[-1] \rightarrow \O_{V^{i i+1 i}}(-V^{i i+1}) \rightarrow \O_{V^{i i+1} \cup V^{i i+1 i}}. 
\end{equation}
Since $V^{i i+1} \cap V^{i i+1 i}$ inside $V^{i i+1 i}$ is cut out by a section of the line bundle $ Hom(L_i/L_{i-1}, L'_{i+2}/L'_{i+1})$ we get 
$$\O_{V^{i i+1 i}}(-V^{i i+1}) \otimes \E'_i \E_{i+2}^\vee \cong \O_{V^{i i+1 i}} \otimes \E_i (\E'_{i+2})^\vee \E'_i \E_{i+2}^\vee \cong \sU^{i i+1 i}.$$ 
Hence tensoring the distinguished triangle (\ref{eq:tri1}) with the $ \E'_i \E_{i+2}^\vee $ gives 
\begin{equation*}
\O_{V^{i i+1}} \E'_i \E_{i+2}^\vee  [-1] \cong \sU^i * \sU^{i+1} [-1] \rightarrow \sU^{i i+1 i} \rightarrow \O_{V^{i i+1} \cup V^{i i+1 i}} \E'_i \E_{i+2}^\vee.
\end{equation*}
Comparing this triangle with (\ref{eq:tri2}) gives $\sU^i * \sT^{i+1} * \sT^i \cong \O_{V^{i i+1} \cup V^{i i+1 i}} \E'_i \E_{i+2}^\vee$
\end{proof}

We are now ready for the final step.

\begin{Proposition}\label{prop:fourthstage}
$\sT^i * \sT^{i+1} * \sT^i = \O_{Z^{i i+1 i}} $
\end{Proposition}
\begin{proof}
Using the fact that the first $ \sT^i $ is a twist, we have that
\begin{equation*}
\sT^i * \sT^{i+1} * \sT^i = \Cone( \sT^{i+1} * \sT^i [-1] \rightarrow \sU^i * \sT^{i+1} * \sT^i )
\end{equation*}
Note that this map is non-zero since it is the result of applying the invertible functor $ * \sT^{i+1} * \sT^i $ to the nonzero map $ \O_\Delta[-1] \rightarrow \sU^i $.

It is easy to see that $ \sT^{i+1} * \sT^i = \O_{Z^{i+1 i}} $.  Combining this with the last proposition we see that
\begin{equation} \label{eq:firstcone}
\sT^i * \sT^{i+1} * \sT^i = \Cone( \O_{Z^{i+1 i}}[-1] \rightarrow \O_{V^{i i+1 i} \cup V^{i i+1}} \E'_i \E_{i+2}^\vee )
\end{equation}

Now, note that $ Z^{i+1 i} \cap (V^{i i+1 i} \cup V^{i i+1}) $ is cut out of $(V^{i i+1 i} \cup V^{i i+1})$ by the condition that $ L'_i \subset L_{i+1} $ and hence by a section of the line bundle $ \Hom(L'_i/L_{i-1}, L_{i+2}/L_{i+1}) $.  Thus we see that $ \O_{V^{i i+1 i} \cup V^{i i+1}}(-Z^{i+1 i}) = \O_{V^{i i+1 i} \cup V^{i i+1}} \E'_i \E_{i+2}^\vee $.  From this we get the exact triangle
\begin{equation} \label{eq:secondcone}
\O_{V^{i i+1 i}} \E'_i \E_{i+2}^\vee \rightarrow \O_{Z^{i i+1 i}} \rightarrow  \O_{Z^{i+1 i}}
\end{equation}
By Corollary \ref{cor:abstractnonsense3}, $ \Hom(\O_{Z^{i+1}}[-1], \O_{V^{i i+1 i} \cup V^{i i+1}} \E_i' \E_{i+2}^\vee) $ is one dimensional and hence the result follows by comparing (\ref{eq:firstcone}) and (\ref{eq:secondcone}).
\end{proof}

Replacing $i+1$ by $i-1$ in all the calculations above yields $\sT^i * \sT^{i-1} * \sT^i \cong \O_{Z^{i i-1 i}}$ where
$$Z^{i i-1 i} := \{ (L_\bullet, L'_\bullet) : L_j = L'_j \text{ for } j \ne i, i-1 \}.$$
Thus 
$$\sT^i * \sT^{i+1} * \sT^i \cong \O_{Z^{i i+1 i}} \cong \O_{Z^{i+1 i i+1}} \cong \sT^{i+1} * \sT^i * \sT^{i+1}$$
proving Theorem \ref{th:impliesR3}.

\subsection{Invariance Under Vertical Isotopies From Figure \ref{f3}}\label{se:invdistant}
The invariance under vertical isotopies follows very similarly as in \cite{kh1}.  Hence we only include one proof in this paper.

\begin{Proposition}\label{prop:capscups}   We have invariance under cap-cap isotopies, i.e.
\begin{equation}
\G_{\beta}^{i+k} \circ \G_{d_{i+k}(\beta)}^i = \G_{\beta}^i \circ \G_{d_i(\beta)}^{i+k-2} 
\end{equation}
\end{Proposition}

\begin{proof}
Let $ \beta' = d_i(d_{i+k}(\beta)) = d_{i+k-2} (d_i(\beta))$.  So we are considering the equality of two functors from $ D(Y_\beta') $ to $ D(Y_\beta) $.

The kernel of $\G_\beta^{i+k} \circ \G_{d_k(\beta)}^{i}$ is given by 
$$\sG_\beta^{i+k} \ast \sG_{d_{i+k}(\beta)}^{i} = \pi_{13\ast} \left( \pi_{12}^\ast(\sG_{d_{i+k}(\beta)}^{i}) \otimes \pi_{23}^\ast(\sG_\beta^{i+k}) \right) $$

We have 
\begin{equation*}
 \sG_{d_{i+k}(\beta)}^i = \O_{X_{d_{i+k}(\beta)}^i} \otimes (\E'_i)^{\beta_{i+1}} \{-(m-1)(i-1)\}  \text{ and } \sG_\beta^{i+k} = \O_{X_\beta^{i+k}} \otimes (\E'_i)^{\beta_{i+k+1}} \{-(m-1)(i+k-1)\}. 
\end{equation*}
Let $ W = \pi_{12}^{-1}(X_{d_k(\beta)}^i) \cap \pi_{23}^{-1}(X_{\beta}^{i+k}) $.  This intersection is transverse and so we have 
\begin{equation*}
\pi_{12}^*(\O_{X_{d_k(\beta)}^i}) \otimes \pi_{23}^*(\O_{X_\beta^{i+k}}) \cong \O_W.
\end{equation*}

A routine computation shows that
\begin{equation*}
\begin{aligned}
W = \{ (L_\cdot, L'_\cdot, L''_\cdot) \in Y_{\beta'} \times Y_{d_k(\beta)} \times Y_\beta :\ &L_j = L'_j \text{ for } j \le i-1,\ L_j = zL'_{j+2} \text{ for } j \ge i-1, \\
&L'_j = L''_j \text{ for } j \le i+k-1,\ L'_j = zL''_{j+2} \text{ for } j \ge i+k-1 \}.
\end{aligned}
\end{equation*}
Hence on $ W $ there is an isomorphism of line bundles 
\begin{equation} \label{eq:lineiso2}
 \E'_i \cong \E''_i 
 \end{equation}

The fibres of $ \pi_{13} $ restricted to $ W $ are points and 
\begin{equation*}
\begin{aligned}
\pi_{13}(W) = \{ (L_\cdot, L'_\cdot) \in Y_{\beta'} \times Y_\beta :\ &L_j = L'_j \text{ for } j \le i-1,\ L_j = zL'_{j+2} \text{ for } i-1 \le j \le i + k -1,\\
&L_j = z^2L'_{j+4} \text{ for } j \ge i+k -1 \}.
\end{aligned}
\end{equation*}

Hence
\begin{equation} \label{eq:lhs2}
\begin{aligned}
{\pi_{13}}_\ast (\pi_{12}^\ast(\sG_{d_{i+k}(\beta)}^i) \otimes \pi_{23}^\ast(\sG_\beta^{i+k})) 
= \O_{\pi_{13}(W)} \otimes  (\E'_i)^{\beta_{i+1}} \otimes (\E'_{i+k})^{\beta_{i+k+1}} \{-(m-1)(2i+k-2)\}
\end{aligned}
\end{equation}
where in the second last step we used the above isomorphism (\ref{eq:lineiso2}).

Now, we want to compute $ \sG_\beta^i \ast \sG_{d_i(\beta)}^{i+k-2} $. Again the intersection $ W' := \pi_{12}^{-1}(X_{d_i(\beta)}^{i+k-2}) \cap \pi_{23}^{-1}(X_\beta^i) $ is transverse.  We have
\begin{equation*}
\begin{aligned}
W' = \{ (L_\cdot, L'_\cdot, L''_\cdot) :\ &L_j = L'_j \text{ for } j \le i+k-3,\ L_j = zL'_{j+2} \text{ for } j \ge i+k-3, \\
&L'_j = L''_j \text{ for } j \le i-1,\ L'_j = zL''_{j+2} \text{ for } j \ge i+1 \}.
\end{aligned}
\end{equation*}
Hence on $ W' $ the map $ z $ induces an isomorphism of line bundles
\begin{equation} \label{eq:lineiso3}
\E'_{i+k-2}\{2\beta_{i+k}\} \cong \E''_{i+k}.
\end{equation}

Also we see that the fibres of $ \pi_{13} $ restricted to $ W' $ are points and $ \pi_{13}(W') = \pi_{13}(W) $. Hence
\begin{equation} \label{eq:rhs2}
\begin{aligned}
{\pi_{13}}_\ast &(\pi_{12}^\ast(\sG_{d_i(\beta)}^{i+k-2}) \otimes \pi_{23}^\ast(\sG_\beta^i))\\ &= {\pi_{13}}_\ast( \O_{\pi_{12}^{-1}(X_{d_i(\beta)}^{i+k-2})} \otimes \O_{\pi_{23}^{-1}(X_\beta^i)} (\E'_{i+k-2})^{\beta_{i+k+1}}  (\E''_i)^{\beta_{i+1}} ) \{-(m-1)(2i +k -4)\} \\
&\cong {\pi_{13}}_\ast( \O_{W'}  (\E''_{i+k})^{\beta_{i+k+1}} \{-2\beta_{i+k}\beta_{i+k+1} \} (\E''_i)^{\beta_{i+1}}) \{ -(m-1)(2i +k -4)\}\\
&= \O_{\pi_{13}(W)} (\E'_i)^{\beta_{i+1}}  (\E'_{i+k})^{\beta_{i+k+1}} \{-(m-1)(2i + k-2)\}
\end{aligned}
\end{equation}
where in the second last step we have used the isomorphism (\ref{eq:lineiso3}) and in the last step we have used that $ \beta_{i+k} \beta_{i+k+1} = m-1 $.

Since (\ref{eq:lhs2}) and (\ref{eq:rhs2}) agree, we conclude that $ \sG_{\beta}^{i+k} * \sG_{d_{i+k}(\beta)}^i = \sG_{\beta}^i * \sG_{d_i(\beta)}^{i+k-2} $ which implies the desired result.
\end{proof}

\subsection{Invariance Under Pitchfork Move}\label{se:pitchfork}

\begin{Theorem}
The following pitchfork moves hold.
\begin{gather}
\T_{s_i(\beta)}^i(1) \circ \G_{s_i(\beta)}^{i+1} = \T_{s_{i+1}(\beta)}^{i+1}(2) \circ \G_{s_{i+1}(\beta)}^i \label{eq:pitch1} \\  
\T_{s_i(\beta)}^i(2) \circ \G_{s_i(\beta)}^{i+1} = \T_{s_{i+1}(\beta)}^{i+1}(1) \circ \G_{s_{i+1}(\beta)}^i \label{eq:pitch2}
\end{gather}
\end{Theorem}

\begin{proof}
The computations are all ``direct'', i.e. all intersections are transverse everywhere, all pushforwards are 1-1.  So it is just a matter of checking that the varieties and line bundles line up.  Unfortunately, we need to break into cases depending on $ \beta_i, \beta_{i+1}, \beta_{i+2} $.  Here we will present the proof of one case of (\ref{eq:pitch1})

Assume that $ \beta $ is a string with substring $ (1,1,m-1) $ starting in the $i$th position.  Let $ \beta' = s_{i+1}(\beta) $.  We have that
\begin{equation*}
\sT_\beta^i(2) = \O_{Z_\beta^i} \E_i' \E_{i+1}^\vee [-m+1]\{m+1\} \quad \sG_\beta^{i+1} = \O_{X_\beta^{i+1}} {\E'_{i+1}}^{m-1} \{-i (m-1) \}
\end{equation*}
and
\begin{equation*}
\sT_{\beta'}^{i+1}(1) = \O_{Z_{\beta'}^{i+1}} {\E'_{i+1}}^{m-2} {\E_{i+2}}^{2-m} [-m+1]\{2m-2\}  \quad \sG_{\beta'}^i = \O_{X_{\beta'}^i} {\E'_i}^{m-1} \{-(i-1)(m-1)\} 
\end{equation*}

Hence the left hand side of the first equation is
\begin{equation*}
\begin{aligned}
\sT_\beta^i(2) * \sG_\beta^{i+1} &= {\pi_{13}}_* ( \pi_{12}^* \sG_\beta^{i+1} \otimes \pi_{23}^* \sT_\beta^i(2) ) \\
&= {\pi_{13}}_*( \O_{\pi_{12}^{-1}(X_\beta^{i+1})} {\E'_{i+1}}^{m-1} \{-i(m-1)\} \otimes \O_{\pi_{23}^{-1}(Z_\beta^i)} \E_i'' {\E_{i+1}'}^\vee [-m+1]\{m+1\} )\\
&= {\pi_{13}}_*( \O_W \otimes \E_i'' {\E'_{i+1}}^{m-2}) [-m+1]\{m+1-i(m-1)\}
\end{aligned}
\end{equation*}
where
\begin{equation*}
W = \{ (L_\bullet, L'_\bullet, L''_\bullet) : zL'_{i+2} = L'_i = L_i, L'_j = L''_j \text{ for } j \ne i \}
\end{equation*}
On $ W $, we see that by Lemma \ref{lem:facts0}, $ \det(L'_{i+2}/L'_i) = \O_W\{2b_\beta^i +2m\}$.  Hence $ \E'_{i+1} = {\E'_{i+2}}^\vee \{2b_\beta^i + 2m\}$.  Also $ \E'_{i+2}= \E''_{i+2} $.  Hence we see that
\begin{equation} \label{eq:LHSp1}
\sT_{1,1,m-1}^i(2) * \sG_{1,1,m-1}^{i+1} = \O_{\pi_{13}(W)} \otimes \E_i' {\E'_{i+2}}^{2-m} [-m+1]\{m+1-i(m-1) + (m-2)(2b_\beta^i + 2m)\}
\end{equation}

For the right hand side, we first note that on $ Z_{\beta'}^{i+1}$, we have $ \E_{i+1} \E_{i+2} \cong \det( L_{i+2}/L_i) \cong \E'_{i+1} \E'_{i+2} $ by the diamond identity.  Hence we can rewrite the kernel as $ \sT_{\beta'}^{i+1}(1) \cong \O_{Z_{\beta'}^{i+1}} {\E_{i+1}}^{m-2} {\E'_{i+2}}^{2-m} [-m+1]\{2m-2\}$.  

So the right hand of the first equation is
\begin{equation*}
\begin{aligned}
\sT_{\beta'}^{i+1}(1) * \sG_{\beta'}^i &= {\pi_{13}}_* ( \pi_{12}^* \sG_{\beta'}^i \otimes \pi_{23}^* \sT_{\beta'}^{i+1}(1) ) \\
&= {\pi_{13}}_*( \O_{\pi_{12}^{-1}(X_{\beta'}^i)} {\E'_i}^{m-1} \{-(i-1)(m-1)\} \otimes \O_{\pi_{23}^{-1}(Z_{\beta'}^{i+1})} {\E'_{i+1}}^{m-2} {\E''_{i+2}}^{2-m} [-m+1]\{2m-2\} ) \\
&= {\pi_{13}}_*( \O_{W'} \otimes {\E'_i}^{m-1} {\E'_{i+1}}^{m-2} {\E''_{i+2}}^{2-m}) [-m+1]\{-(i-3)(m-1)\}
\end{aligned}
\end{equation*}
where
\begin{equation*}
W' = \{ (L_\bullet, L'_\bullet, L''_\bullet) : zL'_{i+1} = L'_{i-1} = L_{i-1}, L'_j = L''_j \text{ for } j \ne i+1 \}
\end{equation*}
On $ W' $, we see that by Lemma \ref{lem:facts0}, $ \det(L'_{i+1}/L'_{i-1}) = \O_{W'}\{2b_{\beta'}^i +2m\}$.  Hence $ \E'_{i+1} = {\E'_i}^\vee \{2b_{\beta'}^{i-1} + 2m\}$.  Also $ \E'_i= \E''_i $.  Hence we see that
\begin{equation} \label{eq:RHSp1}
\sT_{1,1,m-1}^i(2) * \sG_{1,1,m-1}^{i+1} = \O_{\pi_{13}(W')} \otimes \E_i' {\E'_{i+2}}^{2-m}  [-m+1]\{-(i-3)(m-1) + (m-2)(2b_{\beta'}^{i-1} + 2m) \}.
\end{equation}

Since $ b_\beta^i - b_{\beta'}^{i-1}  = 1 $ we have
\begin{equation*}
m+1-i(m-1) + (m-2)(2b_\beta^i + 2m)= -(i-3)(m-1) + (m-2)(2b_{\beta'}^{i-1} + 2m) ,
\end{equation*}
and so the equivariant shifts in (\ref{eq:LHSp1}) and (\ref{eq:RHSp1}) are the same. Since $ \pi_{13}(W) = \pi_{13}(W') $, we see that (\ref{eq:LHSp1}) and (\ref{eq:RHSp1}) agree and so the first pitchfork equation (\ref{eq:pitch1}) holds.

\end{proof}

The pitchfork identities involving cups are obtained from those above involving caps by taking the left adjoint -- thus their validity is almost immediate. 

\section{Tangle graphs and relation to Khovanov-Rozansky Homology} \label{se:graphs}

The goal of this section is to extend our invariant to tangle graphs.  In particular, we will show that on closed crossingless graphs, our invariant is determined by the MOY relations.  This allows us to prove that our knot homology theory is a catgorification of the Reshetikhin-Turaev invariant.  We also conjecture a relation to Khovanov-Rozansky homology.

\subsection{Reshetikhin-Turaev invariants of tangle graphs}

A $(\beta, \beta')$ {\bf tangle graph} is an isotopy class of planar oriented tangle diagram with edges labelled from $ \{1,2,m-2,m-1\} $ which also contains trivalent vertices as shown in Figure \ref{f:tri}.   We will also insist that each cup, cap and crossing involve only the $1, m-1$ edges (the cup, cap requirement is just for simplicity).  A tangle graph without crossings we call a {\bf crossingless tangle graph} (such diagrams are also known as webs).  We usually distinguish the edges labelled $2, m-2$ by drawing them with a thick line.  As usual, arrows pointing down correspond to $1,2$ strands and arrows pointing up correspond to $m-2, m-1$ strands.  We have basic tangle graphs, $ g_\beta^i, f_\beta^i, t_\beta^i(l) $ as before along with two extra $ p_\beta^i, q_\beta^i$ corresponding to trivalent vertices.  Specifically, $ p_\beta^i $ (which is only defined when $ \beta_i = \beta_{i+1} \in \{1, m-1 \} $) is a $(d_i(\beta), \beta) $ tangle with a trivalent vertex in the $i$th column and $ q_\beta^i $ is a $ (\beta, d_i(\beta))$ tangle with a trivalent vertex in the $i$th column.

\begin{figure}
\begin{center}
\psfrag{b_i}{$\beta_i$}
\psfrag{b_i+1}{$\beta_{i+1}$}
\psfrag{b_i+b_i+1}{$\beta_i + \beta_{i+1}$}
\puteps[0.5]{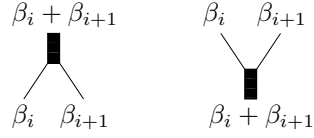}
\end{center}
\caption{The two possible trivalent vertices.}\label{f:tri}
\end{figure}

Such tangle graphs admit a Reshetikhin-Turaev \cite{RT} invariant, where we associate to the trivalent vertex a fixed map from $ V_{\omega_1} \otimes V_{\omega_1} \rightarrow V_{\omega_2} $ (and corresponding dual versions).  In the notation of \cite{RT}, the trivalent vertex is a ``coupon''.  The RT invariant of the basic graphs obeys the identities 
\begin{equation} \label{eq:resolve}
\psi(t_\beta^i(1)) = q^m(q^{-1} - \psi(p_\beta^i q_\beta^i)) \quad
\psi(t_\beta^i(2)) = q^{-m}( q - \psi(p_\beta^i q_\beta^i))
\end{equation}
in the case that $ \beta_i = \beta_{i+1} \in \{1, m-1 \} $.  Thus the Restikhin-Turaev invariant of a tangle graph with a crossing is determined by the RT invariants of the two tangle graphs made by resolving the crossing as shown in Figure \ref{f:resolution}.

\begin{figure}
\begin{center}
\psfrag{0}{$0$}
\psfrag{1}{$1$}
\psfrag{-1}{$-1$}
\puteps[0.5]{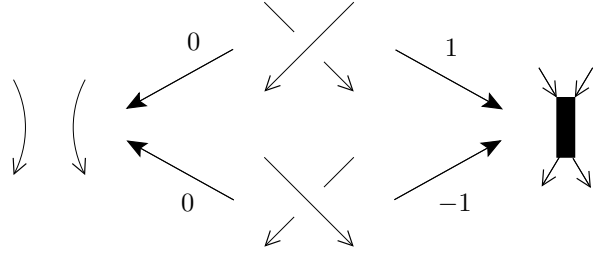}
\end{center}
\caption{The two possible resolutions of a crossing.}\label{f:resolution}
\end{figure}

In particular, the RT invariant of link $K $ equals a sum over RT invariants of (0,0) crossingless graphs (which we call \textbf{closed graphs}) made from resolutions of $ K $.  Hence it is interesting to consider the RT invariants of these closed graphs.  In the case of a closed graph $ \psi(\Gamma)$ is a $\C[q,q^{-1}] $ linear map  $ \C[q, q^{-1}] \rightarrow \C[q,q^{-1}] $ and so can just be regarded as a polynomial.  It is known that for any closed graph $ \Gamma $, $\psi(\Gamma) = P_\Gamma(q)$ is a polynomial whose coefficients are non-negative integers.  In \cite{MOY}, Murakami, Ohtsuki, and Yamada found the following relations which are satisfied by the RT invariants of closed graphs (see Figure \ref{f:MOY}). 
\begin{enumerate}
\item[(ia)] $\psi(f_\beta^i \circ g_\beta^i) =  [m]_q$ 
\item[(ib)] $\psi(f_\beta^{i + 1} \circ p_\beta^i \circ q_\beta^i \circ g_\beta^{i + 1}) =  [m-1]_q = \psi( f_{\beta'}^i \circ p_{\beta'}^{i+1} \circ q_{\beta'}^{i+1}) \circ g_{\beta'}^i)$
 where $\beta$ and $\beta'$ have strings $(1,1,m-1)$ and $(m-1,1,1)$ starting in the $i$th slot
\item[(iia)] $\psi(q_\beta^i \circ p_\beta^i) = [2]_q$
\item[(iib)] $\psi(f_\beta^i \circ u_\beta^{i-1} \circ u_\beta^{i+1} \circ g_\beta^i) = \psi(g_\beta^i \circ f_\beta^i) + [m-2]_q$
\item[(iii)] $\psi(u_\beta^{i+1} \circ u_\beta^i \circ u_\beta^{i+1}) + \psi(u_\beta^i) = \psi(u_\beta^i \circ u_\beta^{i+1} \circ u_\beta^i) + \psi(u_\beta^{i+1})$ where $\beta$ has a string $(1,1,1)$ or $(m-1,m-1,m-1)$ starting in the $i$th slot
\end{enumerate}
Here $u_\beta^i = p_\beta^i \circ q_\beta^i $ and $ [n]_q = (q^n - q^{-n}) / (q - q^{-1}) $.

The following result shows that these relations determine $ \psi(\Gamma) $.  
\begin{Lemma} \label{lem:MOY}
Suppose that $ \Gamma \mapsto \phi(\Gamma) \in \Z[q,q^{-1}] $ is a polynomial invariant of closed graphs which statisfies the (de-categorified) MOY skein relations.  Then $ \phi(\Gamma) = \psi(\Gamma) $ for any closed graph $\Gamma$.
\end{Lemma}
\begin{proof}
The result is ``folklore'' but the authors do not know a complete proof in the literature.  A proof of it in the case of ``braid-like'' graphs may be found in section 4.2 of \cite{Ra}.  The proof for general graphs can be deduced from this case with some effort.

Alternatively, one has the following argument pointed out to us by a referee. Since $ \phi $ satisfies the MOY skein relations it can be used to define a link invariant (also denoted $ \phi $) by (\ref{eq:resolve}).  The fact that this gives a link invariant follows from the proof of Theorem 3.1 in \cite{MOY}.  This link invariant $ \phi $ will satisfy the $\sl$ skein relations (see Theorem 3.2 in \cite{MOY}).  This link invariant $\phi $ is determined by the skein relations and its value on the trivial knot. Thus $\phi$ and $\psi$ agree on links. 

Now, we claim that this implies that they also agree on graphs.  To see this, for each graph $ \Gamma $, construct a link $K $ where all thick edges (those labelled $2$ or $m-2$) are replaced by crossings (arbitrarily over or under).  Then by (\ref{eq:resolve}), $ \phi(K) = c\phi(\Gamma) + \text{terms which involve graphs with fewer thick edges}$ (and simlarly for $ \psi$), where $ c \ne 0$.  By induction on the number of thick edges, the result follows.
\end{proof}

For categorified theories, one expects that (\ref{eq:resolve}) becomes categorified to a cone and that the MOY relations become categorified.  This is what happens with Khovanov-Rozansky's theory.

\subsection{Functors from tangle graphs}
Now, we will extend our invariant $ \Psi $ to tangle graphs.  Given any $(\beta, \beta') $ tangle graph $ \Gamma $, we will assign a functor $ \Psi(\Gamma) : D_\beta \rightarrow D_{\beta'} $, by assigning such a functor to basic tangle graphs.  For caps, cups, and crossings we use the same functor as before.   For the trivalent vertices recall that if $\beta_i = \beta_{i+1} \in \{1, m-1 \} $ we have the diagram
\begin{equation*}\begin{CD} X_\beta^i @>i>> Y_\beta \\ @VqVV \\ Y_{d_i(\beta)} \end{CD} \end{equation*}
where $q$ is a ${\mathbb P}^1$ bundle and $i$ is a codimension one embedding. To the trivalent vertices $p_\beta^i$ and $q_\beta^i$ (where $\beta_i = \beta_{i+1} \in \{1,m-1\}$) in Figure \ref{f:tri} we assign the FM transforms with respect to the kernels
$$\sP_\beta^i := \O_{X_\beta^i} \otimes \E_i' \in D(Y_{d_i(\beta)} \times Y_\beta)$$
and
$$\sQ_\beta^i := \O_{X_\beta^i} \otimes \E_{i+1}^\vee \{1\} \in D(Y_\beta \times Y_{d_i(\beta)})$$
respectively. Notice that $\sP_\beta^i$ agrees with the notation of Lemma \ref{lem:Padj} and consequently $(\sQ_\beta^i)_R \cong \sP_\beta^i [1]\{-1\}$ while $(\sP_\beta^i)_R \cong \sQ_\beta^i [-1]\{1\}$. Also, by Theorem \ref{th:kerneltwist} the kernel $\sT_\beta^i(2)$ for the crossing is (up to a shift) the spherical twist in $\Phi_{\sP_\beta^i}$.  This means that if $ \Gamma $ is a tangle graph containing a crossing of type 2 and if $ \Gamma_1, \Gamma_2 $ are the two tangle graphs made by the resolutions of this crossing, then there is an exact triangle
\begin{equation} \label{eq:exacttri}
\Psi(\Gamma_1)[-m+1]\{m-1\} \rightarrow \Psi(\Gamma) \rightarrow \Psi(\Gamma_2)[-m+1]\{m\}
\end{equation}
(here exact triangle of functors really means exact triangle of the underlying FM kernels).  There is a similar result for crossings of type 1.  

If $ \Gamma $ is a $(0,0) $ crossingless graph then $\Psi(\Gamma): D(pt) \rightarrow D(pt) $. Such a functor is determined by where it sends $\C$, so it can be thought of as a bigraded vector space which we denote by $H^{i,j}_{\text{alg}}(\Gamma) $.  The main result of this section is that $H^{i,j}_{\text{alg}}(\Gamma)$ carries the same information as $\psi(\Gamma)$. 

\begin{Theorem} \label{th:diag}
The bigraded vector space $ H^{i,j}_{\text{alg}}(\Gamma)$ is concentrated on the anti-diagonal, namely it is zero unless $i+j=0$. Moreover, the graded Poincar\'e polynomial of $H^{i,j}_{\text{alg}}(\Gamma)$ satisfies
\begin{equation*}
\sum_{i,j} t^i u^j \dim(H^{i,j}_{\text{alg}}(\Gamma)) = \psi(\Gamma)|_{q = tu^{-1}}
\end{equation*}
so that the graded Euler characterstic $\chi_{\text{alg}}(\Gamma) := \sum_{i,j} (-1)^i (-q)^j \dim(H^{i,j}_{\text{alg}}(\Gamma))$ equals $\psi(\Gamma)$. 
\end{Theorem}

We obtain the following corollary. 

\begin{Corollary} \label{th:Euchar}
Let $ L $ be a link.  Then the graded Euler characteristic of $ H^{i,j}_{\text{alg}}(L) $ is $ \psi(L) $.
\end{Corollary}
Note that here and above, we are using $-q$ as the variable in the Euler characteristic.

\begin{proof}
This follows from combining the theorem with (\ref{eq:exacttri}) and (\ref{eq:resolve}).

The easiest way to see this is to use some of the notation of section \ref{se:Kgroups}.  Using that notation, the statement of the Corollary is equivalent to  $ [\Psi(L)] = \psi(L) $ for all $ L $.  On the other hand, from the theorem we have that $ [\Psi(\Gamma)] = \psi(\Gamma) $ for all closed graphs $\Gamma$.  

We also know that $ [\T_\beta^i(2)] = q^{-m}(q - [\P_\beta^i][\Q_\beta^i]) $ from (\ref{eq:exacttri}).  Since this equality matches that in (\ref{eq:resolve}), we see that $[\Psi(L)] $ is determined by $ [\Psi(\Gamma)] $ for closed graphs $ \Gamma $ in the same way that $ \psi(L) $ is determined by the $\psi(\Gamma) $.  Since $ [\Psi(\Gamma)] = \psi(\Gamma) $ for all closed graphs, the result follows.
\end{proof}

Now we state our conjecture relating our knot homology theory to that of Khovanov-Rozansky.  The evidence for this is Theorem \ref{th:diag} and the exact triangle (\ref{eq:exacttri}).

\begin{Conjecture} \label{th:KRconj}
Let $ L $ be a link, then
\begin{equation*}
H^{i,j}_{\text{alg}}(L) = H^{i+j,j}_{\text{KR}}(L)
\end{equation*}
\end{Conjecture}

\subsection{Verification of the MOY relations}

This section is devoted to proving Theorem \ref{th:diag}. Our main tool is the (categorified) Murakami-Ohtsuki-Yamada \cite{MOY} skein relations depicted in Figure \ref{f:MOY}.  

\begin{figure}
\begin{center}
\psfrag{0}{ia.}
\psfrag{I}{ib.}
\psfrag{II}{iia.}
\psfrag{III}{iii.}
\psfrag{IV}{iib.}
\psfrag{=}{$=$}
\psfrag{+}{$\oplus$}
\psfrag{otimes m-1}{$\otimes H^\star(\mathbb{P}^{m-1})$}
\psfrag{otimes m-2}{$\otimes H^\star(\mathbb{P}^{m-2})$}
\psfrag{otimes m-3}{$\otimes H^\star(\mathbb{P}^{m-3})$}
\psfrag{otimes 1}{$\otimes H^\star(\mathbb{P}^{1})$}
\puteps[0.6]{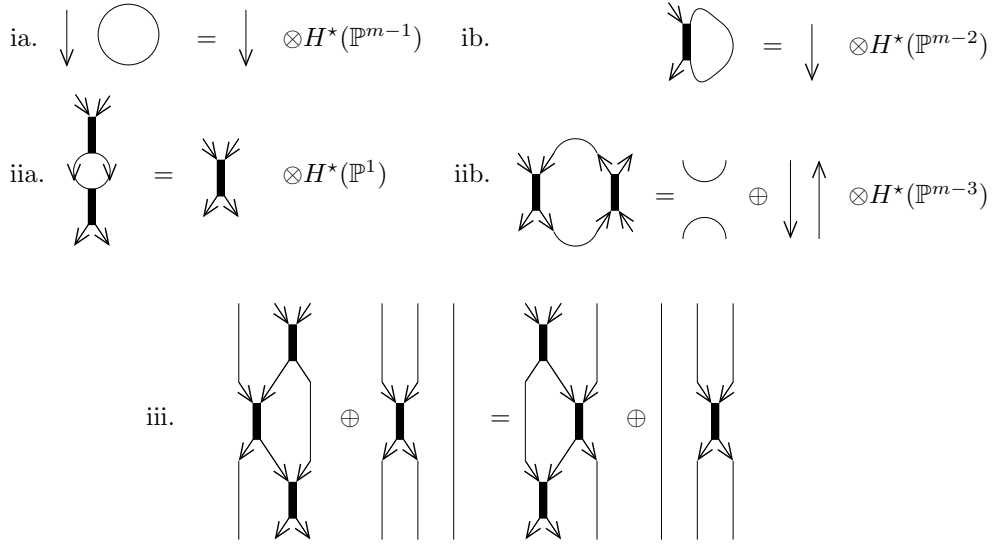}
\end{center}
\caption{The MOY relations where $H^\star(\mathbb{P}^n) := \oplus_{j \ge 0} H^j(\mathbb{P}^n, \C) [-j+n]\{j-n\}$.}\label{f:MOY}
\end{figure}

\begin{Proposition} \label{prop:MOYrels} The following categorified version of the MOY identities hold:
\begin{itemize}
\item[(ia)] $\F_\beta^i \circ \G_\beta^i \cong  (\cdot) \otimes H^{\star}(\mathbb{P}^{m-1})$ 
\item[(ib)] $\F_\beta^{i + 1} \circ (\P_\beta^i \circ \Q_\beta^i) \circ \G_\beta^{i + 1} \cong (\cdot) \otimes H^\star(\mathbb{P}^{m-2}) \cong \F_{\beta'}^i \circ (\P_{\beta'}^{i+1} \circ \Q_{\beta'}^{i+1}) \circ \G_{\beta'}^i$
 where $\beta$ and $\beta'$ have strings $(1,1,m-1)$ and $(m-1,1,1)$ starting in the $i$th slot
\item[(iia)] $\Q_\beta^i \circ \P_\beta^i \cong (\cdot) \otimes H^\star(\mathbb{P}^1)$
\item[(iib)] $\F_\beta^i \circ \U_\beta^{i-1} \circ \U_\beta^{i+1} \circ \G_\beta^i \cong \G_\beta^i \circ \F_\beta^i \oplus (\cdot) \otimes H^\star(\mathbb{P}^{m-3})$
\item[(iii)] $\U_\beta^{i+1} \circ \U_\beta^i \circ \U_\beta^{i+1} \oplus \U_\beta^i = \U_\beta^i \circ \U_\beta^{i+1} \circ \U_\beta^i \oplus \U_\beta^{i+1}$ where $\beta_i = \beta_{i+1} = \beta_{i+2}$.
\end{itemize}
where $H^\star(\mathbb{P}^n) := \oplus_{j \ge 0} H^j(\mathbb{P}^n, \C) [-j+n]\{j-n\}$ and $\U_\beta^i := \P_\beta^i \circ \Q_\beta^i$.  
\end{Proposition}
\begin{proof}
(ia) This is the statement of Corollary \ref{cor:circle}.

(ib) We must prove that $\sF_\beta^{i + 1} * (\sP_\beta^i * \sQ_\beta^i) * \sG_\beta^{i + 1} \cong \O_\Delta \otimes_\C H^\star(\mathbb{P}^{m-2}) \cong \sF_{\beta'}^i * (\sP_{\beta'}^{i+1} * \sQ_{\beta'}^{i+1}) * \sG_{\beta'}^i$. We prove the first isomorphism as the second one follows similarly. 

This was essentially shown in Theorem \ref{th:RI} in order to prove Reidemeister move I. There we saw that the cone of $\sF_\beta^{i+1} * \sP_\beta^i * {\sP_\beta^i}_R * \sG_\beta^{i+1} \rightarrow \sF_\beta^{i+1} * \sG_\beta^{i+1}$ is $\O_\Delta [m-1]\{-m+1\}$. Now $\sF_\beta^{i+1} * \sP_\beta^i * {\sP_\beta^i}_R * \sG_\beta^{i+1}$ is supported in degrees $-m+3,-m+5, \dots, m-1$ so that 
$$\Hom(\O_\Delta[m-2]\{-m+1\}, \sF_\beta^{i+1} * \sP_\beta^i * {\sP_\beta^i}_R * \sG_\beta^{i+1}) = 0$$ 
and hence 
$$\sF_\beta^{i+1} * \sG_\beta^{i+1} \cong \O_\Delta [-m+1]\{m-1\} \oplus \sF_\beta^{i+1} * \sP_\beta^i * {\sP_\beta^i}_R * \sG_\beta^{i+1}.$$
But by Corollary \ref{cor:circle} $\sF_\beta^i * \sG_\beta^i \cong \O_\Delta \otimes_\C H^\star(\p, \C)$. Thus $\sF_\beta^{i+1} * \sP_\beta^i * {\sP_\beta^i}_R * \sG_\beta^{i+1}$ must be isomorphic to $\O_\Delta \otimes_\C H^\star(\mathbb{P}^{m-2})[-1]\{1\}$. The result follows since $(\sP_\beta^i)_R = \sQ_\beta^i [-1]\{1\}$. 

\begin{Remark} This isomorphism can be seen geometrically using Lemma \ref{lem:updown}. This is because ${\sP_\beta^i}_R * \sG_\beta^{i+1}$ is $i_*q^*(\cdot)$ via a diagram as in Lemma \ref{lem:updown} where the fibres of $q$ are $\mathbb{P}^{m-2}$s. Then one just needs to show that the resulting complex is formal (c.f. proof of Corollary \ref{cor:circle}). 
\end{Remark}

(iia) This relation is related to the proof of Reidemeister move II for a $(1,1)$-type crossing. Taking $F \cong \mathbb{P}^1$ and $s=2$ in Lemma \ref{lem:updown} we get an exact triangle 
$$\O_\Delta [1]\{-1\} \rightarrow \sQ_\beta^i * \sP_\beta^i \rightarrow \O_\Delta[-1]\{1\}$$
where we used again that $\sQ_\beta^i \cong (\sP_\beta^i)_R [1]\{-1\}$. Starring on the left and right gives the exact triangle
$$\sP_\beta^i * \sQ_\beta^i [1]\{-1\} \rightarrow \sP_\beta^i * \sQ_\beta^i * \sP_\beta^i * \sQ_\beta^i \rightarrow \sP_\beta^i * \sQ_\beta^i [-1]\{1\}.$$ 
Finally we use the map 
$$\sP_\beta^i * \sQ_\beta^i [-1]\{1\} \rightarrow \sP_\beta^i * {\sP_\beta^i}_R * \sP_\beta^i * \sQ_\beta^i [-1]\{1\} \cong \sP_\beta^i * \sQ_\beta^i * \sP_\beta^i * \sQ_\beta^i$$
coming from adjunction on the ${\sP_\beta^i}_R * \sP_\beta^i$ pair to split the exact sequence. 

(iib) This is the hardest relation to prove since it is basically equivalent to the Reidemeister move II for a $(1,m-1)$-type crossing. We will use the identity coming from Figure \ref{f:MOY-IV}, namely
$$\sF_\beta^i * \sT_\beta^{i-1}(2) * \sT_\beta^{i+1}(1) * \sG_\beta^i \cong \sG_\beta^i * \sF_\beta^i.$$

\begin{figure}
\begin{center}
\psfrag{=}{$=$}
\puteps[0.5]{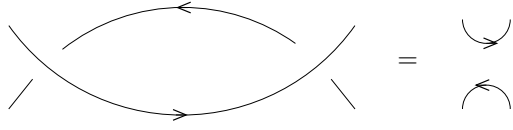}
\end{center}
\caption{Sideways Reidemeister II move.}\label{f:MOY-IV}
\end{figure}

We now replace each crossing on the left side by its corresponding cone and simplify:
\begin{align*}
LS &\cong
\sF_\beta^i * \Cone(\sP_\beta^{i-1}*{\sP_\beta^{i-1}}_R \rightarrow \O_\Delta) * \sT_\beta^{i+1}(1) * \sG_\beta^i [-m+1]\{m-1\} \\
&\cong \Cone(\sF_\beta^i * \sU_\beta^{i-1} * \sT_\beta^{i+1}(1) * \sG_\beta^i [-1]\{1\} \rightarrow \sF_\beta^i * \sT_\beta^{i+1}(1) * \sG_\beta^i) [-m+1]\{m-1\} \\
&\cong \Cone(\sF_\beta^i * \sU_\beta^{i-1} * \Cone(\O_\Delta \rightarrow \sP_\beta^{i+1} * {\sP_\beta^{i+1}}_L)[m-2]\{-m+1\} * \sG_\beta^i [-1]\{1\} \rightarrow \O_\Delta) [-m+1]\{m-1\} \\
&\cong \Cone(\Cone(\sF_\beta^i * \sU_\beta^{i-1} * \sG_\beta^i [-1]\{1\} \rightarrow \sF_\beta^i * \sU_\beta^{i-1} * \sU_\beta^{i+1} * \sG_\beta^i)[-1] \rightarrow \O_\Delta [-m+1]\{m-1\}) \\
&\cong \Cone(\Cone(\O_\Delta \otimes_\C H^\star(\mathbb{P}^{m-2}) [-1]\{1\} \rightarrow \sF_\beta^i * \sU_\beta^{i-1} * \sU_\beta^{i+1} * \sG_\beta^i)[-1] \rightarrow \O_\Delta [-m+1]\{m-1\})
\end{align*}
where we used Reidemeister I to obtain the third line and MOY relation (ib) to obtain the last line. Thus we have the following exact triangle
$$\O_\Delta [-m+1]\{m-1\} \rightarrow \sG_\beta^i * \sF_\beta^i \rightarrow \Cone(\O_\Delta \otimes_\C H^\star(\mathbb{P}^{m-2}) [-1]\{1\} \rightarrow \sX).$$
where $\sX := \sF_\beta^i * \sU_\beta^{i-1} * \sU_\beta^{i+1} * \sG_\beta^i$ is what we need to calculate. We can rewrite this as the exact triangle
$$ \Cone(\O_\Delta [-m+1]\{m-1\} \rightarrow \sG_\beta^i * \sF_\beta^i)[-1] \xrightarrow{\alpha} \O_\Delta \otimes_{\C} H^\star(\mathbb{P}^{m-2})[-1]\{1\} \sX.$$ 
Now consider the long exact sequence in cohomology. The first claim is that $\alpha$ induces a non-zero map $\O_\Delta \{m-1\} \rightarrow \O_\Delta \{m-1\}$ in degree $m-1$. If this map were zero than $\sX$ would have a copy of $\O_\Delta$ in degree $m-1$. Then if one considers $\sF_\beta^i * \sX * \sG_\beta^i$ one would find that it is non-zero in homological degree $2m-3$. But this is not the case because we can calculate
$$\sF_\beta^i * \sX * \sG_\beta^i \cong H^\star(\mathbb{P}^1) \otimes_\C H^\star(\mathbb{P}^{m-2}) \otimes_\C H^\star(\mathbb{P}^{m-1})$$
using MOY relations (ia), (ib) and (iia). Thus $\alpha$ must induce an isomorphism in degree $m-1$. Subsequently, $\sX$ is isomorphic to a cone 
$$\Cone(\sG_\beta^i * \sF_\beta^i [-1] \xrightarrow{\alpha'} \O_\Delta \otimes_\C H^\star(\mathbb{P}^{m-3})).$$
But 
\begin{align*}
\Hom(\sG_\beta^i * \sF_\beta^i [-1], \O_\Delta [j]\{-j\}) 
&\cong \Hom({\sF_\beta^i}_L [m-1]\{-m+1\} * \sF_\beta^i [-1], \O_\Delta [j]\{-j\}) \\
&\cong \Hom(\sF_\beta^i, \sF_\beta^i [-m+2+j] \{m-1-j\})
\end{align*}
which is zero for $j=-m+3, \dots, m-3$ since $\sF_\beta^i$ is a sheaf. Thus $\alpha' = 0$ and the result follows. 

(iii) Let $ \sU_\beta^i = \sP_\beta^i * \sQ_\beta^i $.  This follows from Proposition \ref{th:firststage} used in proving Reidemeister move III. The usual computation shows that, in the notation of section \ref{section:RIII}, $\sU^i_\beta \cong \sU^i \{1\}$. Consequently, we need to show,
$$\sU^i * \sU^{i+1} * \sU^i \{3\} \oplus \sU^{i+1} \{1\} \cong \sU^{i+1} * \sU^i * \sU^{i+1} \{3\} \oplus \sU^i \{1\}.$$
By \ref{th:firststage} we have $\sU^i * \sU^{i+1} * \sU^i \{3\} \cong \sU^{i i+1 i} \{3\} \oplus \sU^i \{1\}$. Similarly, $\sU^{i+1} * \sU^i * \sU^{i+1} \{3\} \cong \sU^{i+1 i i+1} \{3\} \oplus \sU^{i+1} \{1\}$. The result now follows since $\sU^{i i+1 i} \cong \sU^{i+1 i i+1}$.
\end{proof}

Now we are ready to prove the main theorem.
\begin{proof}[Proof of theorem \ref{th:diag}]
Consider the polynomial $ R(\Gamma) := \sum_{i,j} y^i z^{i+j} \dim H^{i,j}_{\text{alg}} (\Gamma) $.  If we specialize $ z = \xi \in \C $, $ R(\Gamma)|_{z= \xi} $ satisfies the MOY relations (Proposition \ref{prop:MOYrels}) and hence by Lemma \ref{lem:MOY}, we have that $ R(\Gamma)|_{z = \xi} = \psi(\Gamma)|_{q = y} $. Hence we see that $ R(\Gamma) = \psi(\Gamma)|_{q=y} $ and does not depend on $ z $. If we let $y=tu^{-1}$ and $z=u$ the desired result follows.
\end{proof}

\section{Grothendieck groups} \label{se:Kgroups}

The goal of this section is to examine what happens when we "de-categorify", i.e. consider our construction on the level of Grothendieck groups.  We denote $ V_\beta = V_{\omega_{\beta_1}} \otimes \cdots \otimes V_{\omega_{\beta_n}} $. Ideally, we would like to prove the following conjecture.

\begin{Conjecture} \label{th:Kconj}
For each $ \beta $, there exists an isomorphism $ K(Y_\beta) \rightarrow V_\beta $ such that for each tangle $ T $, the following diagram commutes
\begin{equation}
\begin{CD}
K(Y_\beta) @>>> V_\beta \\
@V[\Psi(T)]VV @VV\psi(T)V \\
K(Y_{\beta'}) @>>> V_{\beta'}
\end{CD}
\end{equation}
\end{Conjecture}
Here $ K(Y_\beta) $ denotes the complexified Grothendieck group of the category $ D(Y_\beta) $ and $ [\Psi(T)] $ denotes the map that $ T $ induces on this Grothendieck group.

In this section we give evidence for this conjecture.  First we check that $K(Y_\beta)$ has the right dimension in Proposition \ref{th:Kbasis}. Next in Propositions \ref{th:GFmap} and \ref{th:Tmap} we calculate the action of basic tangles on a chosen basis of $K(Y_\beta)$. From this we construct an isomorphism $ K(Y_\beta) \rightarrow V_\beta $ and prove the commutativity of the diagram when $ q = 1$ (see Theorem \ref{th:Kconjq1} and Remark \ref{re:Kconj}).  Notice that the commutativity of the diagram when $ T $ is a link (for arbitary $ q $) follows from Corollary \ref{th:Euchar}.

\subsection{Equivariant K-theory of $ Y_\beta $}
We will now define a set of vector bundles which form a basis for $K(Y_\beta)$.  Fix $ \beta $.  For each $ i $ and each $ 0 \le k \le m-1 $, define a K-theory class $ W_{i,k} $ according to
\begin{equation*}
W_{i,k} = \begin{cases} 
q^{(i-1)k} [(L_i/L_{i-1})^k] \text{ if } \beta_i = 1 \\
q^{(i-1)k} [\Lambda^k(L_i/L_{i-1})] \text{ if } \beta_i = m-1
\end{cases}
\end{equation*}
The reader who is curious about this choice should note that we are just applying all possible Schur functors corresponding to partitions which fit in a $ \beta_i \times (m- \beta_i) $ box (the factor of $q $ has no explanation other than to make later formulas work out better).

\begin{Proposition} \label{th:Kbasis}
The collection 
\begin{equation*}
\Big\{ \prod_i W_{i, \delta(i)} \Big\}
\end{equation*}
where $ \delta $ ranges over all functions from $ \{1, \dots, n\} $ to $ \{0, \dots, m-1 \} $, gives a basis for $K(Y_\beta) $.
\end{Proposition}

\begin{proof}
We proceed by induction on $n $. When $ n=1 $, the result follows immediately.

When $ n\ge 1$, the map $ Y_\beta \rightarrow Y_{(\beta_1, \dots, \beta_{n-1})} $ is a $ \p $ bundle which is formed as the projectivization of an equivariant vector bundle on $ Y_\beta $.  Hence the Projective Bundle Theorem of \cite{CG} (Theorem 5.2.31) applies and shows that $ K(Y_\beta) $ is a free module over $ K(Y_{(\beta_1, \dots, \beta_{n-1})}) $ with basis $ \{ (L_n/L_{n-1})^k \}_{0 \le k \le m-1} $ if $ \beta_n = 1 $ and basis $ \{ (z^{-1}L_{n-1}/L_n)^k \}_{0 \le k \le m-1 } $ if $ \beta_n = m-1 $.  So if $ \beta_n = 1 $, we are done.  If $ \beta_n = m-1 $, then using the short exact sequence
\begin{equation*}
0 \rightarrow L_n/L_{n-1} \rightarrow z^{-1}L_{n-1}/L_{n-1} \rightarrow z^{-1}L_{n-1}/L_n \rightarrow 0
\end{equation*}
shows that powers of $ z^{-1} L_{n-1}/L_n $ may be expressed as linear combinations of the classes of wedge powers of $ L_n/L_{n-1} $.
\end{proof}

\subsection{Maps induced in K-theory}
Now, we will compute the maps on K-theory induced by the functors associated to basic tangles.  We use the convention that $ [\{1\}] = -q^{-1} $.

\begin{Proposition} \label{th:GFmap}
We have
\begin{align*}
[\G_\beta^i]&(\prod_j W_{j, \delta(j)}) \\
&= (-1)^{(i-1)(m-1)} \prod_{j < i} W_{j, \delta(j)} (W_{i, m-1} W_{i+1, 0} - q^1W_{i, m-2} W_{i+1, 1} \dots \pm q^{(m-1)} W_{i, 0} W_{i+1, m-1} ) \prod_{j \ge i} W_{j+2, \delta(j)} \\
[\F_\beta^i]&(\prod_j W_{j, \delta(j)})
= \begin{cases} (-1)^{i(m-1)} (-q)^{-\delta(i)} \prod_{j < i} W_{j, \delta(j)} \prod_{j >i} W_{j+2, \delta(j+2)} \ \text{ if $\delta(i) + \delta(i+1) = m-1 $ } \\
0 \ \text{ otherwise }
\end{cases}
\end{align*}
\end{Proposition}

\begin{proof}
First, consider $ [\G_\beta^i] $.  By definition $ \G_\beta^i(\sF) = i_* q^* \sF \otimes \E_i^{\beta_{i+1}} \{-(i-1)(m-1)\} $.  Note that if $ j < i $, then $ L_j/L_{j-1} \cong L'_j/L'_{j-1} $ on $ X_\beta^i $, while if $ j \ge i $, then $ L_j/L_{j-1} \{2\} \cong L'_{j+2}/L'_{j+1} $ on $ X_\beta^i $.  Hence we see that 
\begin{equation} \label{eq:Gonbasis}
[\G_\beta^i](\prod_j W_{j, \delta(j)}) = \prod_{j < i} W_{j, \delta(j)} \prod_{j \ge i} W_{j+2, \delta(j)} [\O_{X_\beta^i} \otimes \E_i^{\beta_{i+1}} \{-(i-1)(m-1)\}] .
\end{equation}
Hence it will be neccesary to examine $ [\O_{X_\beta^i}] $ in terms of our basis.  $ X_\beta^i $ is defined by the vanishing of a section of $ \mathcal{V} = \Hom(L_{i+1}/L_i, L_i/L_{i-1}\{2\}) $ and hence by the Kozsul resolution of $ \O_{X_\beta^i} $ we have that 
\begin{equation*}
[\O_{X_\beta^i}] = [\O_{Y_\beta}] - [\mathcal{V}^\vee] + [\Lambda^2 \mathcal{V}^\vee] - \dots \pm [\Lambda^{m-1} \mathcal{V}^\vee]
\end{equation*}

Assume for the moment that $ \beta_i = 1 $ which implies $ \beta_{i+1} = m-1 $.  Then equation (\ref{eq:Gonbasis}) gives that
\begin{eqnarray*}
[O_{X_\beta^i} \E_i^{\beta_{i+1}}] &=& [(L_i/L_{i-1})^{m-1}] - q^2[(L_i/L_{i-1})^{m-2} L_{i+1}/L_i] \dots \pm q^{2(m-1)} [\Lambda^{m-1}(L_{i+1}/L_i)] \\
&=& q^{-(i-1)(m-1)} \left( W_{i,m-1} W_{i+1,0} - qW_{i,m-2} W_{i+1,1} \dots \pm q^{m-1}W_{i,0}W_{i+1,m-1} \right)
\end{eqnarray*}
and so the result follows.  

If $ \beta_i = m-1 $, the result follows from a similar analysis using the fact that $\Lambda^j (L_i/L_{i-1})^\vee \otimes \E_i \cong \Lambda^{m-1-j} (L_i/L_{i-1})$ since $\E_i = \det(L_i/L_{i-1})$. 

For the calculation of $ [\F_\beta^i] $, we recall that $ \F_\beta^i(\sF) = q_* (i^* \sF \otimes \E_{i+1}^{-\beta_i}) \{i(m-1)\} $.  Assume that $ \beta_i = 1$. Then, since $q$ is a $\p$ fibre bundle, $\E_i = L_i/L_{i-1} \cong \O_q(-1)$ while $(L_{i+1}/L_i)^\vee \cong \Omega^1_q(1)$. Then by Grothendieck-Verdier duality we have:
\begin{eqnarray*}
q_* ((L_i/L_{i-1})^k \otimes \Lambda^l (L_{i+1}/L_i) \otimes \E_{i+1}^\vee) &\cong&
\left(q_* (\Omega_q^l(l) \otimes \O_q(k) \otimes \E_{i+1} \otimes \omega_q [m-1]) \right)^\vee \\
&\cong& \left( q_* (\Omega_q^l(k+l-m+1)) \right)^\vee [-m+1]
\end{eqnarray*}
where we use that $\omega_q \cong \det(L_i/L_{i-1} \otimes (L_{i+1}/L_i)^\vee) \cong \E_i^{m-1} \E_{i+1}^\vee$. In the range $0 \le k,l \le m-1$, Lemma \ref{lem:vanhom} implies that this pushforward is zero unless $k+l-m+1=0$ in which case we get $\left( q_*(\Omega_q^l) \right)^\vee [-m+1] \cong \O[l-m+1] = \O[-k]$. 

Subsequently, $[\sF_\beta^i](W_{i,k} W_{i+1,l})=0$ unless $k+l=m-1$ in which case we get:
$$ [q_* (i^* (W_{i,k} W_{i+1,l}) \otimes \E_{i+1}^\vee) \{i(m-1)\}] \cong q^{(i-1)k+il} [\O] (-1)^k (-q)^{-i(m-1)} \cong (-q)^{-k} (-1)^{i(m-1)} [\O]$$
and the result follows. 

If $\beta_i = m-1$, a similar analysis holds but without the need to apply Grothendieck-Verdier duality. 
\end{proof}

Now assume $ \beta_i = \beta_{i+1}$.  In order to compute $ [\T_\beta^i(2)] $ we first compute $[{\Q_\beta^i}]$.

\begin{Lemma} \label{th:PRonK}
\begin{equation*}
{\Q_\beta^i}(\E_i^a \E_{i+1}^b) = \begin{cases}
\Sym^{b-a-1}(L_{i+1}/L_{i-1}) \det(L_{i+1}/L_{i-1})^a \{1\} \text{ if $a < b$ }\\
0 \text{ if $ a= b$ } \\
\Sym^{a-b-1}(L_{i+1}/L_{i-1}) \det(L_{i+1}/L_{i-1})^b \{1\}[-1] \text{ if $ a > b $} 
\end{cases}
\end{equation*}
\end{Lemma}

\begin{proof}
Assume that $ \beta_i = \beta_{i+1} = 1 $.

We recall that $ \Q_\beta^i $ is given by $ q_* (i^*(\cdot) \otimes \E_{i+1}^\vee) \{1 \} $, where $ i $ denotes the inclusion of $ X_\beta^i $ and $ q $ the projection to $ Y_{d_i(\beta)}$ which is a $\mathbb{P}^1$ fibre bundle. So then 
$$ \Q_\beta^i(\E_i^a \E_{i+1}^b) = q_*(\E_i^a \E_{i+1}^{b-1}) \{1\} = 
\begin{cases}
q_* ((\E_i \E_{i+1})^a \otimes \E_{i+1}^{b-a-1}) \{1\} \text{ if $a \le b$} \\
q_* ((\E_i \E_{i+1})^{a-1} \otimes \E_i \E_{i+1}^{b-a}) \{1\} \text{ if $a > b$}. 
\end{cases}$$

Now, for $k \ge -1$ we have that $ q_* \E_{i+1}^k = \Sym^k L_{i+1}/L_{i-1} $ which proves the Lemma when $a \le b$. By Grothendieck-Verdier duality we also  have $ q_* (\E_i \E_{i+1}^{-k-1}) = (\Sym^k L_{i+1}/L_{i-1})^\vee [-1] $ using the fact that $\omega_q = \E_i \E_{i+1}^\vee$.  Hence, if $a > b$, $q_*(\E_i \E_{i+1}^{b-a}) \cong \Sym^{a-b-1} (L_{i+1}/L_{i-1})^\vee [-1]$ and the result follows since $\Sym^{a-b-1} (L_{i+1}/L_{i-1})^\vee \otimes \det(L_{i+1}/L_{i-1})^{a-b-1} \cong \Sym^{a-b-1}(L_{i+1}/L_{i-1})$. 
\end{proof}

Next we compute $ [\P_\beta^i] $.

\begin{Lemma} \label{th:PonK}
If $ a < b $  then
\begin{align*}
 [\P_\beta^i] ([\Sym^{b-a-1} & (L_{i+1}/L_{i-1}) (\det L_{i+1}/L_{i-1})^a \{ 1\}]) \\
&= q[\E_i^a][\E_{i+1}^b]  + (q - q^{-1}) [\E_i^{a+1}][\E_{i+1}^{b-1}] + \dots + (q - q^{-1}) [\E_i^{b-1}] [\E_{i+1}^{a+1}] - q^{-1} [\E_i^b] [\E_{i+1}^a] 
\end{align*}
while if $a > b$ then 
\begin{align*}
 [\P_\beta^i] ([\Sym^{a-b-1} & (L_{i+1}/L_{i-1}) (\det L_{i+1}/L_{i-1})^b \{ 1\}[-1]]) \\
&= -q[\E_i^b][\E_{i+1}^a]  - (q - q^{-1}) [\E_i^{b+1}][\E_{i+1}^{a-1}] - \dots - (q - q^{-1}) [\E_i^{a-1}] [\E_{i+1}^{b+1}] + q^{-1} [\E_i^a] [\E_{i+1}^b]. 
\end{align*}
\end{Lemma}

\begin{proof}
Note that $ \O_{Y_\beta}(-X_\beta^i) = \E_i^\vee \E_{i+1} \{-2\}$ which gives us $ [\O_{X_\beta^i}] = [\O] - q^2[\E_i^\vee \E_{i+1}]$.  Using that $ \P_\beta^i (\cdot) = i_*( q^*(\cdot) \otimes \E_i) $, we get
\begin{align*}
 [\P_\beta^i] ([\Sym^{b-a-1} & (L_{i+1}/L_{i-1}) (\det L_{i+1}/L_{i-1})^a \{ 1\}]) = [\Sym^{b-a-1} L_{i+1}/L_{i-1} (\det L_{i+1}/L_{i-1})^a \{1\} \otimes \O_{X_\beta^i} \E_i] \\
&= ([\E_{i+1}^{b-a-1}] +  [\E_i \E_{i+1}^{b-a-2}] + \dots + [\E_i^{b-a-1}])[(\E_i \E_{i+1})^a] (q[\E_{i+1}] - q^{-1} [\E_i ])
\end{align*}
and the desired first equality follows. The second equality follows similarly. 
\end{proof}

Now we are in a position to compute $ [\T_\beta^i(2)] $ on our chosen basis.  
\begin{Proposition} \label{th:Tmap}
Let $ a = \delta(i), b = \delta(i+1) $.

\begin{equation*}
[\T_\beta^i(2)] (\prod_j W_{j, \delta(j)})= q^{-m} \begin{cases}
\begin{aligned} \prod_{j \ne i,i+1} W_{j, \delta(j)} &( 
(q^{-1} - q) \sum_{j=1}^{b-a-1} q^{b-a-j} W_{i,b-j} W_{i+1,a+j} \\
&+ q^{b-a} q^{-1} W_{i,b} W_{i+1, a}) 
\end{aligned} \ \text{ if $ a < b $ } \\
q \prod_j W_{j, \delta(j)} \quad \text{ if $ a = b $ } \\
\begin{aligned}
-\prod_{j \ne i,i+1} W_{j, \delta(j)} ( (q^{-1}-q) \sum_{j=0}^{a-b-1} q^{-j} W_{i,a-j} W_{i+1, b+j} \\
 - q^{b-a}W_{i,b} W_{{i+1},a} )
\end{aligned}\ \text{ if $ a > b $ }
\end{cases}
\end{equation*}

\end{Proposition}

\begin{proof}
By Theorem \ref{th:kerneltwist}, we have that $ \sT_\beta^i(2) \cong \Cone(\sP_\beta^i * \sQ_\beta^i[-1]\{1\} \rightarrow \O_\Delta)[-m+1]\{m-1\} $ from which it follows that $ [\T_\beta^i(2)] = q^{-m}(q - [\P_\beta^i][\Q_\beta^i]) $. The result follows by patiently simplifying using Lemmas \ref{th:PRonK} and \ref{th:PonK}.
\end{proof}

\subsection{Comparison with representations}
Now, we are ready to prove Conjecture \ref{th:Kconj} in the case that $ q= 1 $.  Let $ K(Y_\beta)^\cl $ and $ V_\beta^\cl $ denote the specializations of the representation to $ q = 1 $.  On the geometric side corresponds to working non-equivariant K-theory, whereas on the representation theory side it corresponds to working with representations of $ \sl $ instead of $ \uq $.

Let $ v_0, \dots, v_{m-1} $ denote the standard basis for $ \C^m = V_{\omega_1}^\cl $ and also the standard basis for $ \Lambda^{m-1}(\C^m) = V_{\omega_{m-1}}^\cl $. Define a map $ \alpha : K(Y_\beta)^\cl \rightarrow V_\beta^\cl $ by 
$$ \prod W_{i, \delta(i)} \mapsto v_{\delta(1)} \otimes v_{\delta(2)} \otimes \cdots v_{\delta(n)} \in V_{\omega_{\beta_1}}^\cl \otimes \cdots \otimes V_{\omega_{\beta_n}}^\cl $$

\begin{Theorem} \label{th:Kconjq1}
$ \alpha $ is an isomorphism of vector spaces and for all $(\beta, \beta') $ tangles $ T $, the diagram
\begin{equation}
\begin{CD}
K(Y_\beta)^\cl @>>> V_\beta^\cl \\
@V[\Psi(T)]VV @VV\psi(T)V \\
K(Y_{\beta'})^\cl @>>> V_{\beta'}^\cl
\end{CD}
\end{equation}
commutes.
\end{Theorem}

\begin{proof}
It is enough to check this for the basic tangles $ t_\beta^i(2)$, with $ \beta_i = \beta_{i+1} = 1 $ and any $g_\beta^i, f_\beta^i $ (these suffices because both sides obey the pitchfork moves).  For these basic tangles, the result follows from Propositions \ref{th:GFmap} and \ref{th:Tmap}.
\end{proof}

\begin{Remark} \label{re:Kconj}
This seemingly natural isomophism $\alpha$ does not work in the $ q \ne 1 $ case.  It is ok for the basic tangles $ g_\beta^i, f_\beta^i $, but not for $p_\beta^i$, $q_\beta^i$ and consequently $ t_\beta^i(2) $.  Proving Conjecture \ref{th:Kconj} requires redefining $\alpha$. This amounts to finding a basis for $V_\beta$ on which the Reshetikhin-Turaev maps for $ g_\beta^i, f_\beta^i, t_\beta^i(2) $ act as in Propositions \ref{th:GFmap} and \ref{th:Tmap} (under the identification of $\{ \prod_i W_{i,\delta(i)} \}$ with this basis).  

Let us give one further piece of evidence that such a basis should exist. Consider for simplicity the case that $ \beta = (1, \dots, 1) $. Then the $ [T_\beta^i(2)] $ generate an action of the braid group $B_n$ on $ K(Y_\beta) $.  By Proposition \ref{th:Tmap}, these generators have eigenvalues only $-1$ and $ q $.  Hence the action of $ \C[q, q^{-1}][B_n]$ factors through an action of the Hecke algebra $ H_n$.  For generic $ q $ including $ q= 1 $, the Hecke algebra is semisimple and hence for generic $ q $ this Hecke algebra representation is determined by its $ q=1 $ specialization.  This gives us an isomorphism $ K(Y_\beta) \rightarrow V_\beta $, defined for generic $ q $ which intertwines the action of $ [T_\beta^i(2)] $ and $ \psi(t_\beta^i(2)) $.  Unfortunately, this only holds for generic $ q $ and it is not clear how to generalize this observation for general $ \beta $.
\end{Remark}

\end{document}